\documentclass{amsart}

\usepackage{amssymb,amsmath,stmaryrd,mathrsfs}
\def\definetac{\newif\iftac}    
\ifx\tactrue\undefined
  \definetac
  \ifx\state\undefined\tacfalse\else\tactrue\fi
\fi
\iftac\else\usepackage{amsthm}\fi
\usepackage[all,2cell]{xy}
\UseAllTwocells
\usepackage{enumitem}
\usepackage{xcolor}
\definecolor{darkgreen}{rgb}{0,0.45,0} 
\usepackage[pagebackref,colorlinks,citecolor=darkgreen,linkcolor=darkgreen]{hyperref}
\usepackage{mathtools}          
\usepackage{tikz}
\usepackage{braket}             

\usepackage{url}                
\usepackage{xspace}             

\makeatletter
\let\ea\expandafter

\def\mdef#1#2{\ea\ea\ea\gdef\ea\ea\noexpand#1\ea{\ea\ensuremath\ea{#2}\xspace}}
\def\alwaysmath#1{\ea\ea\ea\global\ea\ea\ea\let\ea\ea\csname your@#1\endcsname\csname #1\endcsname
  \ea\def\csname #1\endcsname{\ensuremath{\csname your@#1\endcsname}\xspace}}

\DeclareRobustCommand\widecheck[1]{{\mathpalette\@widecheck{#1}}}
\def\@widecheck#1#2{%
    \setbox\z@\hbox{\m@th$#1#2$}%
    \setbox\tw@\hbox{\m@th$#1%
       \widehat{%
          \vrule\@width\z@\@height\ht\z@
          \vrule\@height\z@\@width\wd\z@}$}%
    \dp\tw@-\ht\z@
    \@tempdima\ht\z@ \advance\@tempdima2\ht\tw@ \divide\@tempdima\thr@@
    \setbox\tw@\hbox{%
       \raise\@tempdima\hbox{\scalebox{1}[-1]{\lower\@tempdima\box
\tw@}}}%
    {\ooalign{\box\tw@ \cr \box\z@}}}

\newcount\foreachcount

\def\foreachletter#1#2#3{\foreachcount=#1
  \ea\loop\ea\ea\ea#3\@alph\foreachcount
  \advance\foreachcount by 1
  \ifnum\foreachcount<#2\repeat}

\def\foreachLetter#1#2#3{\foreachcount=#1
  \ea\loop\ea\ea\ea#3\@Alph\foreachcount
  \advance\foreachcount by 1
  \ifnum\foreachcount<#2\repeat}

\def\definescr#1{\ea\gdef\csname s#1\endcsname{\ensuremath{\mathscr{#1}}\xspace}}
\foreachLetter{1}{27}{\definescr}
\def\definecal#1{\ea\gdef\csname c#1\endcsname{\ensuremath{\mathcal{#1}}\xspace}}
\foreachLetter{1}{27}{\definecal}
\def\definebold#1{\ea\gdef\csname b#1\endcsname{\ensuremath{\mathbf{#1}}\xspace}}
\foreachLetter{1}{27}{\definebold}
\def\definebb#1{\ea\gdef\csname l#1\endcsname{\ensuremath{\mathbb{#1}}\xspace}}
\foreachLetter{1}{27}{\definebb}
\def\definefrak#1{\ea\gdef\csname f#1\endcsname{\ensuremath{\mathfrak{#1}}\xspace}}
\foreachletter{1}{9}{\definefrak} 
\foreachletter{10}{27}{\definefrak}
\def\definebar#1{\ea\gdef\csname #1bar\endcsname{\ensuremath{\overline{#1}}\xspace}}
\foreachLetter{1}{27}{\definebar}
\foreachletter{1}{8}{\definebar} 
\foreachletter{9}{15}{\definebar} 
\foreachletter{16}{27}{\definebar}
\def\definetil#1{\ea\gdef\csname #1til\endcsname{\ensuremath{\widetilde{#1}}\xspace}}
\foreachLetter{1}{27}{\definetil}
\foreachletter{1}{27}{\definetil}
\def\definehat#1{\ea\gdef\csname #1hat\endcsname{\ensuremath{\widehat{#1}}\xspace}}
\foreachLetter{1}{27}{\definehat}
\foreachletter{1}{27}{\definehat}
\def\definechk#1{\ea\gdef\csname #1chk\endcsname{\ensuremath{\widecheck{#1}}\xspace}}
\foreachLetter{1}{27}{\definechk}
\foreachletter{1}{27}{\definechk}
\def\defineul#1{\ea\gdef\csname u#1\endcsname{\ensuremath{\underline{#1}}\xspace}}
\foreachLetter{1}{27}{\defineul}
\foreachletter{1}{27}{\defineul}

\def\autofmt@n#1\autofmt@end{\mathrm{#1}}
\def\autofmt@b#1\autofmt@end{\mathbf{#1}}
\def\autofmt@l#1#2\autofmt@end{\mathbb{#1}\mathsf{#2}}
\def\autofmt@c#1#2\autofmt@end{\mathcal{#1}\mathit{#2}}
\def\autofmt@s#1#2\autofmt@end{\mathscr{#1}\mathit{#2}}
\def\autofmt@f#1\autofmt@end{\mathsf{#1}}
\def\autofmt@u#1\autofmt@end{\underline{\smash{\mathsf{#1}}}}
\def\autofmt@U#1\autofmt@end{\underline{\underline{\smash{\mathsf{#1}}}}}
\def\autofmt@h#1\autofmt@end{\widehat{#1}}
\def\autofmt@r#1\autofmt@end{\overline{#1}}
\def\autofmt@t#1\autofmt@end{\widetilde{#1}}
\def\autofmt@k#1\autofmt@end{\check{#1}}

\def\auto@drop#1{}
\def\autodef#1{\ea\ea\ea\@autodef\ea\ea\ea#1\ea\auto@drop\string#1\autodef@end}
\def\@autodef#1#2#3\autodef@end{%
  \ea\def\ea#1\ea{\ea\ensuremath\ea{\csname autofmt@#2\endcsname#3\autofmt@end}\xspace}}
\def\autodefs@end{blarg!}
\def\autodefs#1{\@autodefs#1\autodefs@end}
\def\@autodefs#1{\ifx#1\autodefs@end%
  \def\autodefs@next{}%
  \else%
  \def\autodefs@next{\autodef#1\@autodefs}%
  \fi\autodefs@next}

\DeclareSymbolFont{bbold}{U}{bbold}{m}{n}
\DeclareSymbolFontAlphabet{\mathbbb}{bbold}

\newcommand{\bbone}{\ensuremath{\mathbbb{1}}\xspace}

\mdef\delbar{\overline{\partial}}

\newcommand{\inv}{^{-1}}

\mdef\hf{\textstyle\frac12 }
\mdef\thrd{\textstyle\frac13 }
\mdef\qtr{\textstyle\frac14 }

\newcommand{\op}{^{\mathrm{op}}}

\SelectTips{cm}{}
\newdir{ >}{{}*!/-10pt/@{>}}    

\mdef\Id{\mathrm{Id}}
\mdef\id{\mathrm{id}}
\alwaysmath{ell}
\alwaysmath{infty}
\alwaysmath{odot}
\def\frc#1/#2.{\frac{#1}{#2}}   
\mdef\ten{\mathrel{\otimes}}

\mdef\sqten{\mathrel{\boxtimes}}

\DeclareRobustCommand\widecheck[1]{{\mathpalette\@widecheck{#1}}}
\def\@widecheck#1#2{%
    \setbox\z@\hbox{\m@th$#1#2$}%
    \setbox\tw@\hbox{\m@th$#1%
       \widehat{%
          \vrule\@width\z@\@height\ht\z@
          \vrule\@height\z@\@width\wd\z@}$}%
    \dp\tw@-\ht\z@
    \@tempdima\ht\z@ \advance\@tempdima2\ht\tw@ \divide\@tempdima\thr@@
    \setbox\tw@\hbox{%
       \raise\@tempdima\hbox{\scalebox{1}[-1]{\lower\@tempdima\box
\tw@}}}%
    {\ooalign{\box\tw@ \cr \box\z@}}}

\DeclareMathOperator\colim{colim}

\DeclareMathOperator\Hom{Hom}
\DeclareMathOperator\Ext{Ext}

\DeclareMathOperator\ind{ind}
\DeclareMathOperator\nhom{hom}
\DeclareMathOperator\nend{end}
\DeclareMathOperator\Rhom{\mathbf{R}hom}
\newcommand\Lotimes{\otimes^\mathbf{L}}

\newcommand{\ot}{\ensuremath{\leftarrow}}

\let\into\hookrightarrow

\mdef\we{\overset{\sim}{\longrightarrow}}
\mdef\leftwe{\overset{\sim}{\longleftarrow}}

\let\xto\xrightarrow

\def\rightarrowtailfill@{\arrowfill@{\Yright\joinrel\relbar}\relbar\rightarrow}
\newcommand\xrightarrowtail[2][]{\ext@arrow 0055{\rightarrowtailfill@}{#1}{#2}}

\def\twoheadrightarrowfill@{\arrowfill@{\relbar\joinrel\relbar}\relbar\twoheadrightarrow}
\newcommand\xtwoheadrightarrow[2][]{\ext@arrow 0055{\twoheadrightarrowfill@}{#1}{#2}}

\def\slashedarrowfill@#1#2#3#4#5{%
  $\m@th\thickmuskip0mu\medmuskip\thickmuskip\thinmuskip\thickmuskip
   \relax#5#1\mkern-7mu%
   \cleaders\hbox{$#5\mkern-2mu#2\mkern-2mu$}\hfill
   \mathclap{#3}\mathclap{#2}%
   \cleaders\hbox{$#5\mkern-2mu#2\mkern-2mu$}\hfill
   \mkern-7mu#4$%
}
\def\rightslashedarrowfill@{%
  \slashedarrowfill@\relbar\relbar\mapstochar\rightarrow}
\newcommand\xslashedrightarrow[2][]{%
  \ext@arrow 0055{\rightslashedarrowfill@}{#1}{#2}}
\mdef\hto{\xslashedrightarrow{}}
\mdef\htoo{\xslashedrightarrow{\quad}}

\def\toiso{\xto{\smash{\raisebox{-.5mm}{$\scriptstyle\sim$}}}}

\long\def\my@drawfill#1#2;{%
\@skipfalse
\fill[#1,draw=none] #2;
\@skiptrue
\draw[#1,fill=none] #2;
}
\newif\if@skip
\newcommand{\skipit}[1]{\if@skip\else#1\fi}
\newcommand{\drawfill}[1][]{\my@drawfill{#1}}

\newif\ifhyperref
\@ifpackageloaded{hyperref}{\hyperreftrue}{\hyperreffalse}
\iftac
  \let\your@state\state
  \def\state#1{\gdef\currthmtype{#1}\your@state{#1}}
  \let\your@staterm\staterm
  \def\staterm#1{\gdef\currthmtype{#1}\your@staterm{#1}}
  \let\defthm\newtheorem
  \def\currthmtype{}
  \ifhyperref
    \def\autoref#1{\ref*{label@name@#1}~\ref{#1}}
  \else
    \def\autoref#1{\ref{label@name@#1}~\ref{#1}}
  \fi
  \AtBeginDocument{%
    \let\old@label\label%
    \def\label#1{%
      {\let\your@currentlabel\@currentlabel%
        \edef\@currentlabel{\currthmtype}%
        \old@label{label@name@#1}}%
      \old@label{#1}}
  }
\else
  \ifhyperref
    \def\defthm#1#2{%
      \newtheorem{#1}{#2}[section]%
      \expandafter\def\csname #1autorefname\endcsname{#2}%
      \expandafter\let\csname c@#1\endcsname\c@thm}
  \else
    \def\defthm#1#2{\newtheorem{#1}[thm]{#2}}
    \ifx\SK@label\undefined\let\SK@label\label\fi
    \let\old@label\label
    \let\your@thm\@thm
    \def\@thm#1#2#3{\gdef\currthmtype{#3}\your@thm{#1}{#2}{#3}}
    \def\currthmtype{}
    \def\label#1{{\let\your@currentlabel\@currentlabel\def\@currentlabel%
        {\currthmtype~\your@currentlabel}%
        \SK@label{#1@}}\old@label{#1}}
    \def\autoref#1{\ref{#1@}}
  \fi
\fi

\newtheorem{thm}{Theorem}[section]

\defthm{cor}{Corollary}
\defthm{prop}{Proposition}
\defthm{lem}{Lemma}
\defthm{sch}{Scholium}
\defthm{assume}{Assumption}
\defthm{claim}{Claim}
\defthm{conj}{Conjecture}
\defthm{hyp}{Hypothesis}
\defthm{fact}{Fact}
\iftac\theoremstyle{plain}\else\theoremstyle{definition}\fi
\defthm{defn}{Definition}
\defthm{notn}{Notation}
\defthm{con}{Construction}
\iftac\theoremstyle{plain}\else\theoremstyle{remark}\fi
\defthm{rmk}{Remark}
\defthm{eg}{Example}
\defthm{egs}{Examples}
\defthm{ex}{Exercise}
\defthm{ceg}{Counterexample}
\defthm{warn}{Warning}

\def\thmqedhere{\expandafter\csname\csname @currenvir\endcsname @qed\endcsname}

\setitemize[1]{leftmargin=2em}
\setenumerate[1]{leftmargin=*}

\iftac
  \let\c@equation\c@subsection
\else
  \let\c@equation\c@thm
\fi
\numberwithin{equation}{section}

\@ifpackageloaded{mathtools}{\mathtoolsset{showonlyrefs,showmanualtags}}{}

\alwaysmath{alpha}
\alwaysmath{beta}
\alwaysmath{gamma}
\alwaysmath{Gamma}
\alwaysmath{delta}
\alwaysmath{Delta}
\alwaysmath{epsilon}
\mdef\ep{\varepsilon}
\alwaysmath{zeta}
\alwaysmath{eta}
\alwaysmath{theta}
\alwaysmath{Theta}
\alwaysmath{iota}
\alwaysmath{kappa}
\alwaysmath{lambda}
\alwaysmath{Lambda}
\alwaysmath{mu}
\alwaysmath{nu}
\alwaysmath{xi}
\alwaysmath{pi}
\alwaysmath{rho}
\alwaysmath{sigma}
\alwaysmath{Sigma}
\alwaysmath{tau}
\alwaysmath{upsilon}
\alwaysmath{Upsilon}
\alwaysmath{phi}
\alwaysmath{Pi}
\alwaysmath{Phi}
\mdef\ph{\varphi}
\alwaysmath{chi}
\alwaysmath{psi}
\alwaysmath{Psi}
\alwaysmath{omega}
\alwaysmath{Omega}

\makeatother

\usepackage{pdflscape}

\tikzset{lab/.style={auto,font=\scriptsize}} 
\usetikzlibrary{arrows}

\newcommand{\D}{\sD}
\newcommand{\E}{\sE}

\newcommand{\sse}{\stackrel{\mathrm{s}}{\sim}}

\autodefs{\cDER\cCat\cGrp\cCAT\cMONCAT\bSet\cProf\bCat
\cPro\cMONDER\cBICAT\ncomp\nid\ncoker\niso\nIm\sProf\ncyl\fC\fZ\fT\nhocolim\nholim\cPsNat\ncoll\bsSet\cAlg\cIm\cI\cP}

\def\cPDER{\ensuremath{\mathcal{PD}\mathit{ER}}\xspace}

\def\ho{\mathscr{H}\!\mathit{o}\xspace}

\let\oldboxtimes\boxtimes
\def\boxtimes{\mathrel{\oldboxtimes}}

\newcommand{\fib}{\mathsf{fib}}
\newcommand{\cof}{\mathsf{cof}}

\def\ccsub{_{\mathrm{cc}}}
\def\pdh(#1,#2){\llbracket #1,#2\rrbracket}
\def\ldh(#1,#2){\llbracket #1,#2\rrbracket\ccsub}
\def\pend(#1){\pdh(#1,#1)}
\def\lend(#1){\ldh(#1,#1)}

\def\shift#1#2{{#1}^{#2}}

\def\DTl#1#2#3#4#5#6#7{%
  \xymatrix@C=3pc{{#1} \ar[r]^-{#2} &
    {#3} \ar[r]^-{#4} &
    {#5} \ar[r]^-{#6} &
    {#7}
  }}

\newsavebox{\tvabox}
\savebox\tvabox{\hspace{1mm}\begin{tikzpicture}[>=latex',baseline={(0,-.18)}]
  \draw[->] (0,.1) -- +(1,0);
  \node at (.5,0) {$\scriptscriptstyle\bot$};
  \draw[->] (1,-.1) -- +(-1,0);
  \draw[->] (1,-.2) -- +(-1,0);
\end{tikzpicture}\hspace{1mm}}

\renewcommand{\ex}{\mathrm{ex}}
\newcommand{\Ch}{\mathrm{Ch}}
\newcommand{\Mod}{\mathrm{Mod}}

\title{Tilting theory via stable homotopy theory}
\author{Moritz Groth}
\address{MPIM, Vivatsgasse 7, 53111 Bonn, Germany}
\email{mgroth@mpim-bonn.mpg.de}

\author{Jan \v{S}\v{t}ov\'{\i}\v{c}ek}
\address{Department of Algebra, Charles University in Prague, Sokolovska 83, 186 75 Praha~8, Czech Republic}
\email{stovicek@karlin.mff.cuni.cz}


\date{\today}

\thanks{The first named author was supported by the Dutch Science Foundation (NWO). The second named author was supported by grant GA\v{C}R P201/12/G028 from the Czech Science Foundation.}

\begin{document}

\begin{abstract}
We show that certain tilting results for quivers are formal consequences of stability, and as such are part of a formal calculus available in any abstract stable homotopy theory. Thus these results are for example valid over arbitrary ground rings, for quasi-coherent modules on schemes, in the differential-graded context, in stable homotopy theory and also in the equivariant, motivic or parametrized variant thereof. In further work, we will continue developing this calculus and obtain additional abstract tilting results. Here, we also deduce an additional characterization of stability, based on Goodwillie's strongly (co)cartesian $n$-cubes.

As applications we construct abstract Auslander--Reiten translations and abstract Serre functors for the trivalent source and verify the relative fractionally Calabi--Yau property. This is used to offer a new perspective on May's axioms for monoidal, triangulated categories.
\end{abstract}

\maketitle

\tableofcontents

\section{Introduction}
\label{sec:intro}

Morita \cite{morita:duality} gave a complete answer to the following question. Given two rings $R$ and $S$, under which conditions are the categories of modules $\Mod(R)$ and $\Mod(S)$ equivalent? Morita showed that this is the case if and only if there is a bimodule $P\in\Mod(R\otimes S\op)$ with particularly nice properties. It is due to these insights that the study of equivalences of categories of modules is referred to as \emph{Morita theory}.

\emph{Tilting theory} \cite{tilting} is a derived version of Morita theory: the aim is to find conditions guaranteeing that (certain flavors of) derived categories of rings are equivalent as triangulated categories. The key role played by certain bimodules in Morita theory is now taken by the so-called \emph{tilting complexes}. Happel \cite{happel:fdalgebra,happel:triangulated} first observed the special case that for finite-dimensional algebras over a field tilting modules induce equivalences of derived categories. Later on, Rickard in \cite{rickard:morita} established a general `derived Morita theorem' for rings, and Keller~\cite{keller:dg} gave a conceptual explanation and a vast generalization of Rickard's results in the language of differential-graded categories. For further generalizations see also \cite{duggershipley:K,schwede:Morita} and the many references therein. A different generalization to the context of abelian categories was established by Ladkani in~\cite{ladkani:universal-der-eq,ladkani:thesis}.

An important class of algebras over a field is provided by path-algebras associated to quivers. Let us recall that a quiver $Q$ is simply an oriented graph, and that a module over the path-algebra $kQ$ is equivalently specified by a functor $Q\to \Mod(k)$. Tilting theory for quivers thus amounts to searching for conditions on quivers $Q,Q'$ which imply that there are exact equivalences of derived categories
\[
D(kQ)\stackrel{\Delta}{\simeq} D(kQ').
\]
Obviously, these conditions might depend on the field under consideration.

In this paper and its sequels we take a different perspective on tilting theory and show that some aspects of it are \emph{formal consequences of stability}. Let us explain what we mean by this. There are various approaches to axiomatic homotopy theory like model categories \cite{quillen:ha,hovey:modelcats}, $\infty$-categories \cite{joyal:barca,lurie:HTT}, and derivators \cite{heller:htpy,grothendieck:derivateurs}. In each of these cases there is a subclass which models \emph{stable} homotopy theories: the defining property is that there is some kind of a zero object and that homotopy pushouts and homotopy pullbacks coincide. Typical examples of stable homotopy theories are given by unbounded chain complexes, spectra in the sense of topology and the variants of equivariant stable, motivic stable, and parametrized stable homotopy theory.

Here we decide to work in the framework of \emph{derivators} (see \cite{grothendieck:derivateurs} for a comprehensive bibliography) which is particularly well-adapted for our purposes. Recall e.g.~from \cite{groth:ptstab} that a derivator is in a certain sense a minimal extension of a classical homotopy category or derived category to a framework which allows for a well-behaved calculus of homotopy (co)limits. More specifically, a derivator consists of derived categories of all diagram categories and restriction functors between them, and homotopy (co)limit functors are simply adjoints to these restriction functors. \emph{Stable} derivators enhance triangulated categories (see \cite{groth:ptstab}), and in that case all categories in sight can be endowed with canonical triangulations. For example, associated to a field $k$ there is the stable derivator $\D_k$ of unbounded chain complexes over~$k$, enhancing the classical derived category $D(k)$, and similarly $\D_R$ for an arbitrary ground ring~$R$.

A basic operation on derivators is given by shifting: associated to a derivator~\D and a small category $B$ we obtain the shifted derivator $\D^B$ (see \autoref{egs:prederivators} and \autoref{egs:der}), which is stable if~\D is. This operation amounts to considering \emph{coherent} diagrams of shape $B$ in the homotopy theory \D. Since a quiver~$Q$ is simply a graph and hence freely generates a category, this shifting operation applies to quivers. In the special case of the derivator $\D_k$ of a field $k$, it is immediate that the shifted derivator $\D_k^Q$ is equivalent to the stable derivator $\D_{kQ}$ associated to the path-algebra, and similarly for other ground rings.

Now, instead of asking two quivers $Q$ and $Q'$ to be derived equivalent, one could ask them to be \emph{strongly stably equivalent} in the sense that \emph{for all stable derivators} \D there are pseudo-natural equivalences of (stable) derivators
\[
\xymatrix{\D^Q\simeq\D^{Q'}.}
\]
(See \S\ref{sec:strong} for this definition and a brief discussion of examples of stable derivators.)
Such equivalences induce exact equivalences of the underlying triangulated categories, but, in general, having an equivalence of derivators is a stronger statement. Thus, being strongly stably equivalent is a much more restrictive assumption to be imposed on a pair of quivers. \emph{Nevertheless the aim of this paper and its sequels is to justify the claim that quite some classical examples of derived equivalent quivers/categories are in fact pairs of strongly stably equivalent quivers/categories.} 

We emphasize the added generality one obtains by showing that two derived equivalent quivers are actually strongly stably equivalent. Since the results are valid for every stable derivator they apply to the derivator~$\D_R$ associated to an arbitrary ground ring~$R$ (in particular, we obtain exact equivalences between the derived categories $D(RQ)$ and $D(RQ')$). The same result is also true for the stable derivator $\D_X$ of chain complexes of quasi-coherent $\mathcal{O}_X$-modules on a scheme~$X$. Moreover, the corresponding statement is true for the stable derivator $\D_A$ of differential-graded modules over a differential-graded algebra~$A$, for the stable derivator $\D_E$ of $E$-module spectra over a ring spectrum~$E$, and also for other stable derivators arising for example in equivariant stable, motivic stable or parametrized stable homotopy theory. Furthermore, these equivalences are compatible with various (co)induction and restriction of scalar functors, (Bousfield) localizations and colocalizations, and other exact morphisms (see again~\S\ref{sec:strong}).

In this paper we establish first examples of strongly stably equivalent quivers which can be obtained rather directly (as opposed to the abstract tilting results we obtain in \cite{gst:tree,gst:Dynkin-A,gst:acyclic} by more refined techniques). We begin by considering the Dynkin quivers of type $A$, and show that a reorientation of such a quiver yields a strongly stably equivalent quiver (\autoref{thm:An}). The proof is essentially an adaptation of the Waldhausen $S_\bullet$-construction of Algebraic K-Theory \cite{waldhausen:k-theory}. This construction can be recycled to also give a similar result for trees with at most one branching point. More precisely, an arbitrary reorientation of the edges of such a quiver gives rise to a strongly stably equivalent quiver (\autoref{thm:tiltone}). This includes, in particular, the Dynkin quivers of type $ADE$ as well as certain Euclidean graphs, thus generalizing results of Happel~\cite{happel:dynkin} and Ladkani~\cite{ladkani:universal-der-eq}.

As a special case of the latter result, we deduce that the source of valence $n$ and the sink of valence $n$ are strongly stably equivalent. This fact is also an immediate consequence of our discussion of (strongly) cartesian $n$-cubes in \S\ref{sec:cubes} (see \autoref{cor:tiltsourcesink}). In that section we generalize some basic results of Goodwillie \cite{goodwillie:II} to the framework of derivators. In particular, we show that an $n$-cube is strongly cartesian if and only if each subcube is cartesian, and that a coherent morphism of cartesian $n$-cubes is a cartesian $(n+1)$-cube. As an application we obtain an additional characterization of stability; a pointed derivator is stable if and only if the classes of strongly cartesian $n$-cubes and strongly cocartesian $n$-cubes coincide for all $n\geq 2$ (\autoref{cor:stable}).

The final aim of this paper is to begin a development of what we think of as \emph{abstract representation theory} for the commutative square, the trivalent source, and the trivalent sink. Note that, contrary to the other examples, the commutative square is not a free category (it is a quiver with a commutativity relation). We show these categories to be strongly stably equivalent. More interestingly, we reveal certain symmetries on related stable derivators which turn out to be rather useful. For example, these symmetries allow us to construct abstract versions of Serre functors and Auslander--Reiten translations for these quivers. The motivation for this terminology stems from the observation, that if we consider the trivalent source and the derivator of a commutative ground ring then these functors reduce to the classical ones. 

This certainly deserves a more systematic treatment, but here we just develop what is needed for our applications. We show that the trivalent source is relative fractionally Calabi--Yau of dimension $\frac 23$ (\autoref{thm:fracCYring}) and we obtain abstract versions of the Auslander--Reiten quivers in this case (\autoref{thm:ar-quiver-d4}). As a final application, we use our results to shed some additional light on May's axioms \cite{may:traces} for monoidal, triangulated categories which he used to prove additivity results for abstractly defined Euler characteristics. In \cite{kellerneeman:may}, these axioms were studied from a representation-theoretic perspective. May's axioms were also reconsidered in \cite{gps:additivity} where they are shown to be valid in any closed monoidal, stable derivator. Our contribution here is to relate the two approaches discussed in \cite{kellerneeman:may} and \cite{gps:additivity}. Inspired by \cite{kellerneeman:may} we show that some of May's axioms follow from statements that are a formal consequence of stability alone and do not depend on monoidal structure. Moreover, our representation theoretic perspective yields coherent and extended versions of May's braid axioms (see \S\ref{sec:square} and \S\ref{sec:May}).

In a way this paper is a continuation of the papers \cite{groth:ptstab} and \cite{gps:mayer}, and can be considered as a \emph{formal study of stability} (in a similar sense as the paper \cite{gps:additivity} can be thought of as a \emph{formal study of the interaction of stability and monoidal structure}). We expect this calculus which is valid in an arbitrary stable homotopy theory to be rich. In the sequels \cite{gst:tree,gst:acyclic} we continue this program and establish abstract reflection functors in the sense of \cite{gabriel:unzerlegbar,BGP:reflection,happel:fdalgebra} for arbitrary trees, acyclic quivers, and more general shapes, and we plan to come back to further aspects of this calculus somewhere else (see for example \cite{gst:Dynkin-A}).

The content of the sections is as follows. In \S\ref{sec:review} we recall without proof basic facts about derivators and establish some notation. In \S\ref{sec:sub} we briefly study sub(pre)derivators of shifted derivators defined by exactness properties and equivalences between them induced by Kan extensions. In \S\ref{sec:stable} we recall some basics about stable derivators and exact morphisms. In \S\ref{sec:strong} we define strongly stably equivalent quivers and categories, and collect a few examples of stable derivators to which the results of this paper and the sequels apply. In \S\ref{sec:An} we establish an abstract tilting result for Dynkin quivers of type~$A$, based on the $S_\bullet$-construction of Algebraic K-Theory. In~\S\ref{sec:onebranch} we extend this to a tilting result for trees with a unique branching point, including the Dynkin quivers of type $ADE$. In \S\ref{sec:cubes} we establish some properties of (strongly) cartesian $n$-cubes, and obtain an additional characterization of stability. In \S\ref{sec:square} we perform a few first steps in abstract representation theory of the commutative square and the source and the sink of valence three. In \S\ref{sec:May} we apply our results to shed some additional light on May's axioms.

\section{Review of derivators}
\label{sec:review}

In this section we recall the definition of a derivator and establish some basic notation and terminology. For more details and motivational remarks we refer to \cite{groth:ptstab}, its sequel \cite{gps:mayer}, and the many references therein. Let \cCat be the 2-category of small categories, and \cCAT the 2-category of not necessarily small categories. The terminal category with one object and its identity morphism only will be denoted by $\bbone$. There is an isomorphism of categories $A\cong A^\bbone$ and we write $a\colon\bbone\to A$ for the functor corresponding to $a\in A$. 

A \textbf{prederivator} is simply a 2-functor $\D\colon\cCat\op\to\cCAT.$ Morphisms of prederivators are pseudo-natural transformations and transformations of prederivators are modifications so that we obtain a 2-category $\cPDER$ of prederivators (see \cite{borceux1}). Given a functor $u\colon A\to B$, we write $u^*\colon\D(B) \to \D(A)$ for its image under the prederivator \D, and refer to it as \textbf{restriction functor} along $u$. In particular, associated to $a\colon\bbone\to A$ we obtain an \textbf{evaluation functor} $a^\ast\colon\D(A)\to\D(\bbone).$ Given a morphism $f\colon X\to Y$ in $\D(A)$, we denote its image under $a^\ast$ by $f_a\colon X_a\to Y_a.$ This is a morphism in the \textbf{underlying category} $\D(\bbone)$. 

The objects of $\D(A)$ are referred to as \textbf{(coherent) diagrams (of shape~$A$)} in \D, and this terminology is motivated by \autoref{egs:prederivators} and \autoref{egs:derivators}. The evaluation functors can be used to see that any coherent diagram $X\in\D(A)$ has an \textbf{underlying (incoherent) diagram} $A\to\D(\bbone)$. It is important to note that, in general, a coherent diagram is not determined by its underlying diagram, even not up to isomorphism. Put differently, the resulting functor $\D(A)\to\D(\bbone)^A$ is far from being an equivalence. Nevertheless, we say that a diagram has the form of its underlying diagram, which is then drawn in the usual way. 

\begin{egs}\label{egs:prederivators}
~ 
\begin{enumerate}
\item Any category \bC induces a \textbf{represented prederivator}~$y(\bC)$ defined by $y(\bC)(A) \coloneqq \bC^A.$ Its underlying category is equivalent to \bC itself. 
\item Suppose \bC is a \emph{Quillen model category} (see e.g.~\cite{quillen:ha,hovey:modelcats}), with class \bW of \emph{weak equivalences}. Then we obtain a prederivator $\ho(\bC)$ by formally inverting the pointwise weak equivalences $\ho(\bC)(A) \coloneqq (\bC^A)[(\bW^A)^{-1}]$. The underlying category of $\ho(\bC)$ is the homotopy category $\bC[\bW^{-1}]$ of \bC. 
\item Similarly, any \emph{$\infty$-category} $\cC$ (see e.g.~\cite{lurie:HTT,joyal:barca} or also \cite{groth:scinfinity}) has an underlying homotopy prederivator $\ho(\cC)(A) \coloneqq \mathrm{Ho}(\cC^{N(A)}).$ Here, $\mathrm{Ho}(-)$ denotes the usual homotopy category of an $\infty$-category and $N$ the nerve construction.
\item Given a prederivator \D and $B\in\cCat$ we can define the \textbf{shifted prederivator} $\shift\D B$ by $\shift\D B(A) \coloneqq \D(B\times A)$. Moreover, the \textbf{opposite prederivator} $\D\op$ is defined by $\D\op(A) \coloneqq \D(A\op)\op$. These two constructions are compatible with the previous three examples in that we have
\[
\shift{y(\bC)}{B} \cong y(\bC^B),\quad \shift{\ho(\bC)}{B} \cong \ho(\bC^B),\quad \shift{\ho(\cC)}{B} \cong \ho(\cC^B),
\]
and similarly
\[
y(\bC)\op \cong y(\bC\op),\quad \ho(\bC)\op \cong \ho(\bC\op),\quad \ho(\cC)\op \cong \ho(\cC\op).
\]
Shifting and the passage to opposites are compatible in the sense that we have $(\shift\D B)\op \cong \shift{(\D\op)}{B\op}$.
\end{enumerate}
\end{egs}

A prederivator is a \emph{derivator} if --besides two further easily motivated properties-- it `allows for a well-behaved calculus of Kan extensions'. Let \D be a prederivator, $u\colon A\to B$ a functor, and $u^\ast\colon\D(B)\to\D(A)$ the induced restriction functor. If $u^\ast$ admits a left adjoint $u_!\colon \D(A)\to\D(B)$ then we refer to it as a \textbf{left Kan extension}. Dually, right adjoints $u_\ast\colon\D(A)\to\D(B)$ are called \textbf{right Kan extensions}. Thus we follow the established terminology for~$\infty$-categories and refer to the functors as \emph{Kan extensions} as opposed to \emph{homotopy Kan extensions}. There is no risk of confusion since `categorical' Kan extensions are meaningless for an abstract derivator. Similarly, in the special case of $\pi\colon A\to B=\bbone$, we refer to $\pi_!=\colim_A$ as a \textbf{colimit functor} and to~$\pi_\ast=\lim_A$ as a \textbf{limit functor}. 

Recall from classical category theory, that in cocomplete categories left Kan extensions can be calculated pointwise by colimits, and dually (\cite[X.3.1]{maclane}). Axiom (Der4) in the following definition axiomatizes these pointwise formulas. For that purpose one has to consider the following special cases of comma squares
\[
\xymatrix{
(u/b)\ar[r]^-p\ar[d]_-\pi\drtwocell\omit{}&A\ar[d]^-u&&(b/u)\ar[r]^-q\ar[d]_-\pi&A\ar[d]^-u\\
\bbone\ar[r]_-b&B,&&\bbone\ar[r]_-b&B,\ultwocell\omit{}
}
\]
which come with canonical transformations $u\circ p\to b\circ\pi$ and $b\circ\pi\to u\circ q.$

\begin{defn}
  A \textbf{derivator} is a prederivator $\D\colon\cCat\op\to\cCAT$ with the following properties.
  \begin{itemize}[leftmargin=4em]
  \item[(Der1)] $\D\colon \cCat\op\to\cCAT$ takes coproducts to products.  In particular, $\D(\emptyset)$ is the terminal category.
  \item[(Der2)] For any $A\in\cCat$, a morphism $f\colon X\to Y$ is an isomorphism in $\D(A)$ if and only if the morphisms $f_a\colon X_a\to Y_a, a\in A,$ are isomorphisms in $\D(\bbone).$
  \item[(Der3)] Each functor $u^*\colon \D(B) \to\D(A)$ has both a left adjoint $u_!$ and a right adjoint $u_*$.
  \item[(Der4)] For any functor $u\colon A\to B$ and any $b\in B$ the canonical transformation
\[ 
\pi_! p^* \to \pi_! p^* u^* u_! \to \pi_! \pi^* b^* u_! \to b^* u_!
\]
is an isomorphism as is the transformation
\[
b^* u_* \to \pi_* \pi^* b^* u_* \to \pi_* q^* u^* u_* \to \pi_* q^*.
\]
  \end{itemize}
\end{defn}

The axioms (Der1) and (Der3) together imply that $\D(A)$ has small categorical coproducts and products, hence, in particular, initial objects and final objects. These are the only actual 1-categorical (co)limits which must exist in each derivator. This seemingly abstract concept of a derivator encodes key properties of the calculus of (homotopy) Kan extensions which is available in model categories, $\infty$-categories, and ordinary categories, and thus provides a tool to study such examples. Let us again take up the \autoref{egs:prederivators}.

\begin{egs}\label{egs:der}
~ 
\begin{enumerate}
\item A prederivator~$y(\bC)$ is a derivator, called \textbf{represented derivator}, if and only if the category \bC is both complete and cocomplete. In this case $u_!$ and $u_*$ are the usual Kan extension functors.
\item If \bC is an \emph{arbitrary} model category then $\ho(\bC)$ is a derivator, its \textbf{homotopy derivator} (see~\cite{cisinski:idcm} for the general case and \cite{groth:ptstab} for an easy proof in the case of combinatorial model categories). The functors $u_!$ and $u_*$ are derived versions of the functors of $y(\bC)$. 
\item If $\cC$ is a complete and cocomplete $\infty$-category then there is the \textbf{homotopy derivator} $\ho(\cC)$ of $\cC$ (see~\cite{gps:mayer} for a sketch proof).
\item If \D is derivator then so are $\D\op$ and $\D^B$ for $B\in\cCat$ (see \cite[Theorem~1.25]{groth:ptstab}).
\end{enumerate}
\end{egs}

More specific examples are mentioned in \autoref{egs:derivators}. The axioms of a derivator suffice to extend many key features of the ordinary calculus of Kan extensions to the homotopical context. This is based on the theory of \emph{homotopy exact square}, which is arguable the main tool in the study of derivators. Given a derivator \D and a natural transformation living in a square
\begin{equation}
  \vcenter{\xymatrix{
      D\ar[r]^p\ar[d]_q \drtwocell\omit{\alpha} &
      A\ar[d]^u\\
      B\ar[r]_v &
      C
    }}\label{eq:hoexactsq'}
\end{equation}
of small categories, we obtain canonical mate-transformations
\begin{gather}
  q_! p^* \to q_! p^* u^* u_! \xto{\alpha^*} q_! q^* v^* v_! \to v^* u_!  \mathrlap{\qquad\text{and}}\label{eq:hoexmate1'}\\
  u^* v_* \to p_* p^* u^* v_* \xto{\alpha^*} p_* q^* v^* v_* \to p_* q^*\label{eq:hoexmate2'}
\end{gather}
of which one is an isomorphism if and only if the other is. A square \eqref{eq:hoexactsq'} is \textbf{homotopy exact} if and only if the canonical mates \eqref{eq:hoexmate1'} and \eqref{eq:hoexmate2'} are isomorphisms. (Note that if one only asks this to be the case for represented derivators, then one obtains those squares which \emph{satisfy the Beck--Chevalley condition}. Such squares are also referred to as \textbf{Guitart exact squares}; see for example~\cite{guitart:rel-car,guitart:contractible}.)

By definition of derivators, certain comma squares are homotopy exact, and one can establish many more examples. Here, we only collect the examples of constant use in this paper, but see also \cite{ayoub:12,maltsiniotis:exact,groth:ptstab,gps:mayer}.

\begin{egs}\label{egs:htpy}
~ 
\begin{enumerate}
\item If $u\colon A\to B$ is fully faithful, then the square $\id\circ u=\id \circ u$ is homotopy exact. Thus, the canonical maps $\eta\colon\id\to u^\ast u_!$ and $\epsilon\colon u^\ast u_\ast\to\id$ are isomorphisms, i.e., Kan extension along fully faithful functors are fully faithful; see \cite[Proposition~1.20]{groth:ptstab}.
\item For functors $u\colon A\to B$ and $v\colon C\to D$ the squares
\begin{equation}
\vcenter{
\xymatrix{
A\times C\ar[r]^-{u\times\id}\ar[d]_-{\id\times v}&B\times C\ar[d]^-{\id\times v}\\
A\times D\ar[r]_-{u\times\id}&B\times D
}}
\end{equation}
are homotopy exact; see \cite[Proposition~2.5]{groth:ptstab}. Thus, Kan extensions in one variable commute with restriction in an unrelated variable.\label{eq:exactff}
\item If $u\colon A\to B$ is a right adjoint functor, then the square
\[
\xymatrix{
A\ar[r]^-u\ar[d]_-{\pi_A}&B\ar[d]^-{\pi_B}\\
\bbone\ar[r]_-\id&\bbone
}
\]
is homotopy exact (\cite[Proposition~1.18]{groth:ptstab}). Thus, the canonical map $(\pi_A)_!u^\ast\to(\pi_B)_!$ is an isomorphism, which is referred to by saying that the functor $u$ is \textbf{homotopy cofinal}.
\item The passage to mate-transformations is functorial with respect to horizontal and vertical pasting. Consequently, horizontal and vertical pastings of homotopy exact squares are homotopy exact (see \cite[Lemma~1.14]{groth:ptstab}). 
\end{enumerate}
\end{egs}

Combining the 2-out-of-3 property of isomorphisms with the above examples and the compatibility with respect to pasting allows one to establish many key facts about the calculus of Kan extensions in the framework of derivators. 

We close this section with a short discussion of morphisms of derivators and, in particular, parametrized versions of Kan extensions. A \textbf{morphism} of derivators is simply a morphism of underlying prederivators and similarly for \textbf{transformations}. Thus, we define the 2-category $\cDER$ of derivators as the full sub-2-category of $\cPDER$ spanned by the derivators. \autoref{egs:der} can be refined to statements including morphisms. 

An \textbf{adjunction} or \textbf{equivalence} of derivators is an adjunction or equivalence internal to the 2-category of derivators. One observes that a morphism of derivators $F\colon\D\to\E$ is an equivalence if and only if this is the case for each component $F_A\colon\D(A)\to\E(A)$ (see \cite[Proposition~2.9]{groth:ptstab}). Let us say that a morphism of prederivators is \textbf{fully faithful} or \textbf{essentially surjective} if each component has the respective property. Thus, a morphism of derivators is an equivalence if and only if it is fully faithful and essentially surjective.

We recall that a morphism $F$ \textbf{preserves left Kan extensions along $u\colon A\to B$} if the canonical mate-transformation
\[
u_!\circ F_A\to u_!\circ F_A\circ u^\ast\circ u_!\stackrel{\cong}{\to}u_!\circ u^\ast\circ F_B\circ u_!\to F_B\circ u_!
\]
is an isomorphism. A morphism is \textbf{cocontinuous} if it preserves left Kan extensions along all functors $u\colon A\to B$. There are dual notions for right Kan extensions and, in particular, the notion of a \textbf{continuous} morphism. In all these cases it can be shown that a morphism has the respective property if and only if this is the case for all of its shifts (\cite[Corollary~2.7]{groth:ptstab}). For represented derivators these notions reduce to the classical ones of functors preserving certain Kan extensions. 

\begin{lem}[{\cite[Proposition~2.9]{groth:ptstab}}]\label{lem:adj}
A morphism $F\colon\D\to\E$ of derivators is a left adjoint if and only if each component $F_A\colon\D(A)\to\E(A)$ is a left adjoint and $F$ is cocontinuous.
\end{lem}

Examples of adjunctions come from adjunctions between bicomplete categories and Quillen adjunctions. In this paper it will be central to note that also Kan extensions themselves exist in a parametrized form, so that they also induce adjunctions between derivators. Given a functor $u\colon A\to B$, then there is a morphism of derivators $u^\ast\colon\D^B\to\D^A$. By (Der3) this morphism admits levelwise adjoints, and one checks that $u^\ast$ preserves Kan extensions (\cite[Prop.~2.5]{groth:ptstab}). It follows that left and right Kan extensions can be lifted to adjunctions of derivators,
\[
(u_!,u^\ast)\colon\D^A\rightleftarrows\D^B\qquad\text{and}\qquad
(u^\ast,u_\ast)\colon\D^B\rightleftarrows\D^A.
\]
In particular, $u_!,u^\ast$ are cocontinuous and $u^\ast,u_\ast$ are continuous.

We will also use the following abuse of language. If \D is a derivator, then we write $X\in\D$ in order to indicate that there is a category $A\in\cCat$ such that $X\in\D(A)$. We will also use a similar notation in the context of a sub(pre)derivator $\D'\subseteq\D$.

\section{Subderivators defined by exactness conditions}
\label{sec:sub}

In this section we briefly study sub(pre)derivators obtained from shifted derivators by imposing exactness conditions. In particular, this will be used to conclude that Kan extensions along fully faithful functors induce equivalences between such sub(pre)derivators. These rather obvious statements are collected here for the convenience of the reader: in later sections, suitable iterations of such Kan extensions witness that certain quivers/categories are strongly stably equivalent. We suggest that the reader skips this section and only refers back to it as needed.

\begin{lem}
Let $F\colon\D\to\E$ be a morphism of prederivators and, for $A\in\cCat$, let $\E'(A)\subseteq\E(A)$ be the full subcategory given by the essential image of $F_A.$ Then for every $u\colon A\to B$ there is a commutative square
\[
\xymatrix{
\E'(B)\ar@{-->}[d]\ar[r]&\E(B)\ar[d]^-{u^\ast}\\
\E'(A)\ar[r]&\E(A).
}
\]
This defines a prederivator $\cIm(F)=\E'$, the \textbf{image} of $F$, and an \textbf{image factorization} $F\colon\D\to\cIm(F)\to\E$.
\end{lem}
\begin{proof}
This is immediate since for every $X\in\E'(B)$ there exists an object $Y\in\D(B)$ together with isomorphisms
\[
u^\ast(X)\cong u^\ast(F_B(Y))\cong F_A(u^\ast(Y))
\]
where the second isomorphism is the coherence isomorphism belonging to~$F$.
\end{proof}

\begin{lem}
Let $F\colon \D\to\E$ be a morphism of prederivators which is levelwise an equivalence. If $\D$ is a derivator, then so is $\E$.
\end{lem}

\begin{cor}
Let $F\colon\D\to\E$ be a morphism of prederivators which is levelwise fully faithful and let \D be a derivator. The image $\cIm(F)$ is again a derivator and the induced morphism $\D\to\cIm(F)$ is an equivalence of derivators.
\end{cor}

As a special case we obtain the following result.

\begin{cor}\label{cor:kaneq}
Let \D be a derivator and let $u\colon A\to B$ be fully faithful. The image $u_!(\D^A)$ of $u_!\colon\D^A\to\D^B$ is a derivator and $u_!$ restricts to an equivalence of derivators $u_!\colon\D^A\to u_!(\D^A)$.
\end{cor}

For a refined version of this, let us recall the cocone construction. Given a small category $A$, the \textbf{cocone} $A^\rhd$ of~$A$ is the category which is obtained from $A$ by adjoining a new terminal object $\infty$. The cocone construction $(-)^\rhd$ is obviously functorial in~$A$, and there is a natural fully faithful inclusion $i=i_A\colon A\to A^\rhd$. Objects of $\D(A^\rhd)$ are \textbf{(coherent) cocones}, which are \textbf{colimiting} if they lie in the essential image of the left Kan extension $i_!\colon\D(A)\to\D(A^\rhd)$. 

Let us now consider a collection of small cocones and small cones in~$A$, say

\begin{equation}
\vcenter{
\xymatrix{
\cI=\{f_i\colon D_i^\rhd\to A, i\in I\}&\text{and}&\cP=\{g_j\colon E_j^\lhd\to A, j\in J\}.
}\label{eq:cocones}
}
\end{equation}
For a derivator \D and a small category $C$ let
\[
\D^{A,\cI,\cP}(C)\subseteq\D^A(C)
\]
be the full subcategory spanned by the diagrams $X\in\D^A(C)$ such that the cocones $f_i^\ast(X)\in\D^{D_i^\rhd}(C)$ are colimiting and such that the cones $g_j^\ast(X)\in\D^{E_j^\lhd}(C)$ are limiting. Recall from \cite[Corollary~2.6]{groth:ptstab} that this can be checked pointwise. For example, $f_i^\ast(X)$ is colimiting if and only if $f_i^\ast(X)_c\in\D(D_i^\rhd)$ is colimiting for each $c\in C$.

\begin{prop}
$\D^{A,\cI,\cP}\colon C\mapsto \D^{A,\cI,\cP}(C)$ defines a subprederivator of $\D^A$.
\end{prop}
\begin{proof}
We have to check that for $v\colon C\to C'$ the functor $v^\ast\colon\D^A(C')\to\D^A(C)$ restricts to the respective subcategories. Thus, let $X\in\D^{A,\cI,\cP}(C')$ and let us consider one of the cocones $f_i\colon D_i^\rhd\to A$. By assumption $f_i^\ast(X)$ lies in the essential image of the left Kan extension along the fully faithful functor $D_i\to D_i^\rhd.$ But this is equivalent to the fact that $f_i^\ast(X)_{c'}$ is colimiting for all $c'\in C'$, hence this is, in particular, the case for 
\[
f_i^\ast(X)_{v(c)}=v^\ast f_i^\ast(X)_c=f_i^\ast (v^\ast(X))_c,\quad c\in C.
\]
Using \cite[Corollary~2.6]{groth:ptstab} once more we deduce that $f_i^\ast(v^\ast(X))$ is colimiting. This proves that $v^\ast$ restricts to a functor $\D^{A,\cI,\cP}(C')\to \D^{A,\cI,\cP}(C)$, and we thus have a prederivator $\D^{A,\cI,\cP}$. 
\end{proof}

In the context of $\cI,\cP$ as in \eqref{eq:cocones} and a functor $u\colon A\to B$, we denote by $u(\cI),u(\cP)$ the collections of (co)cones obtained from $\cI,\cP$ by composition with~$u$. There is the following refinement of \autoref{cor:kaneq}.

\begin{prop}\label{prop:kaneq}
Let \D be a derivator, $A\in\cCat$, and let $u\colon A\to B$ be fully faithful. The equivalence $u_!\colon\D^A\to u_!(\D^A)$ restricts to an equivalence of prederivators
\[
u_!\colon\D^{A,\cI,\cP}\to u_!(\D^A)\cap \D^{B,u(\cI),u(\cP)}.
\]
In particular, if $\D^{A,\cI,\cP}$ is actually a \emph{derivator}, then so is $u_!(\D^A)\cap \D^{B,u(\cI),u(\cP)}$ and $u_!$ is an equivalence of derivators.
\end{prop}
\begin{proof}
This is obvious.
\end{proof}

We are mostly interested in this result in the case that the image $u_!(\D^A)$ can be described by certain exactness conditions, i.e., in the case that $u_!(\D^A)=\D^{B,\cI',\cP'}$. Under these assumptions, if $\D^{A,\cI,\cP}$ is a \emph{derivator}, then $u_!(\D^{A,\cI,\cP})$ is again a derivator defined by exactness conditions. In fact, this image derivator is given by imposing the exactness conditions $u(\cI)\cup\cI',u(\cP)\cup\cP'$, and we thus obtain an equivalence of derivators
\[
u_!\colon \D^{A,\cI,\cP}\stackrel{\simeq}{\to}\D^{B,u(\cI)\cup\cI',u(\cP)\cup\cP'}.
\]

\begin{rmk}\label{rmk:kaneq}
There is also the following variant of these results. Instead of asking certain (co)cones to be (co)limiting (co)cones, one can consider sets of functors
\begin{equation}
\vcenter{
\xymatrix{
\cI=\{D_i\stackrel{u_i}{\to} D'_i\stackrel{f_i}{\to} A, i\in I\}&\text{and}&
\cP=\{E_j\stackrel{v_j}{\to} E'_j\stackrel{g_j}{\to} A, j\in J\},
}\label{eq:relcocones}
}
\end{equation}
with $u_i$ and $v_j$ fully faithful functors. In that case we obtain a full subprederivator $\D^{A,\cI,\cP}\subseteq\D^A$ spanned by diagrams~$X$ such that $f_i^\ast(X)$ lies in the essential image of $(u_i)_!$ and such that $g_j^\ast(X)$ lies in the essential image of $(v_j)_\ast$. A fully faithful functor $u\colon A\to B$ induces then similar equivalences as in \autoref{prop:kaneq}.
\end{rmk}

In the remainder of this paper we use the results of this section without further reference.

\section{Stable derivators and exact morphisms}
\label{sec:stable}

We give a brief review of pointed and stable derivators, and indicate how to define functors like suspensions and loops, cofibers and fibers in this context. The main aim of this paper and its sequels is to use more refined combinations of similar Kan extensions in order to establish abstract tilting results.

\begin{defn} 
A derivator \D is \textbf{pointed} if $\D(\bbone)$ has a zero object. 
\end{defn}

Derivators represented by pointed categories are pointed as are homotopy derivators of pointed model categories and pointed $\infty$-categories. If $\D$ is pointed then so are $\D^B$ and $\D\op$. It follows that the categories $\D(A)$ have zero objects which are preserved by restriction and Kan extension functors. 

In the pointed context, important functors are given by \emph{left} or \emph{right extensions by zero}. These are special cases of Kan extensions along inclusions of \emph{cosieves} and \emph{sieves}, respectively.
Recall that $u\colon A\to B$ is a \textbf{sieve} if it is fully faithful, and for any morphism $b\to u(a)$ in $B$, there exists an $a'\in A$ with $u(a')=b$. Dually, there is the notion of a \textbf{cosieve}, and we know already that Kan extensions along (co)sieves are fully faithful.

\begin{lem}[{\cite[Prop.~1.23]{groth:ptstab}}]\label{lem:extbyzero}
Let \D be a pointed derivator and let $u\colon A\to B$ be a sieve. Then $u_\ast\colon\D^A\to\D^B$ induces an equivalence onto the full subderivator of $\D^B$ spanned by all diagrams $X\in\D^B$ such that $X_b$ is zero for all $b\notin u(A)$.
\end{lem}
\noindent
The functor $u_\ast$ gives us \textbf{right extension by zero}, and dually there are \textbf{left extensions by zero}.

In the framework of pointed derivators one can define \textbf{suspensions} and \textbf{loops}, \textbf{cofibers} and \textbf{fibers}, and similar constructions. In particular, we have adjunctions of derivators
\[
(\Sigma,\Omega)\colon\D\rightleftarrows\D\quad\text{and}\quad
(\cof,\fib)\colon\D^{[1]}\rightleftarrows\D^{[1]},
\]
where $[1]$ is the category $(0\to 1)$ with two objects and a unique non-identity morphism.
In order to sketch the construction of the suspension $\Sigma$ and the cofiber $\cof$, let us consider the category $\square=[1]^2=[1]\times[1]$,
\[\vcenter{\xymatrix@-.5pc{
    (0,0)\ar[r]\ar[d] &
    (1,0)\ar[d]\\
    (0,1)\ar[r] &
    (1,1).
  }}\]
Let $\ulcorner$ and $\lrcorner$ be the full subcategories obtained by removing $(1,1)$ and $(0,0)$, respectively. These categories come with inclusions $i_\ulcorner\colon \mathord{\ulcorner} \into \square$ and $i_\lrcorner\colon \mathord{\lrcorner} \into \square$.
Since $i_\ulcorner$ is fully faithful, so is~$(i_\ulcorner)_!$ and an object $Q\in\D(\square)$ is \textbf{cocartesian} if it lies in the essential image of $(i_\ulcorner)_!$. Dually, the \textbf{cartesian} squares are precisely the ones in the essential image of $(i_\lrcorner)_\ast$. Given a coherent morphism $f\colon x\to y\in\D([1])$, then in order to define the cofiber $\cof(f)$ one restricts the cocartesian diagram
\[
\xymatrix{
x\ar[r]^-f\ar[r]\ar[d]&y\ar[d]^-{\cof(f)}\\
0\ar[r]&z
}
\]
which in turn is obtained from~$f$ by a right extension by zero followed by a left Kan extension. Similarly, the suspension $\Sigma x$ is the target of $\cof(f\colon x\to 0).$ Here, $x\to 0$ is obtained from $x$ by a right extension by zero. Thus the defining cocartesian square looks like
\[
\xymatrix{
x\ar[r]\ar[r]\ar[d]&0\ar[d]\\
0\ar[r]&\Sigma x.
}
\]
Finally, \textbf{(coherent) cofiber sequences} are coherent diagrams
\[
\xymatrix{
x\ar[r]\ar[d]&y\ar[r]\ar[d]&0\ar[d]\\
0\ar[r]&z\ar[r]&w
}
\]
such that the two squares are cocartesian and such that the two objects vanish. An arbitrary morphism $(x\to y)\in\D([1])$ can be extended to a cofiber sequence by a right extension by zero followed by a left Kan extension. Since the compound square is cocartesian (\cite[Corollary~4.10]{gps:mayer}), there is a natural isomorphism $\Sigma x\toiso w.$ In particular, using this isomorphism we see that every cofiber sequence induces an underlying \textbf{\emph{incoherent} cofiber sequence}
\[
x\to y\to z\to \Sigma x,
\]
which is an ordinary diagram in $\D(\bbone).$

\begin{defn}
A pointed derivator is \textbf{stable} if the classes of cocartesian squares and cartesian squares coincide. These squares are then called \textbf{bicartesian}.
\end{defn}
\noindent
It can be shown that for a pointed derivator stability can also be characterized by the fact that $(\cof,\fib)$ or $(\Sigma,\Omega)$ is an equivalence (\cite[Theorem~7.1]{gps:mayer}). We will obtain an additional reformulation in \autoref{cor:stable}. The homotopy derivators of stable model categories and stable $\infty$-categories are stable. If \D is stable then so is the shifted derivator $\D^B$ as is $\D\op$. Recall that a derivator is \textbf{strong} if it satisfies:
  \begin{itemize}[leftmargin=4em]
  \item[(Der5)] For any $A$, the induced functor $\D(A\times [1]) \to \D(A)^{[1]}$ is full and essentially surjective.
  \end{itemize}
This property does not play an essential role in the \emph{theory} of derivators. It is only used if one wants to relate \emph{properties} of stable derivators to the existence of \emph{structure} on its values. Represented derivators and homotopy derivators associated to model categories or $\infty$-categories are strong as are shifts and opposites of strong derivators. The above-mentiond incoherent cofiber se	quences can be used to prove the following.

\begin{thm}[{\cite[Theorem~4.16]{groth:ptstab}}]\label{thm:triang}
Let \D be a strong, stable derivator. Then $\D(A)$ admits a (canonical) triangulation. 
\end{thm}

To make precise one way in which these triangulations are canonical, let us recall that a morphism between stable derivators is \textbf{exact} if it preserves zero objects and bicartesian squares. 

\begin{prop}[{\cite[Proposition~4.18]{groth:ptstab}}]\label{prop:canonical}
Let \D and \E be strong, stable derivators and let $F\colon\D\to\E$ be exact. Then $F_A\colon\D(A)\to\E(A)$ can be turned into an exact functor with respect to the triangulations of \autoref{thm:triang}.  
\end{prop}

This applies, in particular, to the morphisms $u_!,u^\ast,u_\ast$ induced by a functor $u\colon A\to B$ showing that the triangulations at the various levels as guaranteed by \autoref{thm:triang} are compatible. Similarly, one can show that a natural transformation between such functors induces exact transformations between exact functors of triangulated categories. 

An additional result reinforcing the correctness of the above definition of an exact morphism of stable derivators can be found in \cite{ps:linearity}. The result shows that exact morphisms preserve sufficiently finite (co)limits. 

%

In later sections we will establish abstract tilting results for arbitrary stable derivators. For this purpose it is important to be able to recognize bicartesian squares in larger diagrams. In many cases where it seems obvious that certain squares are bicartesian, a formal justification is provided by the following recognition result of Franke \cite{franke:adams}.

\begin{lem}[{\cite[Prop.~3.10]{groth:ptstab}}]\label{lem:detection}
  Suppose $u\colon C\to B$ and $v\colon [1]^2 \to B$ are functors, with $v$ injective on objects, and let $b = v(1,1)\in B$.
  Suppose furthermore that $b\notin u(C)$, and that the functor $\mathord\ulcorner \to (B\setminus b)/b$ induced by $v$ has a left adjoint.
  Then for any derivator \D and any $X\in \D(C)$, the square $v^* u_! X$ is cocartesian.
\end{lem}

Although this result looks a bit technical, it turns out to be very useful in specific situations. We will also need the following variant for arbitrary \emph{cocones}. 

\begin{lem}[{\cite[Lemma~4.5]{gps:mayer}}]\label{lem:detectionplus}
Let $A\in\cCat$, and let $u\colon C\to B$, $v\colon A^\rhd\to B$ be functors.
Suppose that there is a full subcategory $B'\subseteq B$ such that
\begin{enumerate}
\item $u(C) \subseteq B'$ and $v(\infty) \notin B'$,
\item $vi\colon A\to B$ factors through the inclusion $B'\subseteq B$, and
\item the functor $A \to B' / v(\infty)$ induced by $v$ has a left adjoint.
\end{enumerate}
Then for any derivator \D and any $X\in\D(C)$, the diagram $v^* u_! X$ is in the essential image of $i_!$.
In particular, $(v^* u_! X)_\infty$ is the colimit of $i^* v^* u_! X$.
\end{lem}

\section{Strongly stably equivalent categories}
\label{sec:strong}

In this section we introduce strongly stably equivalent categories, and mention examples of stable derivators to which the abstract tilting results will apply. In this paper, our results are mostly about quivers, hence let us be precise about this notion. A quiver $Q$ is a quadruple $(Q_0,Q_1,s,t)$ consisting of a set of vertices $Q_0$, a set of arrows $Q_1$, and two maps $s,t\colon Q_1\to Q_0$, called the source map and target map, respectively. In other words, a quiver is simply an oriented graph and we hence allow loops at a vertex or multiple arrows between two vertices. A quiver $Q=(Q_0,Q_1,s,t)$ is finite if $Q_0$ and $Q_1$ are finite sets. In what follows we often identify a quiver with the category freely generated by it.

Let $Q$ and $Q'$ be quivers and let $k$ be a field. Recall that $Q$ and $Q'$ are \textbf{derived equivalent} (over~$k$) if the path-algebras $kQ$ and $kQ'$ are derived equivalent, i.e., if there is an exact equivalence of derived categories
$$D(kQ)\stackrel{\Delta}{\simeq}D(kQ').$$
The main aim of this paper and its sequels is to show that in quite some classical cases of derived equivalent quivers we are actually given strongly stably equivalent quivers in a sense we make precise now.

Since identity morphisms are exact and compositions of exact morphisms are again exact, we can consider the $2$-category $\cDER_{\mathrm{St},\mathrm{ex}}\subseteq\cDER$ of stable derivators, exact morphisms, and arbitrary transformations. Note that the shifting operation defines a $2$-functor
\[
\cCat\op\times\cDER\to\cDER\colon (A,\D)\mapsto \D^A.
\]
Hence, for every $A\in\cCat$ we obtain an induced $2$-functor $(-)^A\colon\cDER\to\cDER$ which can be restricted to
\[
(-)^A\colon\cDER_{\mathrm{St},\mathrm{ex}}\to\cDER.
\]

\begin{defn}
Two small categories $A$ and $A'$ are \textbf{strongly stably equivalent}, in notation $A\sse A'$, if there is a pseudo-natural equivalence between the $2$-functors
\[
\Phi\colon(-)^A\simeq (-)^{A'}\colon\cDER_{\mathrm{St},\mathrm{ex}}\to\cDER.
\]
We call such a pseudo-natural equivalence a \textbf{strong stable equivalence}.
\end{defn}

Thus, a strong stable equivalence $\Phi\colon A\sse  A'$ consists of the following datum. 
\begin{enumerate}
\item For \emph{every} stable derivator \D there is an equivalence of derivators 
\[
\Phi_\D\colon\D^A\simeq\D^{A'}.
\]
\item Moreover, for every exact morphism of stable derivators $F\colon\D\to\E$ there is a natural isomorphism 
$\gamma_F\colon F\circ \Phi_\D\to \Phi_E\circ F,$
\[
\xymatrix{
\D^A\ar[r]^-{\Phi_\D}_-\simeq\ar[d]_-F\drtwocell\omit{\cong}&\D^{A'}\ar[d]^-F\\
\E^A\ar[r]^-\simeq_-{\Phi_{\E}}&\E^{A'}.
}
\]
\end{enumerate}
This datum is supposed to satisfy the usual coherence properties of a pseudo-natural transformation.

Before we comment on this notion, we collect the following two facts, and begin with some obvious closure properties.

\begin{lem}\label{lem:closure}
Let $A,A'$, $B,B'$, and $A_i,A_i',i\in I,$ be small categories.
\begin{enumerate}
\item The relation of `being strongly stably equivalent' $\sse$ defines an equivalence relation. 
\item Equivalent categories are strongly stably equivalent.
\item If $A\sse A'$ and $B\sse B'$, then $A\times B\sse A'\times B'$.
\item If $A_i\sse A_i'$ for $i\in I$, then $\bigsqcup A_i\sse \bigsqcup A_i'$.
\end{enumerate}
\end{lem}
\begin{proof}
The first two are obvious and (iv) follows directly from (Der1). To establish~(iii), let us assume that~\D is a stable derivator. Since stable derivators are stable under shifts, by assumption there is a chain of pseudo-natural equivalences
\[
\D^{A\times B}\cong(\D^A)^B\simeq(\D^A)^{B'}\cong(\D^{B'})^A\simeq (\D^{B'})^{A'}\cong\D^{A'\times B'}.
\]
But this implies that $A\times B$ and $A'\times B'$ are also strongly stably equivalent.
\end{proof}

On the other hand, classical results from representation theory provide us with a non-trivial necessary condition for quivers to be strongly stably equivalent.

\begin{prop} \label{prop:strongly-equiv-necessary}
Let $Q$ and $Q'$ be finite quivers without oriented cycles. Then $Q\sse Q'$ only if the underlying non-oriented graphs of $Q$ and $Q'$ are isomorphic.
\end{prop}

\begin{proof}
Let $k$ be a field. It follows from Happel's results~\cite{happel:fdalgebra} that $\D_k(Q)$ is equivalent to $\D_k(Q')$ only if the underlying graphs of $Q$ and $Q'$ are isomorphic. To see that, if $\D_k(Q)$ is equivalent to $\D_k(Q')$, then the categories of compact objects are also equivalent. In other words $D^b(kQ) \simeq D^b(kQ')$ for the corresponding bounded derived categories of finitely generated $kQ$- and $kQ'$-modules, respectively. Attached to each $D^b(kQ)$ and $D^b(kQ')$, there is a combinatorial piece of data called the Auslander--Reiten quiver~\cite[\S4]{happel:fdalgebra}. In our case these quivers, say $\Gamma$ and $\Gamma'$, have been explicitly computed in~\cite[Corollary~4.5]{happel:fdalgebra} and they are invariant under equivalence of categories. That is, $\Gamma \cong \Gamma'$. Moreover, it follows from the shape of $\Gamma,\Gamma'$ that $\lZ Q \cong \lZ Q'$, where $\lZ Q$ and $\lZ Q'$ are so-called repetitive quivers of $Q$ and $Q'$, respectively, in the sense of the construction in~\cite[\S VII.4, p.~250]{ARS:representation}. It is an easy fact that the underlying graphs of $Q$ and $Q'$ can be reconstructed from $\lZ Q$ and $\lZ Q'$, respectively, see for instance~\cite[p.~204]{riedtmann:koecher}. Thus, the underlying graphs of $Q$ and $Q'$ must be isomorphic.
\end{proof}

We now collect a few specific examples of stable derivators in order to indicate the generality of the abstract tilting results we obtain here and in the sequels. In these cases the canonical triangulations of \autoref{thm:triang} agree with the classical ones. These examples re-emphasize that being strongly stably equivalent a priori is a much stronger assumption than being merely derived equivalent. For a substantially longer list of specific examples of stable homotopy theories we refer to \cite{schwede-shipley:stable}.

\begin{egs}\label{egs:derivators}
~ 
\begin{enumerate}
\item Let $R$ be an ordinary, not necessarily commutative ring. Then there is the stable model category $\Ch(R)$ of unbounded chain complexes of~$R$-modules (with the projective model structure, see \cite[\S2.3]{hovey:modelcats}). The associated homotopy derivator $\D_R=\D_{\Ch(R)}$ is stable and the underlying category is the ordinary derived category $D(R)$. The equivalence $\Mod(RQ)\simeq\Mod(R)^Q,$ where $RQ$ is the path-algebra of~$Q$ over~$R$, yields a Quillen equivalence at the level of unbounded chain complexes and hence to an equivalence of stable derivators
\[
\D_R^Q\simeq\D_{RQ},
\]
showing that strongly stably equivalent quivers are derived equivalent. As mentioned, such equivalences at the level of derived categories are usually obtained using tilting theory (see~\cite{tilting} and the many references therein) and have been prominently studied for finite dimensional algebras over a field.
\item The same works more generally for Grothendieck abelian categories~$\cA$, e.g., categories of quasi-coherent modules on schemes. There is the so-called injective model structure on the category $\Ch(\cA)$ of unbounded chain complexes in~$\cA$ (see e.g.~\cite{hovey:model-chain-sheaves} or \cite[Chapter~1]{lurie:ha}) which gives rise to a stable derivator $\D_\cA$. In particular, if $X$ is a quasi-compact, quasi-separated scheme, then there is the stable derivator $\D_X$ of unbounded chain complexes of quasi-coherent $\mathcal{O}_X$-modules \cite{hovey:model-chain-sheaves}. Yet more generally, analogous constructions yield model structures for Quillen's exact categories and hence derivators (at least for finite shapes), see~\cite[Appendice]{m:k-theory-deriv} and \cite{gillespie:exact-model,st:exact-model}.
Thus strongly stably equivalent quivers have equivalent homotopy theories of representations with values in Grothendieck abelian and various exact categories. In particular, we obtain exact equivalences of underlying derived categories.
\item There are variants of this in the differential-graded context (see for example \cite{hinich:homological,schwede-shipley:algebras,fresse:modules}). Given a differential-graded algebra $A$ over an arbitrary ground ring, we can consider functors from $Q$ to dg-modules over~$A$. Endowed with suitable model structures this yields the stable derivator $\D_A$ of dg-modules over~$A$. For a quiver~$Q$ we obtain an equivalence $\D_A^Q\simeq\D_{AQ}$ where $AQ$ is a differential-graded version of the usual path-algebra. Hence, if two quivers are strongly stably equivalent then the homotopy theories of \textbf{differential-graded representations} of the quivers are equivalent.
\item Another algebraic context where one can apply the results is \emph{stable module theory} and \emph{representation theory of groups}. In this context, let $R$ be a quasi-Frobenius ring or, more generally, an Iwanaga--Gorenstein ring. That is, $R$ is supposed to be left and right noetherian and of finite left and right self-injective dimension \cite[\S9.1]{enochs-jenda:rel-hom-alg}. Then there are two model structures on $\Mod (R)$, the so-called Gorenstein projective and Gorenstein injective model structures~\cite[Theorem~8.6]{hovey:rep-th}, which are Quillen equivalent via the identity functor on $\Mod (R)$. Thus, up to equivalence, these model categories induce the same derivator $\D^\mathrm{Gor}_R$. The model structures and hence the derivators are well known to be stable; see~\cite[Corollary~1.1.16]{becker:model-sing} for a proof in the model-theoretic context. Classically, the base category $\D^\mathrm{Gor}_R(\bbone)$ or its subcategories have been extensively studied as the stable categories of Gorenstein projective (also known as maximal Cohen-Macaulay) and Gorenstein injective modules; see for instance~\cite{auslander-bridger:stable-mod,buchweitz:mcm,enochs-jenda:rel-hom-alg,holm:gorenstein} and references therein.
\item A special case of the above is $R = kG$, the group algebra of a finite group $G$ over a field $k$. As then $R$ has a natural Hopf algebra structure compatible with the Gorenstein projective model structure, $\D^\mathrm{Gor}_R$ becomes a monoidal stable derivator. The structure of the base category has been in this case studied for instance in~\cite{benson-rickard-carlson:thick-stmod,rickard:idemp,bik:stratification-finite-gp}.
\item Similarly to the differential-graded context, we can also consider spectra in the sense of topology (see for example \cite{hss:symmetric,ekmm:rings,mmss:diagram}). Choosing one of these approaches, let $E$ be a symmetric ring spectrum. Then the category of $E$-module spectra can be endowed with a stable model structure \cite{hss:symmetric} and there is hence the associated stable derivator $\D_E$ of $E$-module spectra. We can think of $\D_E^Q$ as the derivator of \textbf{spectral representations} of~$Q$ over~$E$. Thus, any pair of strongly stably equivalent quivers will also give us a tilting result for spectral representations.
\item More abstractly, let us recall that many of the typical approaches to axiomatic homotopy theory forget to derivators. Consequently, there is an entire zoo of further examples induced by stable model categories, stable $\infty$-categories \cite{lurie:ha}, or stable cofibration categories \cite{schwede:topologicaltc} (possibly only after restricting the class of allowable shapes of diagrams which would still be enough to obtain the tilting results). In particular, this includes examples of interest coming from equivariant stable homotopy theory \cite{mandell-may:equivariant,lms:equivariant}, motivic stable homotopy theory \cite{voevodsky:a1,morel-voevodsky:a1,jardine:motivic} or parametrized stable homotopy theory \cite{may-sigurdsson:parametrized,ABGHR:units,ABG:twists}. The statement that two quivers are strongly stably equivalent really means that the homotopy theories of \textbf{abstract representations} of the quivers are pseudo-naturally equivalent. For many more examples of stable model categories arising in various areas of algebra, geometry, and topology see \cite{schwede-shipley:stable}. Another candidate member of this zoo where one might expect the existence of a stable derivator is Kasparov's bivariant K-theory of C*-algebras; see~\cite{meyer-nest:filtrated-k-th,koehler:equivariant}.
\end{enumerate}
\end{egs}

Now, having all these examples in mind, let us emphasize that, by definition of strongly stably equivalent quivers or categories, the strong stable equivalences of shifted derivators are \emph{pseudo-natural with respect to exact morphisms}. To fill this with more life, let us recall that Quillen adjunctions between stable model categories induce exact morphisms of homotopy derivators. Similarly, it is expected that exact functors between stable $\infty$-categories induce exact morphisms of homotopy derivators. For strongly stably equivalent quivers or categories this has the consequence that the corresponding strong stable equivalences of shifted derivators commute with many typical constructions. In particular, this implies that strong stable equivalences are pseudo-natural with respect to various kinds of restriction of scalar functors, induction and coinduction functors as well as (Bousfield) localizations and colocalizations.



\section{Abstract tilting theory for Dynkin quivers of type $A$}
\label{sec:An}

In this section we establish an abstract tilting result for categories which are given by `finite zig-zags of morphisms'. These are the finite categories associated to the Dynkin diagrams $A_n$ but with arbitrary orientations of the edges. More specifically, given a natural number $n$, a \textbf{zig-zag of length $n$} is a category freely generated by a connected, oriented graph with precisely $n+1$ vertices $j_0,\ldots, j_n$ and $n$ edges $\alpha_i,1\leq i\leq n,$ such that either $\alpha_i\colon j_{i-1}\to j_i$ or $\alpha_i\colon j_i\to j_{i-1}$. The morphisms of the first kind are \textbf{order-preserving}, the other ones are \textbf{order-reversing}. If all morphisms in the zig-zag are order-preserving, then $J$ is simply the ordinal $[n]=(0<1<\ldots<n)$, and dually.

As an example, the zig-zags of length two are precisely the categories freely generated by the graphs
\[
j_0\to j_1\to j_2,\quad j_0\ot j_1\to j_2,\quad j_0\to j_1\ot j_2,\quad j_0\ot j_1\ot j_2.
\]
A special case of the tilting result of this section will imply that for an arbitrary stable derivator \D there are (exact) equivalences of (triangulated) categories
\[
\D(j_0\to j_1\to j_2)\simeq\D(j_0\ot j_1\to j_2)\simeq \D(j_0\to j_1\ot j_2)\simeq\D(j_0\ot j_1\ot j_2).
\]
As a warm-up, let us establish the existence of such equivalences at the level of underlying categories. Obviously the first and the last category are equivalent because the two shapes under consideration are isomorphic. An equivalence of $\D(j_0\ot j_1\to j_2)$ and $\D(j_0\to j_1\ot j_2)$ is easily deduced from the stability of \D. In fact, using the fully faithfulness of Kan extensions along fully faithful functors, the left Kan extension along $\ulcorner\to\square$ induces an equivalence 
\[
\D(j_0\ot j_1\to j_2)=\D(\ulcorner)\simeq\D(\square)^\ex,
\]
where $\D(\square)^\ex\subseteq\D(\square)$ is the full subcategory spanned by the bicartesian squares. Similarly, right Kan extension along the inclusion $\lrcorner\to\square$ induces an equivalence $\D(j_0\to j_1\ot j_2)\simeq\D(\square)^\ex,$ so that we obtain the desired equivalence. By duality, it suffices now to obtain an equivalence $\D([2])\simeq\D(\lrcorner).$ The category $\D([2])$ is equivalent to the category of coherent diagrams
\begin{equation}
\vcenter{\xymatrix@-.5pc{
    &
    w\ar[r]\ar[d]&
    0\ar[d]\\
    x\ar[r] &
    y\ar[r] &
    z
  }}\label{eq:A2}
\end{equation}
such that the square is bicartesian and the object in the upper right corner is a zero object. In fact, an equivalence sending $x\to y\to z$ to a coherent diagram \eqref{eq:A2} is given by a left extension by zero followed by a right Kan extension. Similarly, given a coherent diagram $x\to y\ot w$, a right extension by zero followed by a left Kan extension yields a coherent diagram of the form \eqref{eq:A2}, and this defines an equivalence of categories. Putting these equivalences together we obtain the desired equivalence $\D([2])\simeq\D(\lrcorner)$. These steps can of course be lifted to yield similar morphisms of derivators, so that we also obtain corresponding pseudo-natural equivalences of derivators, and hence in particular exact equivalences of triangulated categories.

The case of arbitrary finite zig-zags is obtained by an iteration of these constructions. We introduce some notation and terminology which helps organizing the steps. The construction is essentially a variant of Waldhausen's $S_\bullet$-construction from Algebraic K-Theory \cite{waldhausen:k-theory}. 

Let $J$ be a zig-zag of length~$n$ and let $k$ be the number of order-reversing arrows in $J$. There is a unique functor $\phi\colon J\to[n-k]\times[k]$ which sends $j_0$ to $(0,k)$ and which is injective on objects. We then have $\phi(j_n)=(n-k,0),$ and~$\phi$ allows us to consider~$J$ as a full subcategory of $[n-k]\times[k].$ More explicitly, this functor can be defined inductively. Let $\phi$ be already defined for the objects $j_0,\ldots,j_l$ and morphisms $\alpha_1,\ldots,\alpha_l$. If $\phi(j_l)$ is given by $(x_l,y_l),$ then we set
\[
\phi(j_{l+1})= \left\{  \begin{array}{l@{\quad,\quad}l}(x_l+1,y_l)&\alpha_{l+1}\;\text{order-preserving},\\ (x_l,y_l-1)&\alpha_{l+1}\;\text{order-reversing,}\end{array}\right.
\]
and this also forces a definition of $\phi(\alpha_{l+1}).$ 

For each natural number $n$ and $0<k<n$, let us write $J_{n-k,k}$ for the zig-zag of length $n$ given by $n-k$ order-preserving morphisms `followed by' $k$ order-reversing ones. Let $J$ be an arbitrary zig-zag of length $n$ which is not isomorphic to $[n]$ or its dual. Then, by the \textbf{convex hull} $[J,J_{n-k,k}]$ of $J$ and $J_{n-k,k}$ we mean the full subcategory of $[n-k]\times[k]$ spanned by all objects between $J$ and $J_{n-k,k}$, or, more precisely, the up-set generated by~$J.$

As a specific example, let us consider the zig-zag $j_0\to j_1\ot j_2\ot j_3\to j_4\ot j_5$ of length $n=5$ with $k=3$ order-reversing morphisms. The following picture shows the standard identification with a subcategory of $[2]\times[3]$, the zig-zag $J_{2,3}$, and also the convex hull of these two zig-zags containing two squares:

\begin{equation}
\vcenter{
\tiny{
\xymatrix@-.7pc{
&&(2,0)\ar[d]&
&&(2,0)\ar[d]&
&&(2,0)\ar[d]\\
&(1,1)\ar[d]\ar[r]&(2,1)&
&&(2,1)\ar[d]&
&(1,1)\ar[d]\ar[r]&(2,1)\ar[d]\\
&(1,2)\ar[d]&&
&&(2,2)\ar[d]&
&(1,2)\ar[d]\ar[r]&(2,2)\ar[d]\\
(0,3)\ar[r]&(1,3)&&
(0,3)\ar[r]&(1,3)\ar[r]&(2,3)&
(0,3)\ar[r]&(1,3)\ar[r]&(2,3)
}
}
}
\label{eq:example}
\end{equation}

\begin{lem}\label{lem:Annf}
Any zig-zag $J$ of length $n$ with $k$ order-reversing morphisms is strongly stably equivalent to $J_{n-k,k}$.
\end{lem}
\begin{proof}
In the case of $k=0$ or $k=n$ the result is immediate. So let us assume that $0<k<n$. Let $K\subseteq [n-k]\times[k]$ be the convex hull of $J$ and $J_{n-k,k}$ which comes with fully faithful inclusion functors $i\colon J\to K$ and $j\colon J_{n-k,k}\to K$. Moreover, let \D be a stable derivator. A repeated application of \autoref{lem:detection} implies that $i_!\colon \D^J\to\D^K$ restricts to a pseudo-natural equivalence of derivators $i_!\colon\D^J\simeq\D^{K,\ex}$, where $\D^{K,\ex}\subseteq\D^K$ denotes the full subderivator given by the coherent diagrams of shape~$K$ such that all squares are bicartesian. By a similar repeated application of \autoref{lem:detection} it follows that also $j_\ast\colon\D^{J_{n-k,k}}\to\D^K$ restricts to a pseudo-natural equivalence $j_\ast\colon\D^{J_{n-k,k}}\simeq\D^{K,\ex}$, and we conclude by putting these together, $\D^J\simeq\D^{K,\ex}\simeq\D^{J_{n-k,k}}$.
\end{proof}

\begin{lem}\label{lem:Anind}
For $0<k<n$ the quivers $J_{n-k,k}$ and $J_{n-k+1,k-1}$ are strongly stably equivalent.
\end{lem}
\begin{proof}
Let \D be a stable derivator and let $K$ be the zig-zag of length $n+1$ obtained from $J_{n-k,k}$ by adding an order-preserving morphism $\alpha_{n+1}.$ Formally, $K$ is the full subcategory of $[n-k+1]\times [k]$ spanned by $J_{n-k,k}$ and $(n-k+1,0)$. This category comes with an obvious inclusion $i_1\colon J_{n-k,k}\to K$, which is the inclusion of a sieve. Hence, the fully faithful $(i_1)_\ast\colon\D^{J_{n-k,k}}\to\D^K$ is right extension by zero, and induces a pseudo-natural equivalence onto the full subderivator of $\D^K$ spanned by all diagrams vanishing at $(n-k+1,0)$ (\autoref{lem:extbyzero}). Let $L\subseteq [n-k+1]\times[k]$ be the full subcategory spanned by $K$ and $J_{n-k+1,k-1}$. There is an obvious inclusion $i_2\colon K\to L$, and a repeated application of \autoref{lem:detection} implies that $(i_2)_!\colon\D^K\to\D^L$ induces a pseudo-natural equivalence onto the full subderivator spanned by all diagrams making all squares cocartesian. Thus, $(i_2)_!(i_1)_\ast\colon\D^{J_{n-k,k}}\to\D^L$ induces a pseudo-natural equivalence onto the full subderivator $\E\subseteq\D^L$ spanned by all diagrams making the squares cocartesian and vanishing on $(n-k+1,0)$.

Similarly, there is the obvious inclusion $j_1\colon J_{n-k+1,k-1}\to J_{n-k+1,k}$ of a cosieve, so that $(j_1)_!\colon\D^{J_{n-k+1,k-1}}\to\D^{J_{n-k+1,k}}$ is left extension by zero, and hence induces a pseudo-natural equivalence onto the full subderivator of $\D^{J_{n-k+1,k}}$ spanned by all diagrams vanishing at $(n-k+1,0)$ (\autoref{lem:extbyzero}). Note that $L$ is also the full subcategory of $[n-k+1]\times[k]$ spanned by $J_{n-k,k}$ and $J_{n-k+1,k}$, and let $j_2\colon J_{n-k+1,k}\to L$ be the inclusion. A repeated application of \autoref{lem:detection} implies that $(j_2)_\ast\colon\D^{J_{n-k+1,k}}\to\D^L$ induces a pseudo-natural equivalence onto the full subderivator of $\D^L$ spanned by all diagrams making the squares cartesian. But since \D is stable, $(j_2)_\ast(j_1)_!\colon\D^{J_{n-k+1,k-1}}\to\D^L$ induces an equivalence onto the same subderivator $\E\subseteq \D^L$ as $(i_2)_!(i_1)_\ast\colon\D^{J_{n-k,k}}\to\D^L$. Thus, we obtain a pseudo-natural equivalence of derivators $\D^{J_{n-k,k}}\simeq\E\simeq\D^{J_{n-k+1,k-1}}$.
\end{proof}

These two lemmas immediately give the following result. A similar result for zig-zags of length two in the model category of module spectra over a commutative symmetric ring spectrum was written up in \cite{schwede:Morita}.

\begin{thm}\label{thm:An}
Two finite zigzags are strongly stably equivalent if and only if they are of the same length.
\end{thm}

\begin{proof}
We first show that an arbitrary zigzag $J$ of length $n$ is strongly stably equivalent to $[n]$. In the special cases of $J\cong[n]$ or $J\cong[n]\op$ the result is immediate. So let us assume that there are $k$ order-reversing maps in $J$ with $0<k<n$. In that case \autoref{lem:Annf} and \autoref{lem:Anind} imply that 
\[
J\sse J_{n-k,k}\sse J_{n-k+1,k-1}\sse\ldots\sse J_{n,0}\cong[n].
\]
Thus, two zigzags of the same finite length are strongly stably equivalent. On the other hand, two strongly stably equivalent finite zigzags are of the same length by \autoref{prop:strongly-equiv-necessary}.
\end{proof}

\section{Abstract tilting theory for trees with a unique branching point}
\label{sec:onebranch}

For us, a tree is a connected quiver such that the underlying unoriented graph has no cycles. In this section we prove that if $Q$ is a finite tree with precisely one branching point, i.e., vertex of valence at least three, and if $Q'$ is obtained from~$Q$ by an arbitrary reorientation of the edges, then $Q$ and $Q'$ are strongly stably equivalent. This class includes the Dynkin quivers $D_n,E_n$, the Euclidean quivers $\widetilde{D}_n,\widetilde{E}_n$, and of course also the Dynkin quivers of \S\ref{sec:An} (see also~\cite{hazeetal:ubiquity}), leading to a generalization of \cite{happel:dynkin}.

As a preparation, let us note that there is a different way of describing the equivalences obtained in \autoref{thm:An}. Let $J$ be a zig-zag of length $n$ with $k$ orientation-reversing morphisms for $0< k\leq n.$ Then we denote by $M\subseteq[n]\times[k]$ the \emph{full} subcategory given by the convex hull of $J,[n]\times\{k\},$ and $\{(n-k+1,0),(n,k-1)\}.$ In the example described in \eqref{eq:example} the category $M$ is given by
\[
\xymatrix@-.5pc{
&& (2,0)\ar[d]\ar[r]&(3,0)\ar[d]\ar@/^0.8pc/[dr]\\
&(1,1)\ar[d]\ar[r]&(2,1)\ar[d]\ar[r]&(3,1)\ar[r]\ar[d]&(4,1)\ar[d]\ar@/^0.8pc/[dr]\\
&(1,2)\ar[d]\ar[r]&(2,2)\ar[d]\ar[r]&(3,2)\ar[r]\ar[d]&(4,2)\ar[r]\ar[d]&(5,2)\ar[d]\\
(0,3)\ar[r]&(1,3)\ar[r]&(2,3)\ar[r]&(3,3)\ar[r]&(4,3)\ar[r]&(5,3)
}
\]
Let $K$ be the full subcategory spanned by $J$ and $[(n-k+1,0),(n,k-1)]$, and let~$K'$ be the full subcategory spanned by $J'=[n]\times\{k\}$ and $[(n-k+1,0),(n,k-1)]$. There are obvious fully faithful inclusions
\begin{equation}
  \vcenter{\xymatrix@-.5pc{
      J\ar[r]^-{i_1}&K\ar[r]^-{i_2}&M,&&
      J'\ar[r]^-{i_1'}&K'\ar[r]^-{i_2'}&M,
    }}\label{eq:Sdot}
\end{equation}
and one observes that $i_1\colon J\to K$ is a sieve while $i_1'\colon J'\to K'$ is a cosieve. We now consider the compositions 
\[
\D^J\stackrel{(i_1)_\ast}{\to}\D^K\stackrel{(i_2)_!}{\to}\D^M\qquad\text{and}\qquad
\D^{J'}\stackrel{(i_1')_!}{\to}\D^{K'}\stackrel{(i_2')_\ast}{\to}\D^M.
\]
The first composition amounts to right extension by zero (\autoref{lem:extbyzero}) and then forming cocartesian squares (again by a repeated application of \autoref{lem:detection}). Similarly, the second composition is left extension by zero followed by the formation of cartesian squares. These two compositions are thus fully faithful with the same essential image and hence induce an equivalence of both $\D^J$ and $\D^{J'}$ with the same subderivator of $\D^M.$

We now recycle this in the case of a tree $Q$ with a unique branching point $b\in Q$. Let $n\geq 3$ be the \textbf{valence} of $b$, i.e., the number of edges adjacent to~$b$. Moreover, let $k$ be the number of incoming edges at $b$, so that there are $n-k$ outgoing morphisms at $b$. Given such a quiver, let $l_1,\ldots, l_k$ be the lengths of the incoming branches, and let $l_{k+1},\ldots ,l_n$ be the lengths of the outgoing branches. Thus, the $i$-th branch is a zig-zag $J_i$ of length~$l_i.$ We set 
\[
\underline{l}=(l_1,\ldots,l_k;l_{k+1},\ldots,l_n),
\]
and then say that the \textbf{type} of~$Q$ is $(n,k;\underline{l})$.

We single out the coherently oriented examples of such trees. More specifically, we denote by
\[
Q_{n,k;\underline{l}}
\]
the quiver of type $(n,k;\underline{l})$ such that on each branch the edges are oriented in the same way as the unique edge of that branch which is adjacent to~$b$. Such a quiver is uniquely determined up to a relabeling and possibly a permutation of the branches. For example, the quiver $Q_{3,2;(2,3;4)}$ can be depicted by:
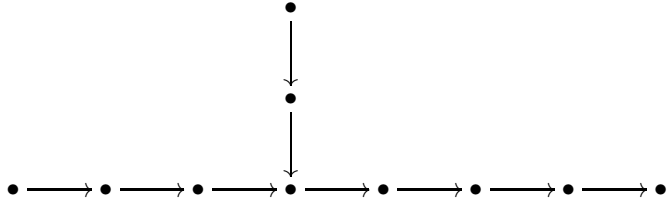
\[
\xymatrix{
&&&\bullet\ar[d]&&&&\\
&&&\bullet\ar[d]&&&&\\
\bullet\ar[r]&\bullet\ar[r]&\bullet\ar[r]&\bullet\ar[r]&\bullet\ar[r]&\bullet\ar[r]&\bullet\ar[r]&\bullet
}
\]

The following lemma tells us that we can always obtain these coherent orientations. In the proof we keep using the notation of \eqref{eq:Sdot}.

\begin{lem}\label{lem:orient}
Any tree of type $(n,k;\underline{l})$ is strongly stably equivalent to $Q_{n,k;\underline{l}}$. 
\end{lem}
\begin{proof}
Let \D be a stable derivator. We first show that we can change the orientations of the edges on a fixed branch to the one of the unique edge on that branch which is adjacent to $b$. Thus, let us consider a branch $J$ of $Q$ which --say-- `begins' at the branching point. Such a branch is a zigzag of a certain length~$l$ with `initial vertex' $q_0$. The branch~$J$ sits in the pushout square on the left in which $Q_0$ denotes the quiver obtained from~$Q$ by removing the given branch~$J$ with the exception of $b$ itself,
\[
\xymatrix{
\bbone\ar[d]_b\ar[r]^-{q_0}&J\ar[d]\ar[r]^-{i_1}&K\ar[r]^-{i_2}\ar[d]&M\ar[d]\\
Q_0\ar[r]&Q\ar[r]_-{j_1}&Q_K\ar[r]_-{j_2}&Q_M.
}
\]
Also the remaining squares are defined as pushouts with the upper horizontal morphisms as in \eqref{eq:Sdot}. It follows from construction that $j_1\colon Q\to Q_K$ is the inclusion of a sieve, and hence that $(j_1)_\ast\colon\D^Q\to\D^{Q_K}$ is right extension by zero, inducing an equivalence onto the full subderivator of $\D^{Q_K}$ spanned by diagrams satisfying the obvious vanishing conditions (\autoref{lem:extbyzero}). The functor $j_2\colon Q_K\to Q_M$ is also fully faithful, and $(j_2)_!$ hence induces an equivalence onto its essential image. One checks that a repeated application of \autoref{lem:detection} implies that $(j_2)_!$ amounts to adding cocartesian squares everywhere. In fact, the slice categories in question contain as homotopy cofinal subcategories the ones showing up in the corresponding discussion of $(i_2)_!$. Thus, $(j_2)_!(j_1)_\ast$ induces a pseudo-natural equivalence between $\D^Q$ and the full subderivator $\E$ of $\D^{Q_M}$ spanned by diagrams satisfying these vanishing conditions and making all squares bicartesian.

Let~$Q'$ be the quiver which is obtained from $Q$ by coherently reorienting the branch~$J$ to get the branch~$J'$ which now comes with an initial vertex $q_0'$. Thus, related to this quiver we have the following diagram consisting of pushout squares
\[
\xymatrix{
\bbone\ar[d]_b\ar[r]^-{q_0'}&J'\ar[d]\ar[r]^-{i_1'}&K'\ar[r]^-{i_2'}\ar[d]&M\ar[d]\\
Q_0\ar[r]&Q'\ar[r]_-{j_1'}&Q'_{K'}\ar[r]_-{j_2'}&Q_M,
}
\]
where the functors $i_1',i_2'$ are again as in \eqref{eq:Sdot}. It is easy to check that $j_1'$ is again a cosieve while $j_2'$ is fully faithful. Hence the composition $(j_2')_\ast(j_1')_!\colon\D^{Q'}\to\D^{Q_M}$
defines an equivalence onto its essential image $\E'\subseteq\D^{Q_M}$. We leave it to the reader to check that $\E'$ is spanned by the diagrams which satisfy the obvious vanishing conditions and which make all squares cartesian (use again \autoref{lem:extbyzero} and repeatedly \autoref{lem:detection}). 

Note that the two subderivators $\E,\E'\subseteq\D^{Q_M}$ coincide, so that we end up with pseudo-natural equivalences
\[
\D^Q\simeq\E=\E'\simeq\D^{Q'},
\]
showing that $Q$ and $Q'$ are strongly stably equivalent. Similar arguments show that if a branch which `ends' at the branching point is coherently reoriented, then we end up with a strongly stably equivalent quiver. Finally, an induction over the number of branches shows that $Q$ and $Q_{n,k;\underline{l}}$ are strongly stably equivalent.
\end{proof}

\begin{lem}\label{lem:induct}
Let $(n,k;\underline{l})$ be a type of a tree with precisely one branching point and write $\underline{l}$ as $\underline{l}=(l_1,\ldots,l_k;l_{k+1},\ldots,l_n).$ If $\underline{l}'=(l_1,\ldots,l_{k-1};l_k,\ldots,l_n)$ then the trees $Q_{n,k;\underline{l}}$ and $Q_{n,k-1;\underline{l}'}$ are strongly stably equivalent.
\end{lem}
\begin{proof}
The strategy of the proof of \autoref{lem:orient} can be easily adapted to this case, and we leave the details to the reader.
\end{proof}

\begin{thm}\label{thm:tiltone}
Let $Q$ and $Q'$ be finite trees with a unique branching point. Then $Q$ and $Q'$ are strongly stably equivalent if and only if~$Q'$ can be obtained from $Q$ by reorienting some edges.
\end{thm}

\begin{proof}
By \autoref{lem:orient} it follows that $Q\sse Q_{n,k;\underline{l}}$ and $Q'\sse Q_{n',k';\underline{l'}}$. If $Q$ and $Q'$ are quivers on the same underlying unoriented graph, then $n$ and $n'$ have to agree and it is easy to see that an inductive application of \autoref{lem:induct} implies $Q_{n,k;\underline{l}}\sse Q_{n',k';\underline{l'}}$. Thus putting this together we obtain
\[
Q\sse Q_{n,k;\underline{l}}\sse Q_{n',k';\underline{l'}}\sse Q'.
\]
If on the other hand $Q \sse Q'$, then $Q$ and $Q'$ must have the same underlying unoriented graph by \autoref{prop:strongly-equiv-necessary}.
\end{proof}

A \textbf{forest} is a disjoint union of finite trees.

\begin{cor}\label{cor:forestone}
Let $F$ be a forest such that each tree in $F$ has at most one branching point. If $F'$ is obtained from $F$ by reorienting some edges, then $F$ and $F'$ are strongly stably equivalent.
\end{cor}
\begin{proof}
This is immediate from \autoref{thm:tiltone}, \autoref{thm:An}, and \autoref{lem:closure}.
\end{proof}

The case of trees with more than one branching points needs a different approach, since the strategy based on the $S_\bullet$-construction breaks down in that case. We will come back to this in \cite{gst:tree}.

\section{Cartesian and strongly cartesian $n$-cubes}
\label{sec:cubes}

Both for our results in \S\S\ref{sec:square}-\ref{sec:May} and also for the reflection functors for arbitrary trees in \cite{gst:tree} and more general shapes \cite{gst:acyclic} we need some facts about (strongly) cartesian $n$-cubes in derivators. This section provides the necessary background. The details are more subtle than for the corresponding results about squares in \cite{groth:ptstab}, and we hence give many details. Some of the ideas are inspired by corresponding situations in the context of topological spaces as discussed by Goodwillie in \cite{goodwillie:II}. \emph{We use the convention that an $n$-cube and any subcube is at least of dimension 2.}

Associated to the $n$-cube $[1]^n$ there are quite some useful full subcategories. For ease of notation, we pass tacitly back and forth between the description of the $n$-cube as $[1]^n$ and as the power set of $\{1,\ldots,n\}$. Under this identification we choose $1=(1,1,...,1)$ to correspond to $\{1,\ldots,n\}$. Any subcube of $[1]^n$ has a minimal element $m'$ and a maximal element $m''$, and the corresponding subcube
\[
[m'_1,m''_1]\times\ldots\times[m'_n,m''_n]
\]
will be denoted by $[m',m'']$. If we denote the cardinality of $m\in [1]^n$ by
\[
d(m)=m_1+\ldots +m_n,
\] 
then the dimension of the cube $[m',m'']$ is given by $d(m'')-d(m')$. In particular, we obtain a canonical functor
\begin{equation}\label{eq:j}
j_{[m',m'']}\colon[1]^{d(m'')-d(m')}\to [1]^n,
\end{equation}
which is an isomorphism onto its image $[m',m'']$.

We will also have uses for the following full subcategories of~$[1]^n$.

\begin{defn}
\begin{enumerate}
\item We denote by $i_{\leq k}\colon [1]^n_{\leq k}\to [1]^n$ the full subcategory given by all subsets of cardinality at most $k$ for $0\leq k\leq n$.
\item There are similar full subcategories $i_{<k}\colon [1]^n_{<k}\to [1]^n,i_{\geq k}\colon [1]^n_{\geq k}\to [1]^n,$ $i_{>k}\colon [1]^n_{>k}\to [1]^n$, and $i_{=k}\colon [1]^n_{=k}\to [1]^n$ for $0\leq k\leq n$.
\end{enumerate}
\end{defn}

\begin{defn}
\begin{enumerate}
\item An $n$-cube in \D is \textbf{cartesian} if it lies in the essential image of $(i_{\geq 1})_\ast\colon\D([1]^n_{\geq 1})\to\D([1]^n)$.
\item  An $n$-cube in \D is \textbf{strongly cartesian} if it lies in the essential image of $(i_{\geq n-1})_\ast\colon\D([1]^n_{\geq n-1})\to\D([1]^n)$.
\end{enumerate}
\end{defn}

For squares these two notions coincide, but from dimension $n=3$ on being strongly cartesian is a stricter assumption (see \autoref{thm:subcubes}). Since the functors $i_{\geq_k}$ are fully faithful, the same is true for $(i_{\geq k})_\ast$, and we can consequently decide whether an $n$-cube is (strongly) cartesian by considering the corresponding adjunction unit $\eta\colon\id\to(i_{\geq_k})_\ast(i_{\geq_k})^\ast$. 

The following proof illustrates how convenient the calculus of homotopy exact squares in $\cCat$ is. We only have to focus on the shapes of diagrams, the rest is taken care of by this calculus.

\begin{thm}\label{thm:subcubes}
An $n$-cube $X\in\D([1]^n)$, $n\geq 2$, is strongly cartesian if and only if all subcubes are cartesian.
\end{thm}
\begin{proof}
Let $X$ be strongly cartesian, and let $[m',m'']$ be a subcube of dimension $c_1=d(m'')-d(m') \ge 2$. We want to show that $(j_{[m',m'']})^\ast(X)\in\D([1]^{c_1})$ is cartesian. Since $X$ is strongly cartesian, it lies in the essential image of the right Kan extension along $i_{\geq n-1}\colon [1]^n_{\geq n-1}\to [1]^n.$ This functor factors over the following intermediate \emph{full} subcategories of $[1]^n$.
\begin{enumerate}
\item $A_1$ is spanned by $[1]^n_{\geq n-1}$ and the cube $[m'',1]$, where $1$ is the terminal object of the cube~$[1]^n$. Let us write $c_2=n-d(m'')$ for its dimension.
\item $A_2$ is obtained from $A_1$ by adding the objects in the cube $[m',m'']$ with the exception of $m'.$
\item $A_3$ is obtained from $A_2$ by adding the object $m'$.
\end{enumerate}
Using this notation, the inclusion $i_{\geq n-1}\colon[1]^n_{\geq n-1}\to[1]^n$ factors as
\[
[1]^n_{\geq n-1}\stackrel{j_1}{\to} A_1\stackrel{j_2}{\to}A_2 \stackrel{j_3}{\to} A_3 \stackrel{j_4}{\to} [1]^n.
\]
In order to calculate the value of the right Kan extension along $j_3$ at $m'$, (Der4) tells us to consider the category $(m'/j_3)$. One checks that the projection $(m'/j_3) \to A_2$ to the target induces an isomorphism of categories
\[
(m'/j_3) \cong [m',1] \cap A_2,
\] 
where we consider $[m',1] \cap A_2$ as a subposet of $[1]^n$. Now we use the fact that $[1]^n$ is actually a lattice, where the meet $\wedge$ and join $\vee$ are defined as the componentwise minimum and maximum, respectively. It is obvious that the inclusion $\ell\colon [m',m''] \to ([m',1] \cap A_2)\cup \{m'\}$ has a right adjoint $r$ defined by $r(y) = y \wedge m''$. We claim that this adjunction restricts to an adjunction if we remove $m'$ on both sides, 
\[
(\ell',r')\colon [m',m''] - \{m'\} \rightleftarrows [m',1] \cap A_2.
\]
In fact, it suffices to check that $r$ restricts accordingly. In other words, we must check that $y \wedge m'' > m'$ for each $y \in [1]^n_{\geq n-1} \cap [m',1]$, since all other elements of $([m',1] \cap A_2) - [m',m'']$ are greater than $m''$. To this end, we have $|y\wedge m''| \ge |m''|-1 \ge |m'| + c_1 - 1 > |m'|$ since $|y| \ge n-1$ and $c_1\ge 2$. Thus, $y \wedge m'' \ne m'$ and the claim is proved.

These observations yield a homotopy final functor $[m',m'']-\{m'\}\to (m'/j_3)$ (see \autoref{egs:htpy}), which, esentially by precomposition with \eqref{eq:j}, gives us a homotopy final functor $i\colon (0/[1]^{c_1}-\{0\})\to (m'/j_3)$. This motivates us to consider the following pasting 
\begin{equation}
\vcenter{\xymatrix{
(0/[1]^{c_1}-\{0\})\ar[r]^-i\ar[d]_-\pi&(m'/j_3)\ar[r]^-p\ar[d]_-\pi&A_2\ar[d]^-{j_3}\ar[r]^-{j_3}&A_3\ar[d]^-=\\
\bbone\ar[r]_-=\ar[r]&\bbone\ar[r]_-{m'}&A_3\ultwocell\omit{}\ar[r]_-=&A_3.
}
}
\label{eq:paste1}
\end{equation}
The compatibility of mates with pasting implies that the mate associated to this pasting evaluated on an object $Y\in\D(A_3)$ is the composition
\[
Y_{m'}\stackrel{\eta}{\to}(j_3)_\ast(j_3)^\ast(Y)_{m'}\stackrel{\cong}{\to} \lim_{(m'/j_3)}p^\ast(j_3)^\ast(Y)\stackrel{\cong}{\to}\lim_{(0/[1]^{c_1}-\{0\})}i^\ast p^\ast(j_3)^\ast(Y),
\]
in which the last two maps are isomorphisms by (Der4) and the homotopy finality of~$i$. Thus, this mate is an isomorphism on $Y$ if and only if $Y$ lies in the essential image of $(j_3)_\ast$ (see \cite[Lemma~1.21]{groth:ptstab} for the fact that it is enough to control the unit at~$m'$). One checks that the pasting \eqref{eq:paste1} can be rewritten as
\begin{equation}
\vcenter{\xymatrix{
(0/[1]^{c_1}-\{0\})\ar[r]^-p\ar[d]_-\pi&[1]^{c_1}-\{0\}\ar[r]^-k\ar[d]_-k&
[1]^{c_1}\ar[r]^-{j_{[m',m'']}}\ar[d]^-=&A_3\ar[d]^-=\\
\bbone\ar[r]_0&[1]^{c_1}\ar[r]_-=\ultwocell\omit{}&[1]^{c_1}\ar[r]_-{j_{[m',m'']}}&A_3.
}}
\end{equation}
Using once more the compatibility of mates with pasting, we conclude that the canonical mate can hence also be written as
\[
(j_{[m',m'']})^\ast(Y)_0\stackrel{\eta}{\to} \big(k_\ast k^\ast (j_{[m',m'']})^\ast(Y)\big)_0\stackrel{\cong}{\to} \lim_{(0/[1]^{c_1}-\{0\})}p^\ast k^\ast (j_{[m',m'']})^\ast(Y),
\]
where the second map is an isomorphism by (Der4). Using \cite[Lemma~1.21]{groth:ptstab} once more, we hence conclude that $Y\in\D(A_3)$ lies in the essential image of $(j_3)_\ast$ if and only if $(j_{[m',m'']})^\ast Y$ is cartesian. If we now consider the strongly cartesian~$X$, then applying this to $Y=(j_4)^\ast(X)$ shows that all cubes in~$X$ have to be cartesian.

Let us conversely assume that all subcubes of $X\in\D([1]^n)$ are cartesian. The functor $[1]^n_{\geq n-1}\to[1]^n$ factors as
\[
[1]^n_{\geq n-1}\stackrel{i_2}{\to}[1]^n_{\geq n-2}\stackrel{i_3}{\to}[1]^n_{\geq n-3}\to\ldots\to[1]^n_{\geq 2}\stackrel{i_{n-1}}{\to}[1]^n_{\geq 1}\stackrel{i_n}{\to}[1]^n_{\geq 0}=[1]^n.
\]
It follows that $X$ is already strongly cartesian if for each $Y=X\!\!\mid_{[1]^n_{\geq n-k}}$ the adjunction unit $\eta\colon Y\to (i_k)_\ast i_k^\ast(Y)$ is an isomorphism. Since $i_k$ is fully faithful, \cite[Lemma~1.21]{groth:ptstab} implies that it is enough to verify that the component $\eta_m$ is an isomorphism for all objects $m\in [1]^n_{=n-k}$. To re-express this differently, let us consider the pasting
\begin{equation}
\vcenter{\xymatrix{
(m/[m,1]-\{m\})\ar[r]^-f_-\cong\ar[d]_-\pi&(m/[1]^n_{\geq n-(k-1)})\ar[r]^-p\ar[d]_-\pi&[1]^n_{\geq n-(k-1)}\ar[r]^-{i_k}\ar[d]_-{i_k}&
[1]^n_{\geq n-k}\ar[d]\\
\bbone\ar[r]&\bbone\ar[r]_m&[1]^n_{\geq n-k}\ar[r]\ultwocell\omit{}&[1]^n_{\geq n-k}.
}}
\end{equation}
Using (Der4) and again the compatibility of mates with pasting, we deduce that $\eta_m$ is an isomorphism if and only if the mate associated to this pasting is an isomorphism when evaluated on~$Y$. But it is easy to see that this is the case if and only if the mate associated to the pasting
\begin{equation}
\vcenter{\xymatrix{
(m/[m,1]-\{m\})\ar[r]\ar[d]&[m,1]-\{m\}\ar[r]\ar[d]&[m,1]\ar[d]\ar[r]&[1]^n\ar[d]\\
\bbone\ar[r]_-m&[m,1]\ar[r]\ultwocell\omit{}&[m,1]\ar[r]&[1]^n
}}
\end{equation}
is an isomorphism when applied to~$X$. Using once more the compatibility of mates with pasting, axiom (Der4), and again \cite[Lemma~1.21]{groth:ptstab}, this in turn is equivalent to the statement that the cube~$(j_{[m,1]})^\ast(X)$ is cartesian which is true by assumption.
\end{proof}

We can consider an $(n+1)$-cube $X\in\D([1]^n\times[1])$ as a morphism between two $n$-cubes. Our next aim is to show that a morphism between cartesian $n$-cubes is itself cartesian (see \autoref{thm:morphism}). The proof is an adaptation of the strategy of Goodwillie in \cite{goodwillie:II} to this more abstract setting. We begin by establishing a few convenient lemmas which are also of independent interest.

Given a small category~$A$, let us recall that $A^\lhd$ denotes the \emph{cone} on~$A$ which is obtained from~$A$ by adjoining a new initial object $-\infty.$ This category comes with a natural inclusion $i_A\colon A\to A^\lhd.$ If we iterate this construction then there is the following result. Let us denote the two additional objects in $(A^\lhd)^\lhd$ by $-\infty$ and $-\infty -1,$ so that there is a map $-\infty-1\to -\infty.$ The inclusion $i_A^\lhd\colon A^\lhd\to (A^\lhd)^\lhd$ does not hit $-\infty$ while the inclusion $i_{A^\lhd}\colon A^\lhd\to (A^\lhd)^\lhd$ avoids $-\infty-1$.

\begin{lem}\label{lem:excone}
Let $A$ be a small category. Then the square
\[
\xymatrix{
A\ar[r]^-{i_A}\ar[d]_{i_A}&A^\lhd\ar[d]^-{i_A^\lhd}\\
A^\lhd\ar[r]_-{i_{A^\lhd}}&(A^\lhd)^\lhd\ultwocell\omit{\id}
}
\]
is homotopy exact.
\end{lem}
\begin{proof}
By (Der2) we only have to show that each component of the canonical mate $\psi\colon i_{A^\lhd}^\ast (i_A^\lhd)_\ast \to (i_A)_\ast i_A^\ast$ is an isomorphism. As a first case, let $a\in A\subseteq A^\lhd$. In that case consider the pasting
\[
\xymatrix{
\bbone\ar[r]^-i\ar[d]&(a/i_A)\ar[r]^-p\ar[d]_-\pi&A\ar[r]^-{i_A}\ar[d]_-{i_A}&A^\lhd\ar[d]^-{i_A^\lhd}\\
\bbone\ar[r]&\bbone\ar[r]_-a&A^\lhd\ar[r]_-{i_{A^\lhd}}\ultwocell\omit{}&(A^\lhd)^\lhd
}
\]
in which the upper left horizontal morphism classifies the initial object. The compatibility of mates with respect to pasting tells us that the mate associated to this diagram factors as
\[
a^\ast (i_{A^\lhd})^\ast (i_A^\lhd)_\ast\stackrel{\psi_a}{\to} a^\ast (i_A)_\ast (i_A)^\ast\to\pi_\ast p^\ast (i_A)^\ast\to i^\ast p^\ast (i_A)^\ast.
\]
But the first undecorated arrow is an isomorphism by (Der4), while the second undecorated morphism is an isomorphism because $i$ is an initial object, and hence a homotopy final functor (see \autoref{egs:htpy}). It follows that $\psi_a$ is an isomorphism if and only if the mate of the total pasting is an isomorphism. But this pasting factors further as
\[
\xymatrix{
\bbone\ar[r]^-{i'}\ar[d]&(a/i_A^\lhd)\ar[r]^-{p'}\ar[d]_-\pi&A^\lhd\ar[d]^-{i_A^\lhd}\\
\bbone\ar[r]&\bbone\ar[r]_-a&(A^\lhd)^\lhd,\ultwocell\omit{}
}
\]
in which $i'$ also classifies an initial object. Thus, $\psi_a$ is an isomorphism if and only if
\[
a^\ast(i_A^\lhd)_\ast \to \pi_\ast p'^\ast\to i'^\ast p'^\ast
\]
is an isomorphism. But both functors in this composition are isomorphisms by (Der4) and again the homotopy finality of initial objects (\autoref{egs:htpy}).

For the remaining component at $-\infty\in A^\lhd$ we consider the pasting
\[
\xymatrix{
A\ar[r]^-\cong\ar[d]_\pi&(-\infty/i_A)\ar[r]^-p\ar[d]_-\pi&A\ar[r]^-{i_A}\ar[d]_-{i_A}&A^\lhd\ar[d]^-{i_A^\lhd}\\
\bbone\ar[r]&\bbone\ar[r]_-{-\infty}&A^\lhd\ar[r]_-{i_{A^\lhd}}\ultwocell\omit{}&(A^\lhd)^\lhd.
}
\]
The compatibility with pasting together with (Der4) allows us to conclude that $\psi_{-\infty}$ is an isomorphism if and only if the mate of this pasting is an isomorphism. However, since this pasting can be rewritten as 
\[
\xymatrix{
A\ar[r]^-\cong\ar[d]_-\pi&(-\infty/i_A^\lhd)\ar[r]^-{p'}\ar[d]_-\pi&A^\lhd\ar[d]^-{i_A^\lhd}\\
\bbone\ar[r]&\bbone\ar[r]_-{-\infty}&(A^\lhd)^\lhd\ultwocell\omit{}
}
\]
we can conclude the proof by applying once more (Der4) and the compatibility of mates with respect to pasting.
\end{proof}

With the aid of the inclusion $i_A^\lhd\colon A^\lhd\to(A^\lhd)^\lhd$ we can give the following characterization of \emph{limiting cones}, i.e., of objects in the essential image of the right Kan extension $(i_A)_\ast\colon\D(A)\to\D(A^\lhd)$. 

\begin{lem}\label{lem:limiting}
An object $X\in\D(A^\lhd)$ is a limiting cone if and only if the map $(i_A^\lhd)_\ast(X)_{-\infty-1}\to(i_A^\lhd)_\ast(X)_{-\infty}$ is an isomorphism in $\D(\bbone)$.
\end{lem}
\begin{proof}
The functor $i_A\colon A\to A^\lhd$ is fully faithful, and hence so is $(i_A)_\ast$. Thus a diagram $X\in\D(A^\lhd)$ is a limiting cone if and only if the adjunction unit $\eta\colon X\to (i_A)_\ast(i_A)^\ast(X)$ is an isomorphism. It follows from \cite[Lemma~1.26]{groth:ptstab} that this is the case if and only if the component $\eta_{-\infty}$ is an isomorphism when evaluated on~$X$. The lemma follows immediately from the more precise statement that there is a commutative diagram
\begin{equation}
\vcenter{
\xymatrix{
(i_A^\lhd)_\ast(X)_{-\infty-1}\ar[r]\ar[d]_-\cong&(i_A^\lhd)_\ast(X)_{-\infty}\ar[d]^-\cong\\
\lim_{(-\infty-1/i_A^\lhd)}p^\ast(X)\ar[dd]_-\cong&\lim_{(-\infty/i_A^\lhd)}p^\ast(X)\ar[d]^-\cong\\
&\lim_{(-\infty/i_A)}p^\ast (i_A)^\ast(X)\\
X_{-\infty}\ar[r]_-{\eta_{-\infty}}&(i_A)_\ast (i_A)^\ast(X)_{-\infty},\ar[u]_-\cong
}
}
\label{eq:limiting}
\end{equation}
in which the various functors denoted by $p$ are instances of projections of certain comma categories. The four morphisms in the counterclockwise direction are obtained by the calculus of mates applied to the pasting
\[
\xymatrix{
(-\infty/i_A)\ar[r]^-p\ar[d]_-\pi&A\ar[rr]^-{i_A}\ar[d]^-{i_A}&&A^\lhd\ar[d]^-=\\
\bbone\ar[d]\ar[r]_-{-\infty}&A^\lhd\ar[r]_-\cong\ultwocell\omit{}&(-\infty-1/i_A^\lhd)\ar[r]_-p\ar[d]_-\pi&A^\lhd\ar[d]^-{i_A^\lhd}\\
\bbone\ar[r]&\bbone\ar[r]&\bbone\ar[r]_-{-\infty-1}&(A^\lhd)^\lhd.\ultwocell\omit{}
}
\]
Three of the associated mates are isomorphisms (by (Der4) in two cases and once more by the homotopy finality of initial objects; see \autoref{egs:htpy}), while the remaining one is the adjunction unit $\eta_{-\infty}$. The three morphisms in the clockwise direction in \eqref{eq:limiting} are obtained from the pasting
\[
\xymatrix{
(-\infty/i_A)\ar[r]^-\cong\ar[dd]_-\pi&(-\infty/i_A^\lhd)\ar[d]_-\pi\ar[r]^-p&A^\lhd\ar[d]^-{i_A^\lhd}\\
&\bbone\ar[r]_-{-\infty}\ar[d]&(A^\lhd)^\lhd\ar[d]^-=\ultwocell\omit{}\\
\bbone\ar[r]&\bbone\ar[r]_-{-\infty-1}&(A^\lhd)^\lhd,\ultwocell\omit{}
}
\]
where the natural transformation in the bottom right square is given by the morphism $-\infty-1\to -\infty$. Two of these three morphism are easily seen to be isomorphism, and the remaining one is the structure map $(i_A^\lhd)_\ast(X)_{-\infty-1}\to (i_A^\lhd)_\ast(X)_{-\infty}$ showing up in the statement of this proposition. It is easy to check that these two pastings define the same natural transformation, so that the compatibility of mates with pasting concludes the proof. 
\end{proof}

Let $A$ be a small category and let us consider $i_A^\lhd\times\id\colon A^\lhd\times[1]\to (A^\lhd)^\lhd\times[1].$ We denote by $B$ the category obtained from $(A^\lhd)^\lhd\times[1]$ by adding a new object $l$ with a unique morphism from $(-\infty-1,0)\to l$ and a unique map from~$l$ to any object different from $(-\infty-1,0)$. Thus, precisely five objects of $B$ do not lie in $A\times[1]\subseteq B$, and the full subcategory spanned by these five objects looks like 
\begin{equation}
\vcenter{
\xymatrix{
(-\infty-1,0\ar[rd])&&\\
&l\ar[r]\ar[d]&(-\infty-1,1)\ar[d]\\
&(-\infty,0)\ar[r]&(-\infty,1).
}
}\label{eq:detect}
\end{equation}
The square $l,(-\infty,0),(-\infty-1,1),(-\infty,1)$ defines a functor $j\colon\square\to B.$ The category $B$ comes with a further functor $i\colon A^\lhd\times[1]\to B$ which is the composition of $i_A^\lhd\times\id\colon A^\lhd\times[1]\to (A^\lhd)^\lhd\times[1]$ followed by the obvious inclusion to $B.$ Note that this functor is isomorphic to the inclusion of the subcategory of $B$ obtained by removing the objects $l,(-\infty,0),$ and also $(-\infty,1)$.

\begin{lem}\label{lem:detect}
If $X\in\D(B)$ lies in the essential image of $i_\ast\colon\D(A^\lhd\times[1])\to\D(B),$ then the square $j^\ast(X)\in\D(\square)$ is cartesian.
\end{lem}
\begin{proof}
This is an immediate application of the detection lemma (\autoref{lem:detection}). In fact, since $l$ does not lie in the image of $i$ and since $B-\{l\}=(A^\lhd)^\lhd\times[1]$, we only have to show that the functor $\lrcorner\to(l/(A^\lhd)^\lhd\times[1])$ given by \eqref{eq:detect} is a left adjoint. But a right adjoint $r$ is given by $r(a,k)=r(-\infty,k)=(-\infty,k),a\in A,k\in[1],$ and $r(-\infty-1,1)=(-\infty-1,1).$
\end{proof}

The case of interest to us is $A=[1]^n_{\geq 1}$. In this case, $B$ can be identified with $(([1]^n)^\lhd\times[1])^\lhd,$ but let us be more precise about this category since some details will be used in the following proof. First, in the category $([1]^n)^\lhd$, let us denote the cone point by $-\infty-1$ and its successor by $-\infty$. Then $B$ is obtained from $([1]^n)^\lhd\times[1]$ by adjoining a new object $l$ and morphisms $(-\infty-1,0)\to l\to(-\infty,0)$ and $l\to (-\infty-1,1).$ This completes the description of the poset $B$. Note that there is again a fully faithful inclusion functor $i\colon [1]^n\times[1]\to B$ which does not hit the objects $l,(-\infty,0),(-\infty,1)$.

\begin{thm}\label{thm:morphism}
Let $X\in\D([1]^n\times [1]), n\geq 2,$ be such that the $n$-cube $X_1$ is cartesian. The $n$-cube $X_0$ is cartesian if and only if the $(n+1)$-cube $X$ is cartesian.
\end{thm}
\begin{proof}
Let $i\colon [1]^n\times[1]\to B$ be the functor described above and let us apply the fully faithful functor $i_\ast$ to our $(n+1)$-cube $X$ in order to obtain $i_\ast(X)\in\D(B).$ The fully faithfulness of $i_\ast$ implies that $i_\ast(X)$ restricts back to the $n$-cubes $X_0$ and $X_1$. Let us now consider the following part of $i_\ast(X),$
\begin{equation}
\vcenter{
\xymatrix@-1pc{
X_{(-\infty-1,0)}\ar[dr]&&\\
&X_l\ar[d]\ar[r]&X_{(-\infty-1,1)}\ar[d]\\
&X_{(-\infty,0)}\ar[r]&X_{(-\infty,1)},
}
}
\end{equation}
which is a coherent diagram in \D of shape $\square^\lhd.$ Since the $n$-cube $X_1$ is cartesian,  \autoref{lem:limiting} implies that the map $X_{(-\infty-1,1)}\to X_{(-\infty,1)}$ is an isomorphism. An application of \autoref{lem:detect} yields that the above square is cartesian, hence also that the map $X_l\to X_{(-\infty,0)}$ is an isomorphism (\cite[Prop.~3.13]{groth:ptstab}), and we deduce that $X_{(-\infty-1,0)}\to X_l$ is an isomorphism if and only if $X_{(-\infty-1,0)}\to X_{(-\infty,0)}$ is an isomorphism. But \autoref{lem:limiting} tells us that $X_{(-\infty-1,0)}\to X_l$ is an isomorphism if and only if $X$ is cartesian and it also implies that $X_{(-\infty-1,0)}\to X_{(-\infty,0)}$ is an isomorphism if and only if $X_0$ is cartesian.
\end{proof}

With this preparation we can now deduce some interesting corollaries.

\begin{cor}\label{cor:subsquares}
An $n$-cube, $n\geq 2$, in a derivator is strongly cartesian if and only if all \emph{subsquares} are cartesian.
\end{cor}
\begin{proof}
We know already that an $n$-cube is strongly cartesian if and only if all subcubes are cartesian (\autoref{thm:subcubes}). But by \autoref{thm:morphism} this is equivalent to the assumption that all subsquares are cartesian.
\end{proof}

We also obtain the following additional characterization of stability.

\begin{cor}\label{cor:stable}
The following are equivalent for a pointed derivator \D. 
\begin{enumerate}
\item The adjunction $(\Sigma,\Omega)\colon\D(\bbone)\to\D(\bbone)$ is an equivalence.
\item The adjunction $(\cof,\fib)\colon\D([1])\to\D([1])$ is an equivalence.
\item The derivator \D is stable, i.e., a square is cartesian if and only if it is cocartesian.
\item An $n$-cube in \D, $n\geq 2,$ is strongly cartesian if and only if it is strongly cocartesian. 
\end{enumerate}
\end{cor}
\begin{proof}
The equivalence of the first three characterizations was already established in \cite{gps:mayer}, while the equivalence of (iii) and (iv) is immediate from \autoref{cor:subsquares}.
\end{proof}

An $n$-cube satisfying the equivalent conditions of being strongly cartesian and strongly cocartesian is called \textbf{strongly bicartesian}. 

There is the following immediate corollary about $[1]^n_{\geq n-1}$, the sink of valence~$n$, and $[1]^n_{\leq 1}$, the source of valence~$n$.

\begin{cor}\label{cor:tiltsourcesink}
The source of valence~$n$ and the sink of valence~$n$ are strongly stably equivalent.
\end{cor}
\begin{proof}
Let \D be a stable derivator. The fully faithful $(i_{\geq n-1})_\ast\colon\D^{[1]^n_{\geq n-1}}\to\D^{[1]^n}$ induces an equivalence onto the full subderivator of $\D^{[1]^n}$ spanned by the strongly cartesian $n$-cubes, while the left Kan extension $(i_{\leq 1})_!\colon\D^{[1]^n_{\leq 1}}\to\D^{[1]^n}$ induces an equivalence onto the full subderivator spanned by the strongly cocartesian $n$-cubes. By \autoref{cor:stable} these two subderivators are the same, and we deduce $[1]^n_{\leq 1}\sse[1]^n_{\geq n-1}$.
\end{proof}

There are the following composition and cancellation properties for (strongly) cartesian $n$-cubes. We follow the standard convention from simplicial homotopy theory in that $d^i\colon[n-1]\to [n],0\leq i\leq n,$ denotes the $i$-th \emph{coface map}, i.e., the unique monotone injective function which does not hit $i$. Moreover, the associated restriction functor is denoted by $d_i\colon\D^{[n]}\to\D^{[n-1]}$. Similarly, $s^i\colon[n+1]\to[n]$ for $0\leq i\leq n$ is the $i$-th \emph{codegeneracy map}, i.e., the unique monotone surjection hitting $i$ twice, while the associated restriction functor is denoted by $s_i$.

\begin{prop}\label{prop:compcanc}
Let \D be a derivator and let $X\in\D([1]^n\times [2])$ be such that $d_0(X)$ is (strongly) cartesian. Then $d_1(X)$ is (strongly) cartesian if and only if $d_2(X)$ is (strongly) cartesian.
\end{prop}
\begin{proof}
Let $j\colon A\to [1]^n\times[2]$ be the full subcategory spanned by all objects except $(\emptyset,0)$ and $(\emptyset,1)$. \autoref{lem:detectionplus} implies that the $n$-cubes $d_2j_\ast j^\ast(X)$ and $d_1j_\ast j^\ast(X)$ are cartesian. Thus, in order to conclude the proof it is enough to show that if $d_0(X)$ is cartesian, then $\eta\colon X\to j_\ast j^\ast(X)$ is an isomorphism under the additional assumption that either of $d_1(X)$ and $d_2(X)$ would be cartesian. But since $j$ is fully faithful, it suffices to show that the components of $\eta$ at $(\emptyset,0)$ and $(\emptyset,1)$ are isomorphisms. Using the typical arguments, one observes that the component at $(\emptyset,1)$ is an isomorphism if and only if the canonical mate associated to
\[
\xymatrix{
[1]^n_{\geq 1}\ar[r]^-\cong\ar[d]&((\emptyset,1)/j)\ar[r]^-p\ar[d]&A\ar[r]^-j\ar[d]_-j&[1]^n\times[2]\ar[d]^-=\\
\bbone\ar[r]&\bbone\ar[r]_-{(\emptyset,1)}&[1]^n\times[2]\ar[r]_-=\ultwocell\omit{}&[1]^n\times[2]
}
\] 
is an isomorphism. But this is the case since we are considering the mate expressing that $d_0(X)$ is cartesian.

It remains to show that the component of $\eta$ at $(\emptyset,0)$ is an isomorphism if $d_1(X)$ or $d_2(X)$ is cartesian. For the case of $d_1(X)$ it suffices to consider
\[
\xymatrix{
[1]^n_{\geq 1}\ar[d]\ar[r]^-{d^1}&A\ar[r]^-\cong\ar[d]&((\emptyset,0)/j)\ar[r]^-p\ar[d]&A\ar[r]^-j\ar[d]_-j&[1]^n\times[2]\ar[d]^-=\\
\bbone\ar[r]&\bbone\ar[r]&\bbone\ar[r]_-{(\emptyset,0)}&[1]^n\times[2]\ar[r]_-=\ultwocell\omit{}&[1]^n\times[2]
}
\] 
and to observe that $d^1\colon[1]\to[2]$ is final since it has a right adjoint given by the codegeneracy $s^0\colon[2]\to[1]$. As for the case of $d_2(X)$ we consider the factorization of $j$ through $j'\colon A'=[1]^n \times [2]-\{(\emptyset,0)\}\to[1]^n\times[2]$ and observe that it is sufficient to show that the canonical mate associated to
\[
\xymatrix{
[1]^n_{\geq 1}\ar[d]\ar[r]^-{d^2}&A'\ar[r]^-\cong\ar[d]&((\emptyset,0)/j')\ar[r]^-p\ar[d]&A'\ar[r]^-{j'}\ar[d]_-{j'}&[1]^n\times[2]\ar[d]^-=\\
\bbone\ar[r]&\bbone\ar[r]&\bbone\ar[r]_-{(\emptyset,0)}&[1]^n\times[2]\ar[r]_-=\ultwocell\omit{}&[1]^n\times[2]
}
\] 
is an isomorphism. Here we used the finality of $d^2\colon[1]\to[2].$ But this is the case if and only if $d_2(X)$ is cartesian.

The case of strongly cartesian $n$-cubes is now easily established. The cube $d_0(X)$ is strongly cartesian if and only if all subcubes in it are cartesian (\autoref{cor:subsquares}). In order to show that $d_1(X)$ is strongly cartesian if and only if $d_2(X)$ is strongly cartesian it suffices to compare the subsquares of these two cubes. The case that the subsquare under consideration is contained neither in $d_0(X)$ nor in $d_2(X)$ is taken care of by the first part of the proof while the other cases are trivial.
\end{proof}

\begin{cor}\label{cor:2-out-of-3}
Let \D be a stable derivator and let $X\in\D([1]^n\times[2])$. If two of the $n$-cubes $d_0(X),d_1(X),d_2(X)$ are strongly bicartesian then the same is true for the third one.
\end{cor}
\begin{proof}
This is immediate from \autoref{prop:compcanc} and \autoref{cor:stable}.
\end{proof}

\begin{rmk}
Since the inclusions $i_{\geq k}\colon[1]^n_{\geq k}\to[1]^n$ are fully faithful we can use \cite[Corollary~2.6]{groth:ptstab} to obtain \emph{parametrized} versions of the results of this sections. Moreover, using opposite derivators it is immediate that also the dual statements are correct.
\end{rmk}

\section{Abstract representation theory of the commutative square}
\label{sec:square}

In this section we use the techniques established so far in order to perform a few first steps of something that could be called \emph{abstract representation theory} of the square, the trivalent source, and the trivalent sink. First, we show these categories to be strongly stably equivalent (\autoref{thm:tiltsquare}). More importantly, in~\S\ref{subsec:symm} we construct and study some autoequivalences of related stable derivators (for example a cube root of the suspension). These alternative descriptions and additional symmetries are of interest to us for the following reasons.
\begin{enumerate}
\item They allow us to introduce extensions and refinements of the more classical \emph{Serre functors} and \emph{Auslander--Reiten translations} defined for derived categories. As an application of these functors we establish the \emph{relative fractionally Calabi--Yau property} of the trivalent source (see \S\ref{subsec:fracCY}).
\item By means of the symmetries we obtain variants of the classical Auslander--Reiten quivers associated to abstract representations of the trivalent source in an arbitrary stable derivator (see \S\ref{subsec:AR}).
\item This study implies that generalized versions of some of May's compatibility axioms \cite{may:traces} on triangulations and monoidal structures follow
from statements that are a formal consequence of stability alone and do not depend on monoidal structure. This observation is inspired by work of Keller--Neeman \cite{kellerneeman:may}, and is spelled out in \S\ref{sec:May}. 
\end{enumerate}

Note that contrary to the other examples considered in this paper so far, the commutative square is not a quiver since it is not a free category. Instead it is a quiver with a commutativity relation. It seems that 
a more systematic study of quivers with additivity relations will need a fully-fledged theory of enriched derivators. First steps in this direction will be developed in \cite{gs:enriched}.

\subsection{Symmetries of stable derivators of coherent squares}
\label{subsec:symm}

Let us start with the strong stable equivalences themselves.
We will have a use for \autoref{lem:excone} in the special case of $A=\lrcorner$ and hence $A^\lhd=\square$, yielding the homotopy exactness of the square
\begin{equation}
\vcenter{\xymatrix@-.5pc{
    \lrcorner\ar[r]^-{i_\lrcorner}\ar[d]_-{i_\lrcorner}&\square\ar[d]^-{i_\lrcorner^\lhd}\\
    \square\ar[r]_-{i_\square}&\square^\lhd.
  }}\label{eq:tiltsquare}
\end{equation}
Thus, a square in a derivator is cartesian if and only if it lies in the essential image of $i_\square^\ast(i_\lrcorner^\lhd)_\ast$.

\begin{thm}\label{thm:tiltsquare}
The source of valence three, the sink of valence three, and the commutative square are strongly stably equivalent.
\end{thm}

\begin{proof}
Using the homotopy exactness of \eqref{eq:tiltsquare}, we can simply adapt the standard representation theoretic argument (see for instance~\cite[Lemma 5.2]{kellerneeman:may}) to the language of stable derivators. The right Kan extension $(i_\lrcorner^\lhd)_\ast\colon\D^\square\to\D^{\square^\lhd}$ induces an equivalence onto the full subderivator of $\D^{\square^\lhd}$ spanned by the diagrams $X$ such that $i_\square^*(X)$ is cartesian.
Let $A\subseteq\square^\lhd$ be the full subcategory obtained by removing the final object $(1,1)$. Thus, we have isomorphisms $A=(\ulcorner)^\lhd$ and $\square^\lhd = A^\rhd$. If we denote by $j\colon A \to A^\rhd=\square^\lhd$ the inclusion, then it is easy to check that the left Kan extension $j_!\colon\D^A\to\D^{\square^\lhd}$ induces an equivalence onto the full subderivator of $\D^{\square^\lhd}$ spanned by objects $X$ such that $i_\square^*(X)$ is cocartesian. But since \D is stable the two subderivators of $\D^{\square^\lhd}$ agree and we deduce that $\square\sse A$.
Noting that $A$ is merely a reorientation of the source or sink of valence three, we can use \autoref{thm:tiltone} to get
\[
\square \sse A \sse [1]^3_{\leq 1} \sse [1]^3_{\geq 2},
\]
concluding the proof.
\end{proof}

We again stress the added generality of this result: it was well known that these quivers are derived equivalent, which we have just seen to be a shadow of equivalences of the stable derivators in the background. More importantly, the same result now also applies to \autoref{egs:derivators}.

In order to explain the relation of \autoref{thm:tiltsquare} to May's axioms, we will reinterpret the above strong stable equivalence in terms of another full subderivator. Let us first fix some notation.

\begin{notn} \label{notn:reinforced-cubes}
Viewing the ordered set $\lZ$ of integers as a small category, let $C_0 \subseteq \lZ^3$ be the full subcategory spanned by all objects $(n+\ep_1,n+\ep_2,n+\ep_3)$, where $n$ runs over the integers and $\ep_1, \ep_2, \ep_3 \in \{0,1\}$. This category can be visualized as an infinite string of cubes glued together. Further, denote by $C$ the full subcategory $C\subseteq\lZ^3$ spanned by the objects of $C_0$ and the objects 
\begin{equation}
(n,n+2,n+1),\quad (n+2,n+1,n),\quad (n+1,n,n+2)\quad\text{for}\quad n \in \lZ.
\label{eq:zeroC}
\end{equation}
One can imagine $C$ as a chain of cubes where the link between each pair of adjacent cubes is `reinforced' by three commutative squares (see \autoref{fig:reinforced-cubes-1}, page~\pageref{fig:reinforced-cubes-1}). When mentioning a \emph{small square} in $C$, we mean either a dimension $2$ subcube of a~cube $[(n,n,n),(n+1,n+1,n+1)] \subseteq C$ or one of the `reinforcing' squares in $C$.

Given a stable derivator $\D$, we will write $\D^{C,\mathrm{ex}}$ for the full subprederivator of $\D^C$ spanned by the objects $Y$ such that all small squares in $Y$ are bicartesian and all objects of $C$ not contained in $C_0$, i.e., the objects as in \eqref{eq:zeroC}, are populated by zero objects (again see \autoref{fig:reinforced-cubes-1} for an illustrative example).
\end{notn}

Now we observe the following.

\begin{prop} \label{prop:source-square-to-cubes}
In the above notation, let us consider the functors
\begin{enumerate}
\item $i\colon [1]^3_{\leq 1} \to C$ given by restriction of the codomain of the obvious embedding $[1]^3_{\leq 1} \subseteq \lZ^3$ and
\item $j\colon \square \to C$ defined as the functor which maps the objects $00, 01, 10, 11 \in \square$ to the objects $010,011,110,112 \in C$, respectively.
\end{enumerate}
If \D is a stable derivator then so is $\D^{C,\mathrm{ex}}$ and we have pseudo-natural equivalences
\[ \xymatrix@1{ \D^{[1]^3_{\leq 1}} & \D^{C,\mathrm{ex}} \ar[l]_{i^*} \ar[r]^{j^*} & \D^\square }. \]
\end{prop}

\begin{rmk}
In the notation of \autoref{fig:reinforced-cubes-1}, page~\pageref{fig:reinforced-cubes-1}, $i^*$ is the restriction to the full subcategory spanned by $s,a,b,c$ and $j^*$ is the restriction to $b,d,f,\ell$. In particular, our convention for drawing an object $(n_1,n_2,n_3)\in C$ is such that $n_1$ is the coordinate in the $a$-direction, $n_2$ the one in the $b$-direction, and $n_3$ the one in the $c$-direction. Note also that in these figures the drawing of the arrows $(n+1,n,n+1)\to(n+1,n,n+2)$ is slightly adjusted in order to avoid that the objects $(n+1,n,n+2)$ and $(n,n+2,n+1)$ overlap.
\end{rmk}

\begin{proof}
The functor $i$ factors as the composition of full inclusions
\[ \xymatrix@1{ [1]^3_{\leq 1} \ar[r]^{i_1} & C_1 \ar[r]^{i_2} & C_2 \ar[r]^{i_3} & C_3 \ar[r]^{i_4} & C } \]
where
\begin{enumerate}
\item $C_1$ contains $[1]^3_{\leq 1}$ and all $(n,n+2,n+1), (n+2,n+1,n), (n+1,n,n+2)$ with $n \ge 0$,
\item $C_2$ is obtained from $C_1$ by adding all $(n_1,n_2,n_3) \in C$ with non-negative coordinates, and
\item $C_3$ contains $C_2$ and also $(n,n+2,n+1), (n+2,n+1,n),$ and  $(n+1,n,n+2)$ for $n<0$.
\end{enumerate}
Associated to these inclusions there is the following sequence of Kan extensions
\[
\D^{[1]^3_{\leq 1}}\stackrel{(i_1)_\ast}{\to}\D^{C_1}\stackrel{(i_2)_!}{\to}\D^{C_2}\stackrel{(i_3)_!}{\to}\D^{C_3}
\stackrel{(i_4)_\ast}{\to}\D^C,
\]
each of which is fully faithful and hence defines an equivalence onto its essential image. By \autoref{lem:extbyzero}, $(i_1)_\ast$ is right extension by zero and that lemma also describes its essential image. A somewhat tedious but completely straightforward checking using intermediate factorizations of $i_2$, \autoref{lem:detection}, and \autoref{cor:subsquares} reveals that the left Kan extension $(i_2)_!$ amounts to alternatively forming a strongly cocartesian cube spanned by 
\[
\{(n+\ep_1,n+\ep_2,n+\ep_3) \mid \ep_1, \ep_2, \ep_3 = 0,1\}
\]
and then adding three cofiber sequences containing the links, both of this in the positive direction. Moreover, we claim that the essential image of $(i_2)_!$ consists precisely of the coherent diagrams $X$ such that all small squares in $C_2$ are cocartesian. To see that, we need to check that the adjunction counit $\ep\colon(i_2)_!(i_2)^*(X) \to X$ is an isomorphism for such an $X$. Thanks to (Der2), this is equivalent to checking that $\ep_c\colon (i_2)_!(i_2)^*(X)_c \to X_c$ is an isomorphism for each $c \in C$. But this follows precisely by checking the $c$-component of the corresponding adjunction counit for the intermediate factorization of $i_2$ through the smallest subsieve of $C_2$ containing $C_1$ and $c$. This proves the claim about $(i_2)_!$. Using \autoref{lem:extbyzero} again, we see that $(i_3)_!$ is left extension by zero and this lemma also determines its essential image. Finally, using arguments dual to the case of $(i_2)_!$, one observes that $(i_4)_\ast$ amounts to alternatively adding strongly cartesian cubes and fiber sequences in the negative direction and we can also determine its essential image.

As an upshot, using the stability of~\D, we obtain an equivalence between $\D^{[1]^3_{\leq 1}}$ and the full subprederivator $\D^{C,\mathrm{ex}}$ of $\D^C$. In particular, $\D^{C,\mathrm{ex}}$ is a stable derivator,
\[ F = (i_4)_\ast(i_3)_!(i_2)_!(i_1)_\ast\colon \D^{[1]^3_{\leq 1}} \to \D^{C,\mathrm{ex}} \]
is an equivalence, and the restriction $i^* = i_1^*i_2^*i_3^*i_4^*$ is an inverse of $F$. 

In order to establish the result for $j$, consider first the restriction of $[1]^3_{\geq 2} \subseteq \lZ^3$ to $k\colon [1]^3_{\geq 2} \to C$. If follows from the description of $\D^{C,\mathrm{ex}}$ that the composition $k^* F\colon \D^{[1]^3_{\leq 1}} \to \D^{[1]^3_{\geq 2}}$ is the equivalence from~\autoref{cor:tiltsourcesink}. Moreover, the composition $j^* F\colon \D^{[1]^3_{\leq 1}} \to \D^\square$ is equivalent to the composition of $k^* F$ with the strong stable equivalence $\D^{[1]^3_{\geq 2}} \to \D^\square$, dual to the one explicitly described in the proof of \autoref{thm:tiltsquare}. Thus, $j^\ast F$ is an equivalence and hence also $j^*\cong (j^* F) i^*$.\end{proof}

In particular, since $\D^{C,\mathrm{ex}}$ is a stable (hence pointed) derivator, we have the (abstractly defined) suspension endomorphism $\Sigma\colon \D^{C,\mathrm{ex}} \to \D^{C,\mathrm{ex}}$. In order to be able to compute with this $\Sigma$, we establish the following lemma. It implies that the suspension for $\D^{C,\mathrm{ex}}$ is simply the restriction of the suspension for $\D^C$.

\begin{lem} \label{lem:DCex-exact}
In the above notation, the inclusion $\D^{C,\mathrm{ex}} \to \D^C$ is exact.
\end{lem}

\begin{proof}
First, we note that each $\D^{C,\mathrm{ex}}(A)$ contains the zero object of $\D^C(A)$ which is simply the constant $A$-shaped diagram taking value $0$. For this it suffices to observe that constant $n$-cubes, i.e., coherent diagrams in the image of the restriction functor $\pi^\ast\colon\D(\bbone)\to\D([1]^n)$ are bicartesian, which is true for example by a combination of \autoref{cor:subsquares} and \cite[Proposition~3.12]{groth:ptstab}. Thus, the inclusion $\D^{C,\mathrm{ex}} \to \D^C$ preserves zero objects. 

It remains to prove that $\D^{C,\mathrm{ex}} \to \D^C$ also preserves cartesian squares (which are also cocartesian by stability). Let $X \in \D^{C,\mathrm{ex}}(\lrcorner)$ and consider the right Kan extension $Y = (i_\lrcorner)_\ast(X) \in \D^C(\square)$ taken in $\D^C$. It suffices to prove that $Y$ in fact belongs to $\D^{C,\mathrm{ex}}(\square)$, i.e., we have to check certain vanishing conditions and bicartesianness conditions. Appealing again to \cite[Proposition~3.12]{groth:ptstab}, the vanishing conditions are satisfied, so it remains to take care of the bicartesianness conditions. To this end, let $k\colon \square \to C$ be one of the small squares from \autoref{notn:reinforced-cubes} whose exactness determines the objects of $\D^{C,\mathrm{ex}}$ among those in $\D^C$. Viewing $Y$ as an object of $\D(C \times \square)$ and $X$ as an object of $\D(C \times \lrcorner)$, we need to prove that the square $(k \times 11)^\ast(Y) \in \D(\square)$ is cartesian. However, by assumption $(k \times \id_\square)^\ast(Y) \in \D(\square \times \square)$ is of the form
\[
(k \times \id_\square)^\ast(\id_C \times i_\lrcorner)_\ast(X) \cong
(\id_\square \times i_\lrcorner)_\ast(k \times \id_\lrcorner)^\ast(X) \cong
(i_\lrcorner \times i_\lrcorner)_\ast(Z),
\]
where the first isomorphism is given as in \autoref{egs:htpy}(ii), the second one is obtained from the fact that $X$ was chosen in $\D^{C,\mathrm{ex}}$, and where $Z \in \D(\lrcorner \times \lrcorner)$ is the restriction of $(k \times \id_\lrcorner)^\ast(X)$ along $i_\lrcorner \times \id_\lrcorner\colon \lrcorner \times \lrcorner \to \square \times \lrcorner$. The conclusion follows by writing the inclusions $i_\lrcorner \times i_\lrcorner$ as the composition
\[
\lrcorner \times \lrcorner \longrightarrow \lrcorner \times \square \longrightarrow \square \times \square
\]
and considering the corresponding right Kan extensions.
\end{proof}

The first noteworthy consequence of the construction from \autoref{prop:source-square-to-cubes}, as indicated in \autoref{fig:reinforced-cubes-1} at the end of the text, is the following. The equivalence~$F$ sends an object $X\in\D^{[1]^3_{\leq 1}}$ with underlying diagram
\begin{equation}
\vcenter{
\xymatrix@-.7pc{
&& a \\
s \ar[urr] \ar[rr] \ar[drr] && b \\
&& c
}
}
\label{eq:X}
\end{equation}
to a diagram $Y\in\D^{C,\mathrm{ex}}$ looking like \autoref{fig:reinforced-cubes-1}. In particular, the `doubly infinite' incoherent diagram of $Y \in \D^{C,\mathrm{ex}}$ contains up to isomorphisms and (de)suspensions only twelve objects $a,b,c,d,e,f,k,\ell,m,s,t,u$ of $\D$ (see \autoref{rmk:twelve}). In order to prove this, we consider the automorphism $\mathrm{sh}\colon C \to C$ which sends $ (n_1,n_2,n_3)$ to $(n_1+1,n_2+1,n_3+1)$. The restriction morphism $\mathrm{sh}^*\colon \D^C \to \D^C$ shifts a diagram one cube to the left, and it is immediate that it restricts to an induced morphism $\mathrm{sh}^*\colon \D^{C,\mathrm{ex}} \to \D^{C,\mathrm{ex}}$. 

\begin{prop} \label{prop:susp-shift}
In the above notation, for every stable derivator \D there is a canonical isomorphism
\[ (\mathrm{sh}^*)^3 \cong \Sigma^2\colon \D^{C,\mathrm{ex}}\to \D^{C,\mathrm{ex}}. \]
\end{prop}

\begin{proof}
Let $(j^*,G)\colon \D^{C,\mathrm{ex}} \rightleftarrows \D^\square$ be the equivalence from \autoref{prop:source-square-to-cubes}. Since $j^*$ is exact and as such commutes with the suspension, it suffices to construct a canonical isomorphism $j^* (\mathrm{sh}^*)^3 G \cong \Sigma^2\colon\D^\square\to\D^\square$. Clearly also $j^* (\mathrm{sh}^*)^3 = j_2^*$, where $j_2 = \mathrm{sh}^3 \circ j\colon \square \to C$ sends $00, 01, 10, 11 \in \square$ to the objects $343,344,443,445 \in C$, respectively. Let us also consider the functor $j_1\colon \square \to C$ which sends $00, 01, 10, 11 \in \square$ to $221,223,232,233\in C$, respectively.

We will obtain the desired natural isomorphisms in two steps. In the first step, we construct a natural isomorphism $j_1^* G \cong \Sigma\colon\D^\square\to\D^\square$. For that purpose, let $q\colon \square \times \square \to C$ be the functor defined as follows. We put $q(x,00) = j(x)$ and $q(x,11) = j_1(x)$ for each $x \in \square$, while on the remaining two corners we define $q$ by
\begin{enumerate}
\item $q(00,01) = 021$, $q(01,01) = 021$, $q(10,01) = 132$, $q(11,01) = 132$, and
\item $q(00,10) = 210$, $q(01,10) = 213$, $q(10,10) = 210$, $q(11,10) = 213$,
\end{enumerate}
and one easily checks that this is well-defined. Note that the values $q(x,01)$ and $q(x,10)$ belong to the objects \eqref{eq:zeroC}. Consequently, for every $Y\in\D^{C,\mathrm{ex}}$ the squares $(\id\times 01)^\ast q^\ast(Y)\in\D^\square$ and $(\id\times 10)^\ast q^\ast(Y)\in\D^\square$ are trivial. In order to show that $(\id \times 11)^*q^*(Y)$ is canonically isomorphic to the suspension of $(\id \times 00)^*q^*(Y)$, let us consider the functor $x \times \id\colon \square \to \square \times \square$ for any $x \in \square$. Then it is easily checked using \autoref{prop:compcanc} that for any diagram $Y$ in $\D^{C,\mathrm{ex}}$ the restriction $(x \times \id)^* q^*(Y)$ to an object of $\D^\square$ is cocartesian. For example, in the case of $x=01$ this amounts to considering the following pasting of bicartesian squares
\[
\xymatrix{
d\ar[r]\ar[d]&u\ar[r]\ar[d]&l\ar[r]\ar[d]&\Sigma a\ar[r]\ar[d]&0\ar[d]\\
0\ar[r]&k\ar[r]&\Sigma c\ar[r]&\Sigma t\ar[r]&\Sigma d.
}
\]
Combining these observations with \cite[Corollary~3.14]{groth:ptstab}, this implies that the square $(\id \times 11)^*q^*(Y)$ is canonically isomorphic to the suspension of $(\id \times 00)^*q^*(Y)$, which is to say that we have natural isomorphisms $j_1^* G \cong \Sigma j^\ast G\cong \Sigma\colon\D^\square\to\D^\square$.

Using essentially the same argument, we prove that for any $Y \in \D^{C,\mathrm{ex}}$, there is a canonical natural isomorphism $j_2^*(Y) \cong \Sigma j_1^*(Y)$. This time we use the functor $q'\colon \square \times \square \to C$ defined by $q'(x,00) = j_1(x)$ and $q'(x,11) = j_2(x)$ for each $x \in \square$, and by
\begin{enumerate}
\item $q'(x,01) = 243$ for each $x \in \square$, and
\item $q'(00,10) = 321$, $q'(01,10) = 324$, $q'(10,10) = 432$, $q'(11,10) = 435$.
\end{enumerate}
Again, the values $q'(x,01)$ and $q'(x,10)$ belong to the objects \eqref{eq:zeroC}, and the squares $(\id\times 01)^\ast q'^\ast(Y)$ and $(\id\times 10)^\ast q'^\ast(Y)$ thus vanish. Similar arguments as in the previous case allow us to deduce that $(\id \times 11)^*q'^*(Y)$ is canonically isomorphic to the suspension of $(\id \times 00)^*q'^*(Y)$, and this gives us a natural isomorphism $j_2^\ast G \cong \Sigma j_1^\ast G\colon\D^\square\to\D^\square$. As an upshot, we obtain natural isomorphisms
\[
j^\ast(\mathrm{sh}^\ast)^3G=j_2^\ast G\cong\Sigma j_1^\ast G\cong \Sigma^2\colon \D^\square\to\D^\square,
\]
and hence $(\mathrm{sh}^\ast)^3\cong\Sigma^2\colon \D^{C,\mathrm{ex}}\to \D^{C,\mathrm{ex}}$ as intended.
\end{proof}

\begin{rmk}\label{rmk:twelve}
Note that in the course of the proof of \autoref{prop:susp-shift}, we have also proved that $\Sigma b$, $\Sigma d$, $\Sigma f$ and $\Sigma\ell$ are located in \autoref{fig:reinforced-cubes-1} exactly where depicted (this was the first step of the construction). By symmetry, the same follows for the suspensions of $a,c,e,k,$ and $m$, so \autoref{fig:reinforced-cubes-1} really depicts the shape of the diagram of an object of $\D^{C,\mathrm{ex}}$. At present the decoration $\Sigma t$ is just a label for an object, but this will be addressed in the course of the following construction. For now we note that \autoref{prop:susp-shift} together with its proof thus shows that up to suspensions and isomorphisms there are only twelve objects in \autoref{fig:reinforced-cubes-1}. For a representation theoretic perspective on this see the introduction to \S\ref{subsec:AR}.
\end{rmk}

Another look at \autoref{fig:reinforced-cubes-1} reveals that although all suspensions and desuspensions of $a,b,c,d,e,f,k,\ell,m$ are in the infinite picture, only every other suspension of $s,t,u$ is there. In order to obtain odd suspensions of $s,u$ and even suspensions of $t$, we must pass to \autoref{fig:reinforced-cubes-2}. In fact, we can do so functorially by constructing an endomorphism $T\colon \D^{C,\mathrm{ex}} \to \D^{C,\mathrm{ex}}$ which will be a square root of the endomorphism $\mathrm{sh}^*\colon \D^{C,\mathrm{ex}} \to \D^{C,\mathrm{ex}}$ and a cube root of $\Sigma\colon \D^{C,\mathrm{ex}} \to \D^{C,\mathrm{ex}}$. This needs a little preparation and will be obtained as \autoref{prop:sqrt-shift}.

Let $X$ be an object of $\D^{[1]^3_{\leq 1}}$ with underlying diagram \eqref{eq:X}. We will functorially construct from $X$ a coherent diagram $Z$ as in \autoref{fig:Z}, whose shape $E$ is described as follows. We start with the subposet $E_0 \subseteq \lZ^4$ spanned by the hypercube $[1]^4$ and the objects $0211,2101,1021,2111,1211,$ and $1121$. Note that if we restrict to the objects of $E_0$ with the last coordinate equal to $1$, forgetting the last coordinate results into a full embedding of this full subcategory of $E_0$ into $C$ of \autoref{notn:reinforced-cubes}. The objects $0211,2101,1021$ then embed as vertices of the `reinforcing' squares which do not belong to $C_0$, and the objects $2111,1211,1121$ belong to the `next cube'. Now $E$ is defined as the poset constructed from $E_0$ by adjoining two new objects $(-\infty,0), (-\infty,1)$ and morphisms $(-\infty,0) \to 0000$, $(-\infty,1) \to 0001,$ and $(-\infty,0) \to (-\infty,1)$. In \autoref{fig:Z}, $(-\infty,0), (-\infty,1)$ correspond to the two leftmost objects, $E_0$ to the other objects, and the curly arrows indicate that for each vertex of the upper cube there is a unique morphism to the corresponding vertex of the lower cube.

\begin{figure}
\centering
\xymatrix@-.7pc{
&&& a \oplus c \ar[drr] \ar[rr] && a \ar[drr]
\\
s \ar[r] \ar[ddd] & a \oplus b \oplus c \ar[urr] \ar@{.>}[rr] \ar[drr] \ar[ddd] &&
a \oplus b \ar@{.>}[urr]|\hole \ar@{.>}[drr]|\hole &&
c \ar[rr] &&
0 \ar[ddd]
\\
&&& b \oplus c \ar[urr] \ar[rr] \ar@{~>}[d] && b \ar[urr] \ar@{~>}[d] &&& 0 \ar[dr]
\\
&&& e \ar[drr] \ar[rr] && k \ar[drr] \ar[urrr]|(.64)\hole &&&& \Sigma a
\\
0 \ar[r] & t \ar[urr] \ar@{.>}[rr] \ar[drr] &&
f \ar@{.>}[urr]|\hole \ar@{.>}[drr]|\hole &&
m \ar[rr] \ar[ddrr] &&
\Sigma s \ar[urr] \ar[rr] \ar[drr] &&
\Sigma b
\\
&&& d \ar[urr] \ar[rr] &&
\ell \ar@{-}+R!(0,1);+(5.35,3.1)\ar@{.>}+R!(5.3,4.2);+(15.5,8.75) \ar@{-}+R;+(7.5,0)\ar@{.>}+R!(8.3,0);+(17,0) &&
0 \ar@{.}+R!(0,1);+(8.8,4.4)\ar@{->}+R!(9.8,5.6);+(18,8.75) &&
\Sigma c
\\
&&&&&&& 0 \ar[urr]
}
\caption{The coherent diagram $Z$ in $\D^E$.}
\label{fig:Z}
\end{figure}
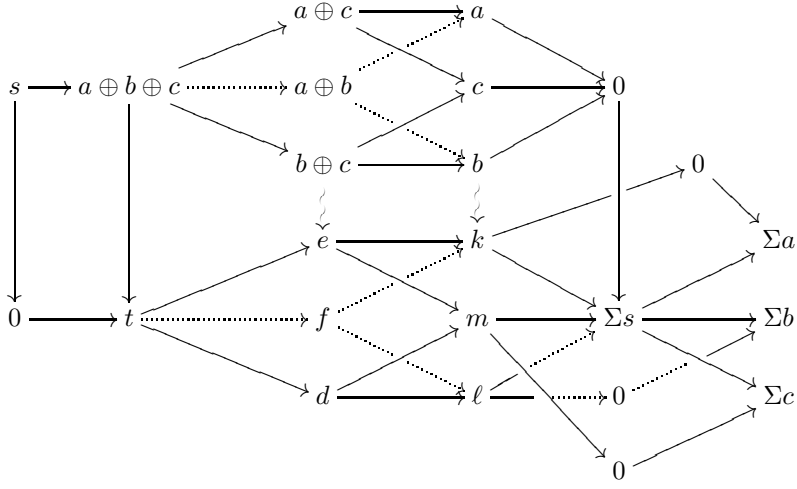

Before giving the construction, let us emphasize that we will also have to show that the objects in \autoref{fig:Z} labeled $t,d,e,f,k,l,m$ agree up to suspensions and isomorphisms with the corresponding objects in \autoref{fig:reinforced-cubes-1}. Now, we obtain $Z$ in $\D^E$ from $X$ in $\D^{[1]^3_{\leq 1}}$ in six steps. Let us consider the fully faithful functor $v\colon [1]^3_{\leq 1} \to E$ given by 
\[
000\mapsto(-\infty,0),\quad 100\mapsto 1100,\quad 010\mapsto 0110,\quad 001\mapsto 1010.
\]
In \autoref{fig:Z}, the positions of $v(x)$, where $x$ runs over $000,100,010,001$, are labeled by $s,a,b,c$ in this order. This functor factors as a composition
\[ \xymatrix@1{ [1]^3_{\leq 1} \ar[r]^{v_1} & E_1 \ar[r]^{v_2} & E_2 \ar[r]^{v_3} & E_3 \ar[r]^{v_4} & E_4 \ar[r]^{v_5} & E_5 \ar[r]^{v_6} & E } \]
where $v_1$ is obtained from $v$ by restricting the codomain, $v_2,v_3,v_4,v_5,$ and $v_6$ are inclusions, and the respective intermediate categories are the following full subcategories of~$E$.
\begin{enumerate}
\item $E_1$ contains besides $(-\infty,0)$-$1100$-$0110$-$1010$ also the objects $1110$.
\item $E_2$ is obtained from $E_1$ by adding all remaining objects of the upper cube $[1]^3\times\{0\}$.
\item $E_3$ contains all objects of $E_2$ and the object $(-\infty,1)$.
\item $E_4$ is obtained from $E_3$ by adding all objects of the lower cube $[1]^3\times\{1\}$.
\item $E_5$ contains $E_4$ and the objects $0211,2101,1021$.
\end{enumerate}
Associated to these fully faithful inclusions, we consider the Kan extension functors
\begin{equation}
\D^{[1]^3_{\leq 1}}\stackrel{(v_1)_\ast}{\to}\D^{E_1}\stackrel{(v_2)_\ast}{\to}\D^{E_2}\stackrel{(v_3)_\ast}{\to}\D^{E_3}
\stackrel{(v_4)_!}{\to}\D^{E_4}\stackrel{(v_5)_\ast}{\to}\D^{E_5}\stackrel{(v_6)_!}{\to}\D^E
\label{eq:H}
\end{equation}
each of which is fully faithful hence inducing an equivalence onto its essential image. We will denote the composition of these functors by $H\colon\D^{[1]^3_{\leq 1}} \to \D^E$.

By \autoref{lem:extbyzero}, $(v_1)_\ast$ is right extension by zero. An application of \autoref{lem:detectionplus} shows that $(v_2)_\ast$ amounts to adding a strongly bicartesian cube. In particular, this upper cube contains all partial products of $a,b,c$ and projections between them. Since we are in the stable context, these products are actually biproducts (see~\cite[Proposition~4.7]{groth:ptstab}). Using \autoref{lem:extbyzero} again, we deduce that $(v_3)_\ast$ is right extension by zero. A repeated application of \autoref{lem:detection} implies that $(v_4)_!$ forms cofibers of the eight maps $(-\infty,0)\to(x,0),x\in[1]^3$. Also $(v_5)_\ast$ is right extension by zero and finally, with \autoref{lem:detectionplus} in mind, $(v_6)_!$ amounts to adding three more cofibers. As an upshot, using \autoref{cor:2-out-of-3} and \autoref{cor:subsquares}, we obtain an equivalence between $\D^{[1]^3_{\leq 1}}$ and the full subderivator of $\D^E$ spanned by all diagrams satisfying the above vanishing conditions coming from extensions by zero and making all subcubes of dimension $2,3,$ and $4$ strongly bicartesian. 

We next show that the objects in \autoref{fig:Z} labeled $t,d,e,f,k,l,m$ match with the corresponding objects in \autoref{fig:reinforced-cubes-1}. For the objects $k,l,m$ this amounts to finding suitable chains of bicartesian squares in \autoref{fig:reinforced-cubes-1} and then to use the pasting property of bicartesian squares. For example, for $m$ it suffices to consider the chain
\[
\xymatrix{
s\ar[r]\ar[d]&a\ar[r]\ar[d]&f\ar[r]\ar[d]&0\ar[d] \\
c\ar[r]&e\ar[r]&u\ar[r]&m,
}
\] 
and similarly in the remaining two cases. For the objects $d,e,f$ this is obtained by considering the first cube in \autoref{fig:reinforced-cubes-1} together with a threefold application of the Mayer--Vietoris theorem for stable derivators (\cite[Theorem~6.1]{gps:mayer} or \cite[Proposition~2.4]{stpo:co-t-str}). Finally, for the object $t$ it is enough to consider the third cube in \autoref{fig:reinforced-cubes-1} and to compare with the lower cube in \autoref{fig:Z}. This finally explains why one object in \autoref{fig:reinforced-cubes-1} was denoted by $\Sigma t$.

In order to conclude the functorial construction of \autoref{fig:Z} it remains to show that in that diagram suspensions show up as indicated. This final step is settled by the following lemma in which we consider the functor $z\colon [1]^3_{\leq 1} \to E$ sending $(n_1,n_2,n_3)$ to $(n_1+1,n_2+1,n_3+1,1)$, i.e., the functor pointing at the objects where we would like to find the suspensions.

\begin{lem} \label{lem:susp-in-Z}
In the above notation, for every stable derivator~\D there is a canonical isomorphism
$z^* H \cong \Sigma\colon\D^{[1]^3_{\leq 1}}\to\D^{[1]^3_{\leq 1}}$.
\end{lem}

\begin{proof}
We define a functor $q\colon[1]^3_{\leq 1}\times \square\to E$ by $q(-,00)=v,$ $q(-,11)=z,$ and
\begin{enumerate}
\item $q(000,01)=(-\infty,1),q(100,01)=2101,q(010,01)=0211,q(001,01)=1021,$ and
\item $q(x,10)=1110$ for all $x\in [1]^3_{\leq 1}$.
\end{enumerate}
One easily checks that $q$ is a well-defined functor and that for $Z$ as in \autoref{fig:Z} the restrictions 
$(\id\times 10)^\ast q^\ast Z,(\id\times 01)^\ast q^\ast Z\in\D^{[1]^3_{\leq 1}}$ vanish. Moreover, using suitable chains of bicartesian squares in $Z$ the squares $(x\times\id)^\ast q^\ast Z\in\D^\square$ are shown to be bicartesian for all $x\in[1]^3_{\leq 1}$. By \cite[Corollary~3.14]{groth:ptstab} this implies that $q^\ast Z$ considered as an object of $(\D^{[1]^3_{\leq 1}})^\square$ is bicartesian, and hence that $(\id\times 11)^\ast q^\ast Z$ is canonically isomorphic to the suspension of $(\id\times 00)^\ast q^\ast Z$. Putting this together we obtain canonical natural isomorphisms
\[
z^\ast H\cong \Sigma v^\ast H\cong \Sigma\colon \D^{[1]^3_{\leq 1}}\to\D^{[1]^3_{\leq 1}},
\]
as intended.
\end{proof}

Now we shall construct the promised square root~$T$ of $\mathrm{sh}^*\colon \D^{C,\mathrm{ex}} \to \D^{C,\mathrm{ex}}$ and cube root of $\Sigma\colon \D^{C,\mathrm{ex}} \to \D^{C,\mathrm{ex}}$. Consider the embedding $w\colon [1]^3_{\leq 1} \to E$ sending $(n_1,n_2,n_3)$ to $(n_1,n_2,n_3,1)$. This corresponds to the positions $t,e,f,$ and $d$ in \autoref{fig:Z}. If we still denote by $H$ the composition \eqref{eq:H}, then we define the endomorphism $T' = w^*H\colon \D^{[1]^3_{\leq 1}}\to \D^{[1]^3_{\leq 1}}$. By construction, if $X\in\D^{[1]^3_{\leq 1}}$ looks like \eqref{eq:X}, then $T'(X)$ has underlying diagram
\begin{equation}
\vcenter{
\xymatrix@-.7pc{
&& e \\
t \ar[urr] \ar[rr] \ar[drr] && f \\
&& d.
}
}
\label{eq:T'}
\end{equation}
Finally, a conjugation of $T'$ by the equivalences $(i^\ast,F)\colon \D^{C,\mathrm{ex}} \rightleftarrows \D^{[1]^3_{\leq 1}}$ from \autoref{prop:source-square-to-cubes} yields the desired endomorphism
\begin{equation}
 T=F T' i^\ast\colon \D^{C,\mathrm{ex}} \to \D^{[1]^3_{\leq 1}} \to \D^{[1]^3_{\leq 1}} \to \D^{C,\mathrm{ex}}.
\label{eq:T} 
\end{equation}
We also recall from that same proposition that there is a further equivalence $(j^\ast,G)\colon\D^{C,\mathrm{ex}} \rightleftarrows \D^\square$ since it will be used in the proof of the following result.

\begin{prop} \label{prop:sqrt-shift}
In the above notation, $T$ sends an object $Y \in \D^{C,\mathrm{ex}}$ with underlying diagram as in \autoref{fig:reinforced-cubes-1} to an object with underlying diagram as in \autoref{fig:reinforced-cubes-2}. Moreover, there are canonical isomorphisms
\[
T \cong (\mathrm{sh}^*)\inv \Sigma \cong \Sigma (\mathrm{sh}^*)\inv,  \qquad
T^2 \cong \mathrm{sh}^*, \qquad \textrm{and} \qquad  
T^3 \cong \Sigma
\]
of endomorphisms of $\D^{C,\mathrm{ex}}$. In particular, $T$ commutes both with $\mathrm{sh}^*$ and $\Sigma$ up to canonical isomorphisms. 
\end{prop}

\begin{proof}
Let $Y \in \D^{C,\mathrm{ex}}$ have underlying diagram as in \autoref{fig:reinforced-cubes-1}. By the reasoning prior to this proposition, the diagram $T'i^\ast (Y)$ looks like \eqref{eq:T'}. A combination of \autoref{lem:susp-in-Z} and \autoref{prop:susp-shift} implies that $T(Y)$ has an underlying diagram as in \autoref{fig:reinforced-cubes-2}, establishing the first claim. From this and \autoref{lem:DCex-exact} it immediately follows that we have the first of the canonical isomorphisms $T \cong (\mathrm{sh}^*)\inv\Sigma \cong \Sigma (\mathrm{sh}^*)\inv$. The second canonical isomorphism exists because $(\mathrm{sh}^*)\inv$ is an equivalence of derivators, and hence commutes with Kan extensions. From this it also follows easily that $T$ commutes with $(\mathrm{sh}^*)\inv$ and $\Sigma$ up to canonical isomorphisms. Combining these results with \autoref{prop:susp-shift} we calculate
\[
T^2 \cong \big((\mathrm{sh}^*)\inv\Sigma\big)^2 \cong (\mathrm{sh}^*)^{-2}\Sigma^2 \cong (\mathrm{sh}^*)^{-2}(\mathrm{sh}^*)^3=\mathrm{sh}^*
\]
and
\[
T^3 = T^2 T \cong \mathrm{sh}^* \big((\mathrm{sh}^*)\inv \Sigma\big) \cong \Sigma,
\]
and these calculations hence conclude the proof.
\end{proof}

\begin{cor} \label{cor:sqrt-shift}
Let $Y \in \D^{C,\mathrm{ex}}$ be as in \autoref{fig:reinforced-cubes-1}. Then, for each $n \in \lZ$ there are coherent cofiber sequences
\[
\xymatrix@-.5pc{
\Sigma^n s \ar[r] \ar[d] & \Sigma^n(a \oplus b \oplus c) \ar[d] \\
0 \ar[r] & \Sigma^n t, \\
}
\qquad
\xymatrix@-.5pc{
\Sigma^n t \ar[r] \ar[d] & \Sigma^n(d \oplus e \oplus f) \ar[d] \\
0 \ar[r] & \Sigma^n u, \\
}
\qquad
\xymatrix@-.5pc{
\Sigma^n u \ar[r] \ar[d] & \Sigma^n(k \oplus \ell \oplus m) \ar[d] \\
0 \ar[r] & \Sigma^{n+1} s. \\
}
\]
\end{cor}

\begin{proof}
The first cofiber sequences for arbitrary values of $n$ are easily obtained by considering suitable suspensions of \autoref{fig:Z}. By \autoref{prop:sqrt-shift} the morphism $T$ commutes with $\mathrm{sh}^\ast$ and $\Sigma$ and the explicit form of \autoref{fig:reinforced-cubes-1} and \autoref{fig:reinforced-cubes-2} thus settles the remaining two cases.
\end{proof}

\subsection{The fractionally Calabi--Yau property}
\label{subsec:fracCY}

Having constructed all the autoequivalences of $\D^\square \simeq \D^{[1]^3_{\le1}} \simeq \D^{C,\mathrm{ex}}$ in the last section, we will explain their interpretation in more classical situations. This allows us to introduce generalizations of certain classically important functors (Serre functors and Auslander--Reiten translations) to the context of an arbitrary stable derivator. It seems to the authors that this extension deserves a more systematic discussion, but in order to keep the paper at a reasonable length we refrain from including such a discussion here.

Suppose that $k$ is a field and $\cT$ is a $k$-linear triangulated category such that each  $\nhom_\cT(x,y)$ is finite dimensional over~$k$. Then a \textbf{Serre functor} on $\cT$ (see~\cite[\S I.1]{RvdB02} and also~\cite{BK89}) is an autoequivalence $\lS\colon \cT \to \cT$ for which there is an isomorphism
\begin{equation} \label{eq:Serre}
\nhom_\cT(x,y) \stackrel{\cong}\longrightarrow \big(\nhom_\cT(y,\lS x)\big)^*
\end{equation}
which is natural both in $x$ and $y$, and where $(-)^*$ stands for the vector space dual. Note that by Yoneda, a Serre functor is unique up to unique isomorphism if it exists.

In some sense the existence of a Serre functor says that $\cT$ is almost a self-dual category. The importance of a Serre functor is that it among others implies by~\cite[Theorem I.2.4]{RvdB02} the existence of almost split triangles which allow us to extract a great deal of information about the structure of $\cT$. We will get back to this point in \S\ref{subsec:AR}. There are two large classes of examples of triangulated categories with a Serre functor (see \cite[Examples 3.2]{BK89}), namely
\begin{enumerate}
\item bounded derived categories of coherent sheaves over a smooth projective variety over $k$ (\cite[\S III.7]{Har77}), and
\item bounded derived categories of finitely generated modules over a finite dimensional algebra of finite global dimension (\cite[\S3.6]{happel:fdalgebra}, \cite{KrLe06}).
\end{enumerate}
We will be mainly interested in case (ii) as this in particular applies to $\cT = D^b(kQ)$ for $Q = [1]^3_{\le1}$, which is none other than the category of compact objects of $\D_k^Q(\bbone)$. If $A$ is a finite dimensional algebra over $k$, then $A^*$ is canonically an $A$-$A$-bimodule and the Serre functor on $D^b(A)$ can be constructed as $\lS = A^* \Lotimes_A -$; see~\cite[\S3.6]{happel:fdalgebra}.

Following~\cite[\S2]{KeSch12}, one can introduce a relative version of the latter if $k$ is an arbitrary commutative ring rather than a field. Suppose that $A$ is a dg algebra over $k$ whose underlying $k$-module is finitely generated and projective. The \textbf{Serre functor} relative to $k$ on $D(A)$ is defined as $\lS = A^* \Lotimes_A -$, where in this case $(-)^* = \Rhom_k(-,k)$. Then we have by~\cite[Proposition 2.2]{KeSch12} for each $x,y \in D(A)$, $y$ compact, an isomorphism
\begin{equation} \label{eq:rel-Serre}
\Rhom_A(x,y) \stackrel{\cong}\longrightarrow \big(\Rhom_A(y,\lS x)\big)^*
\end{equation}
which is natural in both variables. If, moreover, $A^*$ is compact in $D(A)$, then $\lS$ restricts to an autoequivalence of the full subcategory of compact objects of $D(A)$. If $Q$ is a finite quiver without oriented cycles and $A = kQ$ is viewed as a dg algebra concentrated in degree $0$, then $kQ^*$ is compact by the proof of~\cite[Lemma 1]{CB96} (which works for any commutative ground ring) and again $D(kQ) \simeq \D_k^Q(\bbone)$. In particular, the isomorphism from \eqref{eq:rel-Serre} in this case generalizes the one from \eqref{eq:Serre}.

For $Q = [1]^3_{\leq1}$, we can now give a definition of a Serre functor for $\D^Q$ for any stable derivator $\D$ which specializes to the above definitions for the derivator $\D_k$ of a field or a commutative ring $k$. For convenience we will also introduce the Auslander--Reiten translation, which classically differs from a Serre functor by a suspension; see~\cite[\S I.2]{RvdB02}. As an application, we will later prove the (relative) fractionally Calabi--Yau property of $D(kQ)$ for any commutative ring $k$ (see \autoref{thm:fracCYring}). In fact, this will be an immediate consequence of an abstract version of this statement valid in any stable derivator (see \autoref{thm:frac-CY}).

Let $C$ again be the category as in \autoref{notn:reinforced-cubes}. Denote by $r\colon C \to C$ the automorphism of $C$ obtained by restriction from the cyclic rotation of coordinates $\lZ^3 \to \lZ^3$, $(n_1,n_2,n_3) \mapsto (n_3,n_1,n_2)$. Then the autoequivalence $r^*\colon \D^C \to \D^C$ restricts to $r^*\colon \D^{C,\mathrm{ex}} \to \D^{C,\mathrm{ex}}$ by the definition of $\D^{C,\mathrm{ex}}$ in \autoref{notn:reinforced-cubes}. Moreover, $r^*\colon \D^{C,\mathrm{ex}} \to \D^{C,\mathrm{ex}}$ commutes up to canonical isomorphisms with $\mathrm{sh}^*$, $\Sigma$, and by \autoref{prop:sqrt-shift} also with $T$ from \eqref{eq:T}.

\begin{defn}\label{defn:ARS}
Let $\D$ be a stable derivator. 
\begin{enumerate}
\item The \textbf{Serre endomorphism} is $\lS = r^* \mathrm{sh}^*\colon\D^{C,\mathrm{ex}}\to\D^{C,\mathrm{ex}}$.
\item The \textbf{Auslander-Reiten translation} is $\tau = r^* T\inv\colon\D^{C,\mathrm{ex}}\to\D^{C,\mathrm{ex}}$.
\end{enumerate}
\end{defn}

Note that by \autoref{prop:sqrt-shift} there is a canonical isomorphism $\lS \cong \tau\Sigma$, exactly as mentioned above for the classical situation. The relation to the previous concepts is given by the following proposition.

\begin{prop} \label{prop:relate-Serre}
Let $k$ be a commutative ring, $\D_k$ the standard stable derivator for $k$, and let $Q = [1]^3_{\leq1}$. Then the component $\lS_\bbone\colon \D_k^{C,\mathrm{ex}}(\bbone) \to \D_k^{C,\mathrm{ex}}(\bbone)$ of the Serre endomorphism is isomorphic to the conjugate of $(kQ)^* \Lotimes_{kQ} - \colon \D^Q_k(\bbone) \to \D^Q_k(\bbone)$ by the equivalence $(i^*,F)\colon \D_k^{C,\mathrm{ex}} \rightleftarrows \D_k^{[1]^3_{\leq 1}}$ from \autoref{prop:source-square-to-cubes}. More precisely,
\[ (kQ)^* \Lotimes_{kQ} - \; \cong \; i^*_\bbone \lS_\bbone F_\bbone.  \]
\end{prop}

Before proving the proposition, we recall some standard facts about $kQ$-modules which can be found for instance in~\cite{CB96}.

\begin{lem} \label{lem:lattices}
Let $f\colon k \to \ell$ be a homomorphism of commutative rings and $Q$ a finite quiver without oriented cycles.
\begin{enumerate}
\item If $x \in \Mod{(kQ)}$ is finitely generated projective as a $k$-module, then the projective dimension of $x$ in $\Mod{(kQ)}$ is at most one. Moreover, the canonical homomorphism
\[
\Ext^i_{kQ}(x,y) \otimes_k \ell \longrightarrow \Ext^i_{\ell Q}(x \otimes_k \ell, y \otimes_k \ell)
\]
is an isomorphism for $i=0,1$ and each $y \in \Mod{(kQ)}$.
\item The canonical homomorphism $kQ \to \nend_{kQ}(kQ^*)$ which sends $r \in kQ$ to the right multiplication by $r$ is an isomorphism, and $\Ext^i_{kQ}(kQ^*,kQ^*) = 0$ for all $i>0$.
\end{enumerate}
\end{lem}

\begin{proof}
(i) This is essentially just \cite[Lemma 1]{CB96}. To each vertex $v$ of $Q$ we have attached, by the definition of the path algebra $kQ$, a formal path $e_v$ of length zero which becomes an idempotent element of $kQ$ (see~\cite[III.1]{ARS:representation}).  Let $S$ be the $k$-subalgebra of $kQ$ generated by all the trivial paths $e_1,\dots,e_n$, so that $S \cong k \times \dots \times k$ ($n$ times). Let $B \subseteq kQ$ be the $S$-$S$-subbimodule generated by the arrows of $Q$. Then the arrows actually form a $k$-basis of $B$ and $B$ is projective over $S$ both from the left and the right.

Consider now the sequence of $kQ$-$kQ$-bimodules
\begin{equation} \label{eq:bimod-res}
0 \longrightarrow kQ \otimes_S B \otimes_S kQ \stackrel{d_1}\longrightarrow kQ \otimes_S kQ \stackrel{d_0}\longrightarrow kQ \longrightarrow 0
\end{equation}
where the maps $d_0,d_1$ act by $d_0(p \otimes q) = pq$ and $d_1(p \otimes \alpha \otimes q) = p \otimes \alpha q - p\alpha \otimes q$ for all paths $p,q$ and arrows $\alpha$ in $Q$. We claim that this is a projective bimodule resolution of $kQ$. First of all, $kQ \otimes_S kQ \cong kQ \otimes_S S \otimes_S kQ \cong \bigoplus_{i=1}^n kQ \otimes_S e_iS \otimes_S kQ$ as bimodules and each summand is projective since $kQ \otimes_S e_iS \otimes_S kQ \cong kQ \cdot e_i \otimes_k e_i \cdot kQ$. Similarly
\[
kQ \otimes_S B \otimes_S kQ \cong \bigoplus_{\alpha\colon i \to j} kQ \otimes_S e_j B e_i \otimes_S kQ \cong  \bigoplus_{\alpha\colon i \to j} kQ \cdot e_j \otimes_k e_i \cdot kQ,
\]
where the sums run over all arrows of $Q$.
Hence both $kQ \otimes_S B \otimes_S kQ$ and $kQ \otimes_S kQ$ are projective $kQ$-$kQ$-bimodules and it remains to check the exactness of \eqref{eq:bimod-res}. This follows from the fact that there exists an explicitly given null-homotopy $s_0\colon kQ \to kQ \otimes_S kQ$ and $s_1\colon kQ \otimes_S kQ \to kQ \otimes_S B \otimes_S kQ$ of \eqref{eq:bimod-res} viewed as a complex of left $kQ$-modules. Namely, we can define $s_0(p) = p \otimes 1$ and $s_1(p \otimes \alpha_1\cdots\alpha_m) = \sum_{i=1}^m p\alpha_1\cdots\alpha_{i-1} \otimes \alpha_i \otimes \alpha_{i+1}\cdots\alpha_m$ for all paths $p$ and $q = \alpha_1\cdots\alpha_m$, where $\alpha_1,\dots,\alpha_m$ are the arrows of which $q$ is composed. This proves the claim.

Suppose now that $x$ is a left $kQ$-module whose underlying $k$-module is projective. Then $x$ is also projective as an $S$-module, since $x = e_1\cdot x \oplus \cdots \oplus e_n \cdot x$ as a $k$-module and each component of $S \cong k \times \dots \times k$ acts on the corresponding summand $e_i \cdot x$. Thus, applying $-\otimes_{kQ} x$ to \eqref{eq:bimod-res} yields an exact sequence of left $kQ$-modules, and it is straghtforward to check that it is a projective resolution of $x$ of length one. The last claim about the canonical isomorphisms follows directly from the existence of the latter resolution and the fact that the canonical homomorphism $\nhom_{kQ}(p,y) \otimes_k \ell \to \nhom_{\ell Q}(p \otimes_k \ell, q \otimes_k \ell)$ is an isomorphism for each finitely generated projective $kQ$-module $p$.

(ii) For $\nend_{kQ}(kQ^*)$ we just inspect the chain of canonical isomorphisms
\[ kQ \cong \nhom_k(kQ^*,k) \cong \nhom_k(kQ \otimes_{kQ} kQ^*, k) \cong \nhom_{kQ}(kQ^*,\nhom_k(kQ,k)). \]
More precisely, the first isomorphism sends $r \in kQ$ to the evaluation map $\epsilon_r\colon kQ \to k$, $f \mapsto f(r)$, the second homomorphism is the identification of $kQ \otimes_{kQ} kQ^*$ with $kQ^*$ via $r' \otimes f \mapsto r' \cdot f$, where $r' \cdot f$ is defined via $(r'\cdot f)(r) = f(rr')$, and the last isomorphism is the standard adjunction. Thus, the composition sends $r$ to the endomorphism $e_r\colon kQ^* \to kQ^*$, $f \mapsto f \cdot r$, where by definition $(f\cdot r)(r') = f(rr')$.

For the extensions, note that $\Ext^i_{kQ}(kQ^*,kQ^*) = 0$ for all $i>1$ by part (i). Moreover, if $f\colon k \to \ell$ is a projection of $k$ to a quotient field $\ell$, then we know that
\[
\Ext^i_{kQ}(kQ^*,kQ^*) \otimes_k \ell \cong \Ext^i_{\ell Q}(\ell Q^*,\ell Q^*) = 0,
\]
where the second equality follows from the fact that $\ell Q^*$ is an injective $\ell Q$-module. Hence also the localization $\Ext^i_{kQ}(kQ^*,kQ^*) \otimes_k k_{\mathfrak m}$ for ${\mathfrak m} = \ker f$ vanishes by the Nakayama lemma and, as we have proved this for every maximal ideal $\mathfrak m$ of $k$, also $\Ext^i_{kQ}(kQ^*,kQ^*)$ vanishes.
\end{proof}

Now we can finish the proof of the proposition.

\begin{proof}[Proof of \autoref{prop:relate-Serre}]
By \autoref{lem:lattices}(ii), $kQ^*$ is isomorphic to a perfect complex in $D(kQ)$ and we obtain a canonical isomorphism $\Rhom_{kQ}(kQ^*,kQ^*) \cong kQ$. Note also that $kQ^*$ is a compact generator for $D(kQ)$. Indeed, a version of \autoref{lem:lattices}(i) for right $kQ$-modules provides us with a projective resolution $0 \to p_1 \to p_0 \to kQ^* \to 0$ in $\Mod{(kQ\op)}$ and if we apply $(-)^*$ to this sequence, we obtain an exact sequence of left $kQ$-modules
\[ 0 \longrightarrow kQ \longrightarrow p_1^* \longrightarrow p_0^* \longrightarrow 0 \]
with $p_0^*,p_1^*$ in the additive closure of $kQ^*$. Thus, $kQ^*$ generates $kQ$ and so the whole of $D(kQ)$. It follows essentially from the results of Happel and Rickard (see for instance~\cite{rickard:derived-funct} or \cite[Proposition 8.1.4]{keller:construction-eq}) that 
\[
\Rhom_{kQ}(kQ^*,-)\colon D(kQ) \longrightarrow D(\nend_{kQ}(kQ^*)) = D(kQ)
\]
is a triangle equivalence and its left adjoint $(kQ)^* \Lotimes_{kQ} -$ must be its inverse. Given the definition of $\lS$ and $\tau$, we must equivalently prove that $\tau\inv = (r^*)^2T\colon \D^{C,\mathrm{ex}} \to \D^{C,\mathrm{ex}}$ is transferred via $(i^*,F)$ to the endofunctor $\Sigma \Rhom_{kQ}(kQ^*,-)$.

The core of the proof rests on the construction of a suitable resolution of $kQ^*$ which will allow us to compute $\Rhom_{kQ}(kQ^*,-)$. Note the category of  $kQ$-$kQ$-bimodules is equivalent to $(\Mod kQ)^{Q\op}$, the category of representations of $Q\op$ in $\Mod kQ$. The equivalence sends an object of $(\Mod kQ)^{Q\op}$
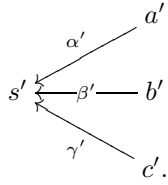
\[
\xymatrix@-.7pc{
&& a' \ar[dll]_{\alpha'} \\
s' && b' \ar[ll]|{\beta'} \\
&& c' \ar[ull]^{\gamma'}.
}
\]
to $s' \oplus a' \oplus b' \oplus c'$, where the left action of the vertices and arrows of $Q$ is given by the left $kQ$-module structure of $s' \oplus a' \oplus b' \oplus c'$, and the right action of $Q$ is given by the morphisms $\alpha',\beta',\gamma'$ in the representation. Similarly, the category of chain complexes of $kQ$-$kQ$-bimodules is equivalent to $\Ch(kQ)^{Q\op}$.

Now we will describe a particular element of $\Ch(kQ)^{Q\op}$ which will turn out to be, upon the just mentioned identification, a shift of a resolution of $kQ^*$ which is cofibrant enough for the computation of $\Rhom_{kQ}(kQ^*,-)$. Let us fix some notation for this purpose. We will label the vertices and arrows of $Q$ as follows,
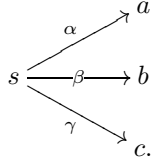
\[
\xymatrix@-.7pc{
&& a \\
s \ar[urr]^\alpha \ar[rr]|\beta \ar[drr]_\gamma && b \\
&& c.
}
\]
For each vertex $v$ we have a projective left $kQ$-module $p_v = kQ \cdot e_v$. In our case we can write the corresponding representations in $(\Mod k)^Q$ very explicitly as follows, where the maps are the identity maps whenever possible and zero maps otherwise:
\begin{equation} \label{eq:projs-source}
\xymatrix@-.7pc{ {\phantom{k}} \\ {\phantom{k}p_s\colon} }
\xymatrix@-.7pc{
& k \\
k \ar[ur] \ar[r] \ar[dr] & k \\
& k
}
\quad\quad
\xymatrix@-.7pc{ {\phantom{k}} \\ {\phantom{0}p_a\colon} }
\xymatrix@-.7pc{
& k \\
0 \ar[ur] \ar[r] \ar[dr] & 0 \\
& 0
}
\quad\quad
\xymatrix@-.7pc{ {\phantom{0}} \\ {\phantom{k}p_b\colon} }
\xymatrix@-.7pc{
& 0 \\
0 \ar[ur] \ar[r] \ar[dr] & k \\
& 0
}
\quad\quad
\xymatrix@-.7pc{ {\phantom{0}} \\ {\phantom{0}p_c\colon} }
\xymatrix@-.7pc{
& 0 \\
0 \ar[ur] \ar[r] \ar[dr] & 0 \\
& k
}
\end{equation}
We also have some distinguished morphisms between these projective representations. We have the morphism $p_a = kQ \cdot e_a \to kQ \cdot e_s = p_s$ given by the right multiplication by $\alpha$ (as both $p_a,p_s$ are subsets of $kQ$) and similarly we have $p_b \to p_s$ given by the right multiplication by $\beta$ and $p_c \to p_s$ given by the multiplication by $\gamma$. In terms of the representations in~\eqref{eq:projs-source}, these are simply the obvious inclusions. Now we can write the promised complex $X$ of $kQ$-bimodules as an element of $\Ch(kQ)^{Q\op}$ where the maps are the obvious ones: identities, split inclusions, and the right multiplications by arrows and their coproducts:
\begin{equation} \label{eq:res-kQ*}
\xymatrix{
\ar@{}[]|\cdots
& 0 \ar[dl] \ar[rr] &
& p_b\!\oplus\!p_c \ar[dl] \ar[rr] &
& p_s \ar[dl] \ar[rr] &
& 0 \ar[dl] \ar@{}[r]|\cdots &
\\
0 \ar@/^1pc/[rr] & 0 \ar[l] \ar@/_1pc/[rr]|(.62)\hole &
p_a\!\oplus\!p_b\!\oplus\!p_c \ar@/^1pc/[rr] & p_a\!\oplus\!p_c \ar[l] \ar@/_1pc/[rr]|(.67)\hole &
p_s \ar@/^1pc/[rr] & p_s \ar[l] \ar@/_1pc/[rr]|(.66)\hole &
0 & 0 \ar[l] \ar@{}[r]|\cdots &
\\
\ar@{}[]|\cdots
& 0 \ar[ul] \ar[rr] &
& p_a\!\oplus\!p_b \ar[ul] \ar[rr] &
& p_s \ar[ul] \ar[rr] &
& 0 \ar[ul] \ar@{}[r]|\cdots &
}
\end{equation}
The only two non-zero representations of $Q\op$ are located in homological degrees $0$ and $-1$. As a complex of left modules, this is simply
\[
\xymatrix@1{ \cdots \ar[r] & 0 \ar[r] & p_a^{\oplus 3} \oplus p_b^{\oplus 3} \oplus p_c^{\oplus 3} \ar[r] & p_s^{\oplus 4} \ar[r] & 0 \ar[r] & \cdots }
\]
In particular, it is a bounded complex of finitely generated projective left modules and $\Rhom_{kQ}(X,-) \cong \Hom_{kQ}(X,-)$, where $\Hom_{kQ}\colon \Ch(kQ \otimes kQ\op) \times \Ch(kQ) \to \Ch(kQ)$ stands for the total Hom of complexes. Suppose now that $Y$ is a $kQ$-module given as representation by
\[
\xymatrix@-.7pc{
&& a \\
s \ar[urr]^\alpha \ar[rr]|\beta \ar[drr]_\gamma && b \\
&& c.
}
\]
Abusing notation slightly, we denote the $k$-modules and $k$-linear maps by the same symbols as vertices and arrows of $Q$. Note that $\nhom_{kQ}(p_s,Y) \cong s$ canonically via $(f\colon p_s \to Y) \mapsto f(e_s)$, and similarly for $a,b,c$. Hence $\Hom_{kQ}(X,Y) \in \Ch(kQ)$ has, when viewed as an object of $\Ch(k)^Q$, the following form
\begin{equation} \label{eq:Rhom-kQ^*}
\xymatrix{
\ar@{}[]|\cdots
& 0  \ar@{<-}[rr] &
& b\!\oplus\!c \ar@{<-}[rr]^{\left(\begin{smallmatrix} \beta \\ \gamma \end{smallmatrix}\right)} &
& s \ar@{<-}[rr] &
& 0 \ar@{}[r]|\cdots &
\\
0 \ar[ur] \ar[r] \ar[dr] \ar@{<-}@/^1pc/[rr] & 0 \ar@{<-}@/_1pc/[rr]|(.62)\hole &
a\!\oplus\!b\!\oplus\!c \ar[ur] \ar[r] \ar[dr] \ar@{<-}@/^1pc/[rr]^(.75){\left(\begin{smallmatrix} \alpha \\ \beta \\ \gamma \end{smallmatrix}\right)} &
a\!\oplus\!c \ar@{<-}@/_1pc/[rr]|(.67)\hole_{\left(\begin{smallmatrix} \alpha \\ \gamma \end{smallmatrix}\right)} &
s \ar[ur] \ar[r] \ar[dr] \ar@{<-}@/^1pc/[rr] & s \ar@{<-}@/_1pc/[rr]|(.65)\hole &
0 \ar[ur] \ar[r] \ar[dr] & 0 \ar@{}[r]|\cdots &
\\
\ar@{}[]|\cdots
& 0 \ar@{<-}[rr] &
& a\!\oplus\!b \ar@{<-}[rr]_{\left(\begin{smallmatrix} \alpha \\ \beta \end{smallmatrix}\right)} &
& s \ar@{<-}[rr] &
& 0 \ar@{}[r]|\cdots &
}
\end{equation}
Indeed, the $Q\op$ representations in the components of the complex~\eqref{eq:res-kQ*} capture the right $kQ$-modules structure of $X$, which is transferred to the left module structure of $\Hom_{kQ}(X,Y)$. Since the maps in the representation in degree $0$ of $X$ are split inclusions, the corresponding maps in degree $0$ of~\eqref{eq:Rhom-kQ^*} are split projections. Similarly, the maps in degree $1$ of~\eqref{eq:Rhom-kQ^*} are the identities on $s$ and the horizontal maps have components $\alpha,\beta,\gamma$ as depicted. If $Y$ is a complex of left $kQ$-modules rather than a $kQ$-module, then~\eqref{eq:Rhom-kQ^*} depicts a bicomplex of $kQ$-modules and $\Hom_{kQ}(X,Y) \in \Ch(kQ)$ is none other than its totalization. Comparing this to \autoref{fig:Z} and using that the totalization is the cofiber, we see that we precisely get a complex left $kQ$-modules which, in the notation of the figure, can be depicted as
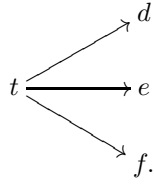
\[
\xymatrix@-.7pc{
&& d \\
t \ar[urr] \ar[rr] \ar[drr] && e \\
&& f.
}
\]
As this identification is clearly functorial, we have $\Rhom_{kQ}(X,-) \cong i^*(r^*)^2TF = i^*\tau\inv F$.

Thus, it remains to prove that $X$ is isomorphic to $\Sigma\inv kQ^*$ in the derived category of $kQ$-$kQ$-bimodules. To this end, recall that we have described the horizontal maps in \eqref{eq:res-kQ*} very explicitly. In particular we know that they are all inclusions, so that $X$ has non-zero homology only in degree $-1$. One also easily computes that the homology is of the form
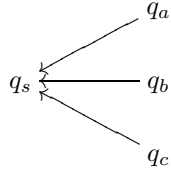
\begin{equation} \label{eq:kQ^*-as-rep}
\xymatrix@-.7pc{
&& q_a \ar[dll] \\
q_s && q_b \ar[ll] \\
&& q_c \ar[ull].
}
\end{equation}
where the left $kQ$-modules $q_s,q_a,q_b,q_c$ correspond to the representations
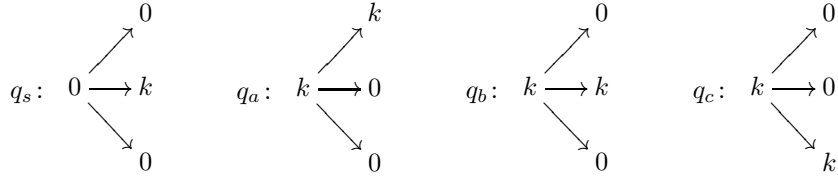
\[
\xymatrix@-.7pc{ {\phantom{0}} \\ {\phantom{k}q_s\colon} }
\xymatrix@-.7pc{
& 0 \\
0 \ar[ur] \ar[r] \ar[dr] & k \\
& 0
}
\quad\quad
\xymatrix@-.7pc{ {\phantom{k}} \\ {\phantom{k}q_a\colon} }
\xymatrix@-.7pc{
& k \\
k \ar[ur] \ar[r] \ar[dr] & 0 \\
& 0
}
\quad\quad
\xymatrix@-.7pc{ {\phantom{0}} \\ {\phantom{k}q_b\colon} }
\xymatrix@-.7pc{
& 0 \\
k \ar[ur] \ar[r] \ar[dr] & k \\
& 0
}
\quad\quad
\xymatrix@-.7pc{ {\phantom{0}} \\ {\phantom{k}q_c\colon} }
\xymatrix@-.7pc{
& 0 \\
k \ar[ur] \ar[r] \ar[dr] & 0 \\
& k
}
\]
and the maps in \eqref{eq:kQ^*-as-rep} can be identified with the obvious projections. Analogously to the description of indecomposable injective $kQ$-modules from~\cite[\S III.1, p.~54]{ARS:representation} in the case when $k$ is a field, for general $k$ we obtain a $k$-basis $e_s^*,e_a^*,e_b^*,e_c^*,\alpha^*,\beta^*,\gamma^*$ for $kQ^*$ with explicitly defined left and right $kQ$-actions. This allows us to compute that $q_s,q_a,q_b,q_c$ are isomorphic as left $kQ$-modules to $kQ^* \cdot e_s,kQ^* \cdot e_a,kQ^* \cdot e_b$ and $kQ^* \cdot e_c$, respectively, and that \eqref{eq:kQ^*-as-rep} is indeed the representation from $(\Mod kQ)^{Q\op}$ corresponding to $kQ^*$.
\end{proof}

Now that we have given a meaning to the Serre endomorphism $\lS\colon \D^{C,\mathrm{ex}} \to \D^{C,\mathrm{ex}}$ for some specific stable derivators $\D$, we will focus on the so-called fractionally Calabi--Yau property. To start with, let us collect the following theorem which now is a simple computation using our previous results.

\begin{thm} \label{thm:frac-CY}
Let $\D$ be a stable derivator, let $\D^{C,\mathrm{ex}}$ be the induced stable derivator from \autoref{notn:reinforced-cubes}, and let $\Sigma, \lS\colon \D^{C,\mathrm{ex}} \to \D^{C,\mathrm{ex}}$ be the suspension and the Serre endomorphism, respectively. Then we have a canonical isomorphism
\[ \lS^3 \cong \Sigma^2. \]
\end{thm}

\begin{proof}
By definition, $\lS = r^* \mathrm{sh}^*$. Since $r$ and $\mathrm{sh}$ commute and since $r^3 = \id_C$, we have $\lS^3 \cong (\mathrm{sh}^*)^3$ and further $(\mathrm{sh}^*)^3 \cong \Sigma^2$ by \autoref{prop:susp-shift}.
\end{proof} 

In order to put the above result into a wider context, suppose that $k$ is a field, $\cT$ is a $k$-linear triangulated category with finite dimensional $\nhom_\cT(-,-)$ and that $\lS\colon \cT \to \cT$ is a Serre functor defined via \eqref{eq:Serre}. Then $\cT$ is said to be \textbf{fractionally Calabi--Yau} if there are integers $m \in \lZ$ and $n > 0$ such that $\lS^n \cong \Sigma^m$ (see~\cite[\S8.2]{keller:orbit} or~\cite{vRoos12}). In such a case, $\cT$ is also said to have \textbf{Calabi--Yau dimension}~$\frac mn$. Note that it may happen that $m$ and $n$ are not coprime, but the common factor cannot be canceled~\cite[Example 5.3]{vRoos12}. Thus, the fraction $\frac mn$ is a priori only formal and written in the non-simplified form. However, given the motivation from geometry and also some formal properties, it makes sense to consider $\frac mn$ as an actual rational number---thus the notation.

In our case, suppose that $Q = [1]^3_{\leq1}$ and that $k$ is a field. Then $D^b(kQ)$ is well-known to be fractionally Calabi--Yau of dimension $\frac 23$ \cite[Example 8.3(2)]{keller:orbit}. A proof can be found in~\cite[Theorem 4.1]{mi-ye:picard-groups}. One can also wonder about a relative version of the statement when $k$ is an arbitrary commutative ring. It does not seem easy to generalize the proof from \cite{mi-ye:picard-groups} as it heavily depends on the fact that $kQ$ is a hereditary algebra, but it can be quickly deduced from what we have done so far. This is to indicate the potential of our methods.

\begin{thm} \label{thm:fracCYring}
Let $k$ be a commutative ring, $Q = [1]^3_{\leq1}$ be a standalone trivalent source, and let $\lS = kQ^* \Lotimes_{kQ} -\colon D(kQ) \to D(kQ)$ be the relative Serre functor as defined in~\cite{KeSch12}. Then $\lS^3 \cong \Sigma^2$. In other words, $D(kQ)$ is ``relative fractionally Calabi--Yau of dimension $\frac 23$''.
\end{thm}

\begin{proof}
We know from \autoref{prop:relate-Serre} that $kQ^* \Lotimes_{kQ} -$ is isomorphic to the base component $\lS_\bbone$ of the abstract Serre endomorphism $\lS\colon \D_k^{C,\mathrm{ex}} \to \D_k^{C,\mathrm{ex}}$. Hence the conclusion follows from \autoref{thm:frac-CY}.
\end{proof}

\subsection{Relation to Auslander--Reiten theory}
\label{subsec:AR}

Now we show how to organize the previous results with the help of Auslander--Reiten theory. This in particular involves a far-reaching generalization of the results by Keller and Neeman~\cite{kellerneeman:may} to the context of arbitrary stable derivators.

Let us first sketch the classical role of Auslander--Reiten theory. As mentioned above, if $k$ is a field and $Q$ is a finite quiver without oriented cycles, then $D^b(kQ)$ has a Serre functor $\lS$ and therefore also the \emph{Auslander--Reiten translation} $\tau = \Sigma\lS$ and so-called almost split triangles (also known as Auslander--Reiten triangles); see~\cite[\S3]{happel:fdalgebra} or~\cite[Section~I]{RvdB02}. An amazing consequence proved by Happel in~\cite[Proposition 4.6]{happel:fdalgebra} is that for a Dynkin quiver $Q$, which is in particular the case for our trivalent source $Q = [1]^3_{\leq1}$, one can explicitly describe $D^b(kQ)$ in terms of generators and relations. More precisely, every object of $D^b(kQ)$ decomposes as a direct sum of indecomposable objects and this decomposition is unique up to isomorphism and reordering. Therefore, it suffices to describe a skeleton $\ind D^b(kQ)$ of indecomposable objects in $D^b(kQ)$, and Happel gave an explicit presentation of $\ind D^b(kQ)$ as a quotient of a free category over a quiver, the so-called \emph{Auslander--Reiten quiver} of $\ind D^b(kQ)$, modulo certain relations.

The Auslander--Reiten quiver for $\ind D^b(kQ)$ for the trivalent source can be found in~\cite[\S5.7, p.~356]{happel:fdalgebra}, \cite[p.~560]{kellerneeman:may} or~\cite[Fig.~5.2, p.~98]{tilting}. A well-known fact in this case is that $\ind D^b(kQ)$ contains precisely $12$ suspension orbits, each of which is represented by a $kQ$-module. The latter follows from the description of indecomposable $kQ$-modules~\cite[\S3]{gabriel:unzerlegbar} and the fact that every object of $D^b(kQ)$ is isomorphic to its homology~\cite[Lemma~4.1]{happel:fdalgebra}. Of course, we will obtain the same Auslander--Reiten quiver and $12$ suspension orbits for $\ind D^b(kQ')$ if $Q' = [1]^3_{\geq2}$ is a trivalent sink, and also for $\ind D^b(k\square)$. This follows from the well-known triangle equivalences which we have generalized in \autoref{thm:tiltsquare}.

Starting with a stable derivator $\D$ and an object $X$ in $\D^Q$ with underlying diagram
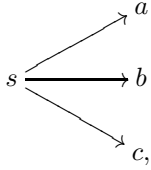
\[
\xymatrix@-.7pc{
&& a \\
s \ar[urr] \ar[rr] \ar[drr] && b \\
&& c,
}
\]
we have previously also constructed $12$ suspension orbits in the coherent diagrams \autoref{fig:reinforced-cubes-1} and \autoref{fig:reinforced-cubes-2}. The relation to Happel's results mentioned above is as follows. We consider $\D = \D_{kQ\op}$ and we take for $X \in \D^Q(\bbone)$ the $Q$ shaped diagram of projective right $kQ$-modules
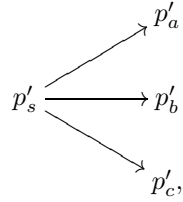
\[
\xymatrix@-.7pc{
&& p'_{a} \\
p'_{s} \ar[urr] \ar[rr] \ar[drr] && p'_b \\
&& p'_c,
}
\]
where $p'_v = e_v \cdot kQ$ stands for the projective module corresponding to vertex $v$. The maps between the projective modules are given by the left multiplication by the corresponding arrows of $Q$. In fact, we can identify $D^Q(\bbone) = D\big((\Mod kQ\op)^Q\big)$ with the derived category of $kQ$-$kQ$-bimodules as in the proof of \autoref{prop:relate-Serre}, and under this identification $X$ is none other than $kQ$ viewed as the bimodule over itself. If we construct \autoref{fig:reinforced-cubes-1} and \autoref{fig:reinforced-cubes-2} for $Y = F(X)$, where $F$ is from \autoref{prop:source-square-to-cubes}, the twelve objects $a,b,c,d,e,f,k,\ell,m,s,t,u \in \D(\bbone) = D(kQ\op)$ will be precisely the $12$ indecomposable right $kQ$-modules.

This all suggests that it should be possible to organize the $12$ objects in a simpler picture than our two chains of cubes in \autoref{fig:reinforced-cubes-1} and \autoref{fig:reinforced-cubes-2}. Namely, one can expect a diagram as in~\cite[\S5.7, p.~356]{happel:fdalgebra}, \cite[p. 560]{kellerneeman:may} or~\cite[Fig.~5.2, p.~98]{tilting}. This is indeed possible, but we are forced to construct such an \emph{incoherent} diagram for a reason explained in \autoref{rmk:anticommute}. Of course, in view of \autoref{thm:tiltsquare}, we can do the same when we start with a coherent trivalent sink or commutative square in a stable derivator.

\begin{thm} \label{thm:ar-quiver-d4}
Let $\D$ be a stable derivator and $X \in \D([1]^3_{\leq 1})$ be an object with underlying diagram
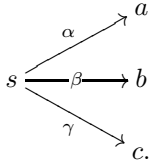
\[
\xymatrix@-.7pc{
&& a \\
s \ar[urr]^{\alpha} \ar[rr]|{\hole}|{\beta} \ar[drr]_{\gamma} && b \\
&& c.
}
\]
Then the data from the coherent diagrams obtained in \autoref{fig:reinforced-cubes-1} and \autoref{fig:reinforced-cubes-2} give rise to the following infinite incoherent diagram in $\D(\bbone)$ of the shape
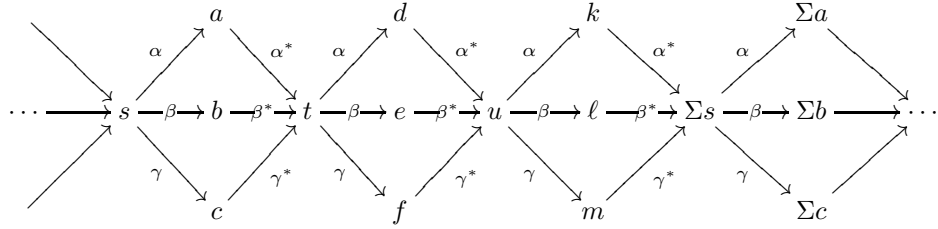
\begin{equation} \label{eq:ar-quiver-d4}
\vcenter{
\xymatrix{
\ar[dr] && a \ar[dr]^{\alpha^*} && d \ar[dr]^{\alpha^*} && k \ar[dr]^{\alpha^*} && \Sigma a \ar[dr]
\\
\cdots \ar[r] &
s \ar[ur]^{\alpha} \ar[r]|{\hole}|{\beta} \ar[dr]_{\gamma} & b \ar[r]|{\hole}|{\beta^*} &
t \ar[ur]^{\alpha} \ar[r]|{\hole}|{\beta} \ar[dr]_{\gamma} & e \ar[r]|{\hole}|{\beta^*} &
u \ar[ur]^{\alpha} \ar[r]|{\hole}|{\beta} \ar[dr]_{\gamma} & \ell \ar[r]|{\hole}|{\beta^*} &
\Sigma s \ar[ur]^{\alpha} \ar[r]|{\hole}|{\beta} \ar[dr]_{\gamma} & \Sigma b \ar[r] & \cdots
\\
\ar[ur] && c \ar[ur]_{\gamma^*} && f \ar[ur]_{\gamma^*} && m \ar[ur]_{\gamma^*} && \Sigma c \ar[ur]
}
}
\end{equation}
and the morphisms satisfy the relations
\[ \alpha^* \alpha + \beta^*\beta + \gamma^*\gamma = 0 \]
and also
\[\alpha\alpha^* = 0, \quad \beta\beta^* = 0, \textrm{ and} \quad \gamma\gamma^* = 0\]
which all come from cofiber sequences.
\end{thm}

\begin{rmk} \label{rmk:mesh-rel}
Here we abuse notation in that $\alpha,\beta,\gamma, \alpha^* ,\beta^* ,\gamma^*$ stand for different maps in the different positions in the diagram above. We have chosen this notation to be able to easily state the relations, which are usually called \textbf{mesh relations} (cf.\ \cite[\S~VII.1, p.~232]{ARS:representation}) in representation theory. We also refer to \cite[\S4.6]{happel:fdalgebra}, \cite[\S6.5]{gabriel:ar-seq} and \cite{riedtmann:koecher} for more background.
\end{rmk}

\begin{proof}[Proof of \autoref{thm:ar-quiver-d4}]
This just amounts to comparing the incoherent diagrams of the two objects in \autoref{fig:reinforced-cubes-1} and \autoref{fig:reinforced-cubes-2} produced from $X$ in \autoref{prop:source-square-to-cubes} and by means of \eqref{eq:H}, considering the bicartesian squares from \autoref{cor:sqrt-shift}, and bookkeeping signs using the Mayer--Vietoris theorem for stable derivators (\cite[Theorem~6.1]{gps:mayer} or \cite[Proposition~2.4]{stpo:co-t-str}).
\end{proof}

If one considers the Auslander--Reiten quiver of \autoref{thm:ar-quiver-d4} in the case of $D^b(kQ\op)$ for a field~$k$, then it is classical that the Auslander--Reiten translation is given by shifting one position to the left. The picture thus suggests that a threefold application of this translation yields the loop functor~$\Omega$. The following corollary gives us a coherent formulation of this observation.

\begin{cor}\label{cor:loop}
For an arbitrary stable derivator $\D$ there is a canonical isomorphism
\[
\tau^3\cong\Omega\colon\D^{C,\mathrm{ex}}\to\D^{C,\mathrm{ex}}.
\]
\end{cor}
\begin{proof}
This follows immediately from the calculation $\tau^3\cong T^{-3}\cong \Sigma^{-1}=\Omega$ which uses \autoref{defn:ARS} and \autoref{prop:sqrt-shift}.
\end{proof}

\begin{rmk} \label{rmk:anticommute}
The reader can readily see where all the bicartesian cubes, bicartesian squares, and cofiber sequences (hence triangles) from \autoref{fig:reinforced-cubes-1}, \autoref{fig:reinforced-cubes-2} and \autoref{fig:Z} are located in the diagram in \autoref{thm:ar-quiver-d4} in terms of objects. As for the morphisms, they are the same as in the coherent diagrams \emph{up to signs}. One immediately sees that the mesh relations imply that the small cubes from \autoref{fig:reinforced-cubes-1} and \autoref{fig:reinforced-cubes-2} \emph{anticommute} rather than commute in the diagram in \autoref{thm:ar-quiver-d4}. The limited ability to presently treat signs in coherent diagrams systematically unfortunately forces us to
\begin{itemize}
\item either split the whole picture into several coherent diagrams,
\item or pass to an incoherent diagram as in \autoref{thm:ar-quiver-d4}.
\end{itemize}
\end{rmk}

\begin{rmk} \label{rmk:KK-theory}
The same diagram as above has been rediscovered in situations where one might not have expected it---see~\cite[\S5]{meyer-nest:filtrated-k-th} in connection with $KK$-theory of C*-algebras.
\end{rmk}

\section{May's axioms via abstract representation theory}
\label{sec:May}

Given a triangulated category with an exact closed symmetric monoidal structure, any two distinguished triangles 
\[
x\to y\to z\to \Sigma x\qquad\text{and}\qquad x'\to y'\to z'\to \Sigma x'
\]
yield a plethora of additional distinguished triangles and these can be nicely organized in star-shaped diagrams (as in \autoref{fig:maytc3} with the exception of the object $v$ and the morphisms adjacent to it). In this context May's axioms~\cite{may:traces} ask for the existence of certain non-canonical additional objects and morphisms which then allow him to establish an abstract additivity result for Euler characteristics. Two seemingly different perspectives on these axioms are offered in the papers~\cite{kellerneeman:may} and~\cite{gps:additivity}.
\begin{enumerate}
\item It has been recognized in~\cite{kellerneeman:may} that there is a close relation to the fact that the commutative square and the trivalent source are strongly stably equivalent (although the authors of~\cite{kellerneeman:may} did not study the problem at this generality and mainly considered algebraic triangulated categories over a field~$k$). 
\item The paper~\cite{gps:additivity} is a formal study of the interaction of stability and monoidal structure, and therein the axioms are shown to simply capture certain naturality properties of the calculus of (homotopy) tensor products. In particular, coherent versions of certain diagrams showing up in May's axioms are constructed in that paper.
\end{enumerate}
Our constructions in \S\ref{sec:square} allow us to relate these two perspectives on the braid and the additivity axioms of May \cite{may:traces}, and to offer short proofs showing that these axioms are formal consequences of stability alone. 

A particular source of coherent commutative squares in homotopy theory is the following situation: If $\D$ is a monoidal derivator with monoidal structure 
\[
\otimes\colon\D\times\D\to\D
\]
(see~\cite[\S7]{gps:additivity}), then the external tensor product of coherent maps $X,Y\in \D^{[1]}$ is an object $X \otimes Y\in\D^\square$. We emphasize that, by definition of a monoidal derivator, the monoidal structure is supposed to preserve colimits in each variable independently. In particular, if we have a stable, monoidal derivator, then the monoidal structure is bi-exact.

In~\cite{gps:additivity} it is shown that May's axioms are satisfied in the context of an arbitrary stable, closed symmetric monoidal derivator. However, as indicated in~\cite[Remark 3.6]{kellerneeman:may}, some of the axioms follow easily from statements that are a formal consequence of stability
alone and do not depend on monoidal structure -- instead, as we will see, it suffices to start with an arbitrary coherent square in a stable derivator. The advantage of the representation theoretic perspective stressed in~\cite{kellerneeman:may} is that Auslander--Reiten quivers offer a different organization of these diagrams. Moreover, our diagrams \autoref{fig:reinforced-cubes-1} and \autoref{fig:reinforced-cubes-2} specialize to two extended and completely coherent versions of May's braid axioms (see \autoref{fig:tc3} and \autoref{fig:tc3'}).

To begin with, recall from \autoref{notn:reinforced-cubes} the definition of the category $C$ and $\D^{C,\mathrm{ex}}$. By \autoref{prop:source-square-to-cubes} for every stable derivator \D there is an equivalence of derivators $(G,j^\ast)\colon\D^\square\rightleftarrows\D^{C,\mathrm{ex}}$. The functor $G$ sends a square with underlying diagram 
\begin{equation}
\vcenter{
\xymatrix{
b\ar[r]\ar[d]&f\ar[d]\\
d\ar[r]&l
}
}
\label{eq:gensquare}
\end{equation}
to a coherent diagram as in \autoref{fig:reinforced-cubes-1}. 

Before we use this to describe our new perspective on the first braid axiom of May, let us recall from \cite{gps:additivity} that monoidal structures on derivators admit different equivalent reformulations, giving rise to internal, external, or canceling variants. In particular, given two small categories $A,B$ and a monoidal derivator $\D$ with monoidal structure $\otimes\colon\D\times\D\to\D$ there are (suitably associative and unital) external product operations of the form
\[
\otimes\colon\D^A\times\D^B\to \D^{A\times B}.
\] 
Specializing this to the case of $A=B=[1]$ there is the following result.

\begin{prop}\label{prop:tc3}
Let $\D$ be a stable monoidal derivator with monoidal structure $\otimes\colon\D\times\D\to\D$. The composition
\[
G\circ\otimes\colon\D^{[1]}\times\D^{[1]}\to\D^\square\to\D^{C,\mathrm{ex}}
\]
sends coherent diagrams $X=(x\to y),X'=(x'\to y')\in\D^{[1]}$ to a coherent diagram in $\D^{C,\mathrm{ex}}$ with underlying diagram as in \autoref{fig:tc3}.
\end{prop}
\begin{proof}
This is easily seen to follow from the bi-exactness of the monoidal structure and the defining exactness properties of $\D^{C,\mathrm{ex}}$.
\end{proof}

\begin{figure}
  \centering
   \begin{equation}\label{eq:tc3}
    \xymatrix@R=3pc@C=5pc{
      \Sigma^{-1}(y\otimes z')
      \ar[dr]
      \ar@(dl,ul)[d] &
      x\otimes x' \ar[dl] \ar[dr] &
      \Sigma^{-1}(z\otimes y')
      \ar[dl]
      \ar@(dr,ur)[d] \\
      y\otimes x'
      \ar@(dl,ul)[dd]
      \ar[ddr]
      \ar[dr]^(.4){p_1} &
      \Sigma^{-1}(z\otimes z')
      \ar[ddl]
      \ar[d]^{p_2}
      \ar[ddr] &
      x\otimes y'
      \ar[dl]_(.4){p_3}
      \ar[ddl]
      \ar@(dr,ur)[dd] \\
      & v \ar[dl]^(.6){j_3} \ar[d]^{j_2} \ar[dr]_(.6){j_1} \\
      z\otimes x' \ar@(dl,ul)[d] \ar[dr] &
      y\otimes y' \ar[dl] \ar[dr] &
      x\otimes z' \ar[dl] \ar@(dr,ur)[d] \\
      z\otimes y' &
      \Sigma(x\otimes x') &
      y\otimes z'
    }
  \end{equation}
  \caption{The (TC3) diagram from~\cite{may:traces}}
  \label{fig:maytc3}
\end{figure}

Let now $B$ be the poset given by the shape of the axiom (TC3) introduced by May in \cite{may:traces} and depicted in \autoref{fig:maytc3}.  The decoration of the objects in \autoref{fig:maytc3} and in \autoref{fig:tc3} uniquely defines a functor $k\colon B\to C$. If we apply the composition
\[
k^\ast\circ G\circ\otimes\colon\D^{[1]}\times\D^{[1]}\to\D^\square\to\D^{C,\mathrm{ex}}\to\D^B
\]
to $(X,X')$ then we obtain a coherent diagram whose underlying diagram looks objectwise like \autoref{fig:maytc3} (we do not want to get involved in a discussion of the morphisms which involves signs and instead refer the reader to \cite[Theorem~6.2]{gps:additivity}). The exactness properties of the coherent diagram $G(X\otimes X')$ imply that there are many classical distinguished triangles associated to $k^\ast G(X\otimes X')$. In fact, there are six distinguished triangles coming from the braids and three more triangles which come from the `reinforced link' given by~$v$, i.e., triangles which begin with the morphisms $j_i$ and $p_i$ for $i=1,2,3$. Moreover, in $k^\ast G(X\otimes X')$ the six squares having $v$ as corner are bicartesian, which by the Mayer--Vietoris theorem for stable derivators (\cite[Theorem~6.1]{gps:mayer} or \cite[Proposition~2.4]{stpo:co-t-str}) leads to six additional distinguished triangles. A coherent formulation of some of these exactness properties is not possible anymore once we pass to the star-shaped diagram \autoref{fig:maytc3} since we are lacking all the zero objects, and we loose even more information if we pass to the incoherent diagram in the underlying triangulated category. However, all these exactness properties are conveniently encoded in the derivator $\D^{C,\mathrm{ex}}$ (see again \autoref{fig:tc3}). In fact, that figure encodes a coherent version of May's braid axiom (TC3) which is extended infinitely in the positive and the negative direction. Let us emphasize once more that $G$ associates such a diagram to an \emph{arbitrary} coherent square as in \eqref{eq:gensquare}.

Let us now turn to our new perspective on the second braid axiom of May, the (TC3') axiom, which is depicted in \autoref{fig:maytc3'}. This diagram is obtained by rotating back the two original distinguished triangles to get
\[
\Sigma^{-1}z\to x\to y\to z\qquad\text{and}\qquad \Sigma^{-1} z'\to x'\to y'\to z',
\]
and then passing to the suspension of the associated (TC3)-diagram. We will now show that at the level of $\D^{C,\mathrm{ex}}$ the passage from (TC3) to (TC3') amounts to an application of the inverse Auslander--Reiten translation $\tau^{-1}$ (see \autoref{cor:TC3'}).

The identification $\square=[1]\times[1]$ has some implications for shifted derivators. More generally, given two small categories $A,B$ and a derivator $\D$, there are canonical isomorphisms of shifted derivators
\[
(\D^A)^B\cong \D^{A\times B}\cong (\D^B)^A.
\]
In our situation of the square, this means that there are two different identifications of $\D^\square$ with $(\D^{[1]})^{[1]}$. In particular, this implies that the cofiber morphism $\cof\colon\D^{[1]}\to\D^{[1]}$ induces two morphisms
\[
\cof_1\colon\D^\square\to\D^\square\qquad\text{and}\qquad\cof_2\colon\D^\square\to\D^\square,
\]
which take a coherent square and pass to the cofiber in the first or second direction, respectively. These functors behave as follows with respect to $G\colon\D^\square\to\D^{C,\mathrm{ex}}$.

\begin{figure}
  \centering
   \begin{equation}\label{eq:tc3'}
    \xymatrix@R=3pc@C=5pc{
      x\otimes y'
      \ar[dr]
      \ar@(dl,ul)[d] &
      \Sigma^{-1}(z\otimes z') \ar[dl] \ar[dr] &
      y\otimes x'
      \ar[dl]
      \ar@(dr,ur)[d] \\
      x\otimes z'
      \ar@(dl,ul)[dd]
      \ar[ddr]
      \ar[dr]^(.4){k_1} &
      y\otimes y'
      \ar[ddl]
      \ar[d]^{k_2}
      \ar[ddr] &
      z\otimes x'
      \ar[dl]_(.4){p_3}
      \ar[ddl]
      \ar@(dr,ur)[dd] \\
      & w \ar[dl]^(.6){q_3} \ar[d]^{q_2} \ar[dr]_(.6){q_1} \\
      y\otimes z' \ar@(dl,ul)[d] \ar[dr] &
      \Sigma(x\otimes x') \ar[dl] \ar[dr] &
      z\otimes y' \ar[dl] \ar@(dr,ur)[d] \\
      \Sigma(y\otimes x') &
      z\otimes z' &
      \Sigma(x\otimes y')
    }
  \end{equation}
  \caption{The (TC3') diagram from~\cite{may:traces}}
  \label{fig:maytc3'}
\end{figure}

\begin{prop}\label{prop:cofcof}
Let \D be a stable derivator. There are canonical isomorphisms
\[
G\circ\cof_1\cong r^\ast\circ\tau^{-1}\circ G\qquad\text{and}\qquad G\circ \cof_2\cong (r^\ast)^{-1}\circ\tau^{-1}\circ G
\]
of morphisms $\D^\square\to\D^{C,\mathrm{ex}}$.
\end{prop}
\begin{proof}
We begin by the statement about $\cof_1$ and consider an arbitrary coherent square $Q\in\D^\square$ as in \eqref{eq:gensquare}. The equivalence $G\colon\D^\square\to\D^{C,\mathrm{ex}}$ sends $Q$ to a coherent diagram as in \autoref{fig:reinforced-cubes-1}. We claim that $\cof_1(Q)$ is given by the (not necessarily bicartesian) square
\begin{equation}\label{eq:cof1}
\vcenter{
\xymatrix{
f\ar[r]\ar[d]&k\ar[d]\\
l\ar[r]&\Sigma c
}
}
\end{equation}
which can be found in that figure. To this end, we consider $q\colon\square\times[1]\to C$ defined by
\begin{enumerate}
\item $q(00,0)=010$, $q(10,0)=110$, $q(01,0)=021$, $q(11,0)=121$ and
\item $q(00,1)=011$, $q(10,1)=112$, $q(01,1)=021$, $q(11,1)=122$.
\end{enumerate}
The functor $q$ is checked to be well-defined, and the restrictions $q^\ast(G(Q))_{(-,0)}$ and $q^\ast(G(Q))_{(-,1)}$ respectively look like
\[
\xymatrix{
b\ar[r]\ar[d]&f\ar[d]&&d\ar[r]\ar[d]&l\ar[d]\\
0\ar[r]&k,&&0\ar[r]&\Sigma c.
}
\]
Moreover, since we know that $G(Q)$ lies in $\D^{C,\mathrm{ex}}$ and hence satisfies the corresponding exactness properties, it follows that these two squares are bicartesian. It further follows from \cite[Corollary~3.14]{groth:ptstab} that $q^\ast(G(Q))$ considered as an object of $(\D^{[1]})^\square$ is bicartesian, and, taking also into account the vanishing conditions, we see that $\cof_1(Q)$ is canonically isomorphic to \eqref{eq:cof1}. On the other hand, we showed in the proof of \autoref{prop:susp-shift} that the square in \autoref{fig:reinforced-cubes-1} with vertices $\Sigma b,\Sigma f,\Sigma d,$ and $\Sigma\ell$ is canonically isomorphic to the suspension of \eqref{eq:gensquare} and, by definition of $\mathrm{sh}\colon C \to C$, we also obtain a canonical identification of \eqref{eq:cof1} with $j^*\circ(\mathrm{sh}^\ast)\inv \circ \Sigma \circ G(Q)$, where $j^*$ is the inverse of $G$ from \autoref{prop:source-square-to-cubes}. Thus there is a canonical isomorphism $G(\cof_1(Q))\cong T(G(Q))$ by \autoref{prop:sqrt-shift}, which together with \autoref{defn:ARS} concludes the proof of this first case.

The proof of the second claim is essentially the same, and we hence leave most of the details to the reader. In that case we show that $\cof_2(Q)$ is canonically isomorphic to the (not necessarily bicartesian) square
\[
\xymatrix{
d\ar[r]\ar[d]&l\ar[d]\\
m\ar[r]&\Sigma a,
}
\]
which can be found in $G(Q)$ (see \autoref{fig:reinforced-cubes-1}). To achieve this, we consider this time the functor $q'\colon[1]\times\square\to C$ defined by
\begin{enumerate}
\item $q'(0,00)=010$, $q'(0,10)=011$, $q'(0,01)=210$, $q'(0,11)=211$ and
\item $q'(1,00)=110$, $q'(1,10)=112$, $q'(1,01)=210$, $q'(1,11)=212$.
\end{enumerate}
From this, we obtain a canonical isomorphism $G(\cof_2(Q))\cong r^\ast T(G(Q))$ which allows us to conclude.
\end{proof}

This allows for the following different interpretation of the Serre functor and the (inverse of the) Auslander--Reiten translation, which, for convenience, is formulated in the monoidal context. 

\begin{cor}\label{cor:cofcoffibfib}
Let \D be a stable, monoidal derivator with monoidal structure $\otimes\colon\D\times\D\to\D$ and let $X,Y\in\D^{[1]}$.
\begin{enumerate}
\item There is a canonical isomorphism $G(\cof X\otimes\cof Y)\cong \lS\circ G(X\otimes Y)$.
\item There is a canonical isomorphism $G\circ\Sigma(\fib X\otimes\fib Y)\cong \tau^{-1}\circ G(X\otimes Y)$.
\end{enumerate}
\end{cor}
\begin{proof}
The first case follows from the following chain of canonical isomorphisms,
\begin{align}
G(\cof X\otimes\cof Y)&\cong G\circ \cof_1 (X\otimes\cof Y)\\
&\cong G\circ \cof_1\circ \cof_2 (X\otimes Y)\\
&\cong \tau^{-2}\circ G(X\otimes Y)\\
&\cong \tau\circ\Sigma\circ G(X\otimes Y)\\
&\cong \lS \circ G(X\otimes Y).
\end{align}
The canonical isomorphisms respectively result from the bi-exactness of $\otimes$ in a stable derivator in the first two cases, from \autoref{prop:cofcof} and \autoref{cor:loop} in the next two steps, and from \autoref{prop:sqrt-shift} and \autoref{defn:ARS} in the final case.

The second statement can be derived from this one by observing that it yields a canonical isomorphism
$G(\fib X\otimes \fib Y)\cong \lS^{-1} \circ G (X\otimes Y)$. An application of $\Sigma$ implies
\[
G\circ \Sigma (\fib X\otimes \fib Y)\cong \Sigma\circ G(\fib X\otimes \fib Y)\cong \Sigma \circ \lS^{-1} \circ G (X\otimes Y),
\]
and we conclude using the equivalence $\tau\circ \Sigma\cong\lS$ as in the previous case.
\end{proof}

\begin{cor}\label{cor:TC3'}
Let $\D$ be a stable monoidal derivator with monoidal structure $\otimes\colon\D\times\D\to\D$. The composition
\[
\tau^{-1}\circ G\circ\otimes\colon\D^{[1]}\times\D^{[1]}\to\D^\square\to\D^{C,\mathrm{ex}}
\]
sends coherent arrows $X=(x\to y),X'=(x'\to y')\in\D^{[1]}$ to a coherent diagram in $\D^{C,\mathrm{ex}}$ with underlying diagram as in \autoref{fig:tc3'}.
\end{cor}
\begin{proof}
This is immediate from \autoref{prop:tc3} and \autoref{cor:cofcoffibfib} together with the remark that the (TC3') diagram is the suspension of the (TC3) diagram applied to the fibers of $X$ and $X'$.
\end{proof}

Thus, combining the construction described in this corollary with a suitable restriction functor as in the case of (TC3) gives us a recipe for constructing coherent diagrams as in (TC3'). Again, all exactness properties are conveniently encoded in \autoref{fig:tc3'}. At the level of $\D^{C,\mathrm{ex}}$ the inverse Auslander--Reiten translation $\tau^{-1}$ invokes a refined version of passing from a (TC3) diagram to a (TC3') diagram.

\begin{rmk}
We also have a coherent analogue of May's additivity axiom (TC4). Note that since we can pass from \autoref{fig:tc3} to \autoref{fig:tc3'} by $\tau\inv$, \autoref{cor:sqrt-shift} applies. In fact, inspecting its proof and \autoref{fig:Z}, we obtain a coherent triangle of the form
\[
\xymatrix{
v \ar[r] \ar[d] & (x\otimes z') \oplus (y\otimes y') \oplus (z\otimes x') \ar[r] \ar[d] & 0 \ar[d] \\
0 \ar[r] & w \ar[r] & \Sigma v.
}
\]
Thanks to the Mayer--Vietoris theorem for stable derivators (\cite[Theorem~6.1]{gps:mayer} or \cite[Proposition~2.4]{stpo:co-t-str}) we obtain a bicartesian square with underlying diagram of the form
\[
\xymatrix{
v \ar[r] \ar[d] & y\otimes y' \ar[d] \\
(x\otimes z') \oplus (z\otimes x') \ar[r] & w.
}
\]
\end{rmk}


\appendix

\bibliographystyle{alpha}
\bibliography{tilting}

\begin{landscape}
\begin{figure}
\centering
\xymatrix@-0.8pc{
\ar@{}[r]|\cdots & 0 \ar[drr] &&&&&& 0 \ar[drr] &&&&&& 0 \ar[drr] &&&&&& 0 \ar[drr]
\\
&&& a \ar[drr] \ar[rr] && f \ar[drr] \ar[urr] &
&&& m \ar[drr] \ar[rr] && \Sigma b \ar[drr] \ar[urr] &
&&& \Sigma e \ar[drr] \ar[rr] && \Sigma k \ar[drr] \ar[urr] &
&&& \Sigma^2 a \ar@{}[r]|\cdots & ~
\\
\ar@{}[r]|\cdots &
s \ar[urr] \ar@{.>}[rr] \ar[drr] &&
b \ar@{.>}[urr]|\hole \ar@{.>}[drr]|\hole &&
e \ar[rr] \ar[ddrr] &&
u \ar[urr] \ar@{.>}[rr] \ar[drr] &&
k \ar@{.>}[urr]|\hole \ar@{.>}[drr]|\hole &&
\Sigma a \ar[rr] \ar[ddrr] &&
\Sigma t \ar[urr] \ar@{.>}[rr] \ar[drr] &&
\Sigma f \ar@{.>}[urr]|\hole \ar@{.>}[drr]|\hole &&
\Sigma m \ar[rr] \ar[ddrr] &&
\Sigma^2 s \ar[urr] \ar@{.>}[rr] \ar[drr] &&
\Sigma^2 b \ar@{}[r]|\cdots & ~
\\
\ar@{}[r]|\cdots & 0 \ar@{.>}[urr]|\hole &&
c \ar[urr] \ar[rr] &&
d \ar@{-}+R!(0,1);+(4.8,2.8)\ar@{.>}+R!(5.1,3.95);+(15,8.75) \ar@{-}+R;+(7.5,0)\ar@{.>}+R!(8.3,0);+(15,0) &&
0 \ar@{.>}[urr]|\hole &&
\ell \ar[urr] \ar[rr] &&
\Sigma c \ar@{-}+R!(0,1);+(5.65,2.5)\ar@{.>}+R!(5.0,4.0);+(16,8.75) \ar@{-}+R;+(8,0)\ar@{.>}+R!(8.3,0);+(16,0) &&
0 \ar@{.>}[urr]|\hole &&
\Sigma d \ar[urr] \ar[rr] &&
\Sigma \ell \ar@{-}+R!(0,1);+(5.65,2.5)\ar@{.>}+R!(5.0,4.0);+(16,8.75) \ar@{-}+R;+(8.2,0)\ar@{.>}+R!(8.5,0);+(16,0) &&
0 \ar@{.>}[urr]|\hole &&
\Sigma^2 c \ar@{}[r]|\cdots & ~
\\
\ar@{}[r]|\cdots & 0 \ar[urr] &&&&&& 0 \ar[urr] &&&&&& 0 \ar[urr] &&&&&& 0 \ar[urr]
}
\caption{The coherent diagram $Y$ in $\D^{C,\mathrm{ex}}$.}
\label{fig:reinforced-cubes-1}
\end{figure}

\begin{figure}
\centering
\xymatrix@-0.8pc{
\ar@{}[r]|\cdots & 0 \ar[drr] &&&&&& 0 \ar[drr] &&&&&& 0 \ar[drr] &&&&&& 0 \ar[drr]
\\
&&& e \ar[drr] \ar[rr] && k \ar[drr] \ar[urr] &
&&& \Sigma a \ar[drr] \ar[rr] && \Sigma f \ar[drr] \ar[urr] &
&&& \Sigma m \ar[drr] \ar[rr] && \Sigma^2 b \ar[drr] \ar[urr] &
&&& \Sigma^2 e \ar@{}[r]|\cdots & ~
\\
\ar@{}[r]|\cdots &
t \ar[urr] \ar@{.>}[rr] \ar[drr] &&
f \ar@{.>}[urr]|\hole \ar@{.>}[drr]|\hole &&
m \ar[rr] \ar[ddrr] &&
\Sigma s \ar[urr] \ar@{.>}[rr] \ar[drr] &&
\Sigma b \ar@{.>}[urr]|\hole \ar@{.>}[drr]|\hole &&
\Sigma e \ar[rr] \ar[ddrr] &&
\Sigma u \ar[urr] \ar@{.>}[rr] \ar[drr] &&
\Sigma k \ar@{.>}[urr]|\hole \ar@{.>}[drr]|\hole &&
\Sigma^2 a \ar[rr] \ar[ddrr] &&
\Sigma^2 t \ar[urr] \ar@{.>}[rr] \ar[drr] &&
\Sigma^2 f \ar@{}[r]|\cdots & ~
\\
\ar@{}[r]|\cdots & 0 \ar@{.>}[urr]|\hole &&
d \ar[urr] \ar[rr] &&
\ell \ar@{-}+R!(0,1);+(5,2.8)\ar@{.>}+R!(5.3,4.2);+(15,8.75) \ar@{-}+R;+(7.5,0)\ar@{.>}+R!(8.3,0);+(15,0) &&
0 \ar@{.>}[urr]|\hole &&
\Sigma c \ar[urr] \ar[rr] &&
\Sigma d \ar@{-}+R!(0,1);+(5.6,2.4)\ar@{.>}+R!(5.0,4.0);+(16,8.75) \ar@{-}+R;+(8,0)\ar@{.>}+R!(8.3,0);+(16,0) &&
0 \ar@{.>}[urr]|\hole &&
\Sigma\ell \ar[urr] \ar[rr] &&
\Sigma^2 c \ar@{-}+R!(0,1);+(6.1,2.4)\ar@{.>}+R!(4.75,3.9);+(16,8.75) \ar@{-}+R;+(8.2,0)\ar@{.>}+R!(8.3,0);+(16,0) &&
0 \ar@{.>}[urr]|\hole &&
\Sigma^2 d \ar@{}[r]|\cdots & ~
\\
\ar@{}[r]|\cdots & 0 \ar[urr] &&&&&& 0 \ar[urr] &&&&&& 0 \ar[urr] &&&&&& 0 \ar[urr]
}
\caption{The coherent diagram $T(Y)$ in $\D^{C,\mathrm{ex}}$.}
\label{fig:reinforced-cubes-2}
\end{figure}

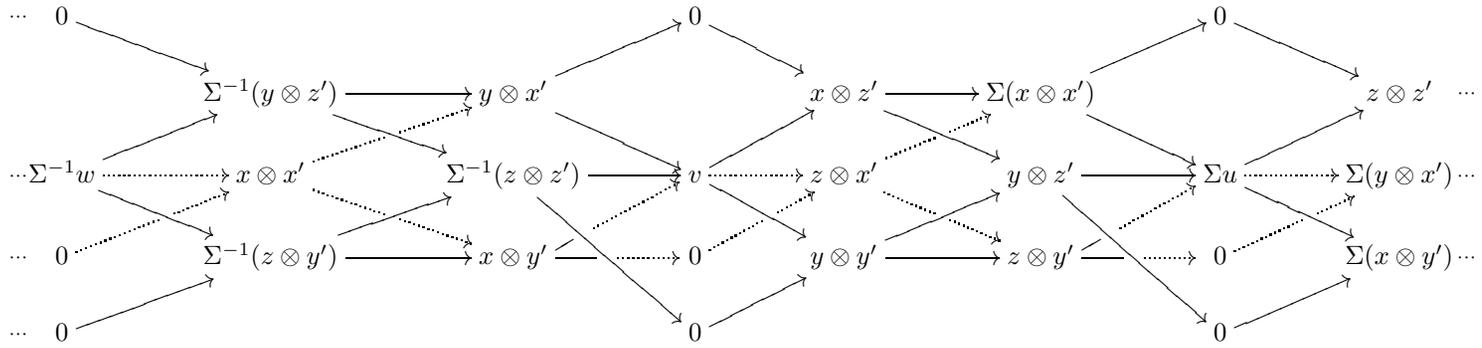
\begin{figure}
\centering
\xymatrix@-0.8pc{
\ar@{}[r]|\cdots & 0 \ar[drr] &&&&&& 0 \ar[drr] &&&&&& 0 \ar[drr]
\\
&&& \Sigma\inv(y\otimes z') \ar[drr] \ar[rr] &&
y\otimes x' \ar[drr] \ar[urr] &
&&& x\otimes z' \ar[drr] \ar[rr] &&
\Sigma(x\otimes x') \ar[drr] \ar[urr] &
&&& z \otimes z' \ar@{}[r]|{\quad\cdots} &
\\
\ar@{}[r]|\cdots &
\Sigma\inv w \ar[urr] \ar@{.>}[rr] \ar[drr] &&
x\otimes x' \ar@{.>}[urr]|\hole \ar@{.>}[drr]|\hole &&
\Sigma\inv(z\otimes z') \ar[rr] \ar[ddrr] &&
v \ar[urr] \ar@{.>}[rr] \ar[drr] &&
z\otimes x' \ar@{.>}[urr]|\hole \ar@{.>}[drr]|\hole &&
y\otimes z' \ar[rr] \ar[ddrr] &&
\Sigma u \ar[urr] \ar@{.>}[rr] \ar[drr] &&
\Sigma(y \otimes x') \ar@{}[r]|{\quad\cdots} &
\\
\ar@{}[r]|\cdots & 0 \ar@{.>}[urr]|\hole &&
\Sigma\inv(z\otimes y') \ar[urr] \ar[rr] &&
x\otimes y' \ar@{-}+R!(0,1);+(8,2.4)\ar@{.>}+R!(5.1,3.95);+(22,10) \ar@{-}+R;+(11,0)\ar@{.>}+R!(8.3,0);+(22,0) &&
0 \ar@{.>}[urr]|\hole &&
y\otimes y' \ar[urr] \ar[rr] &&
z\otimes y' \ar@{-}+R!(0,1);+(8,2.45)\ar@{.>}+R!(5.1,3.95);+(20.5,9.5) \ar@{-}+R;+(11,0)\ar@{.>}+R!(8.3,0);+(20.5,0) &&
0 \ar@{.>}[urr]|\hole &&
\Sigma(x \otimes y') \ar@{}[r]|{\quad\cdots} &
\\
\ar@{}[r]|\cdots & 0 \ar[urr] &&&&&& 0 \ar[urr] &&&&&& 0 \ar[urr]
}
\caption{The extended coherent version of May's axiom (TC3).}
\label{fig:tc3}
\end{figure}

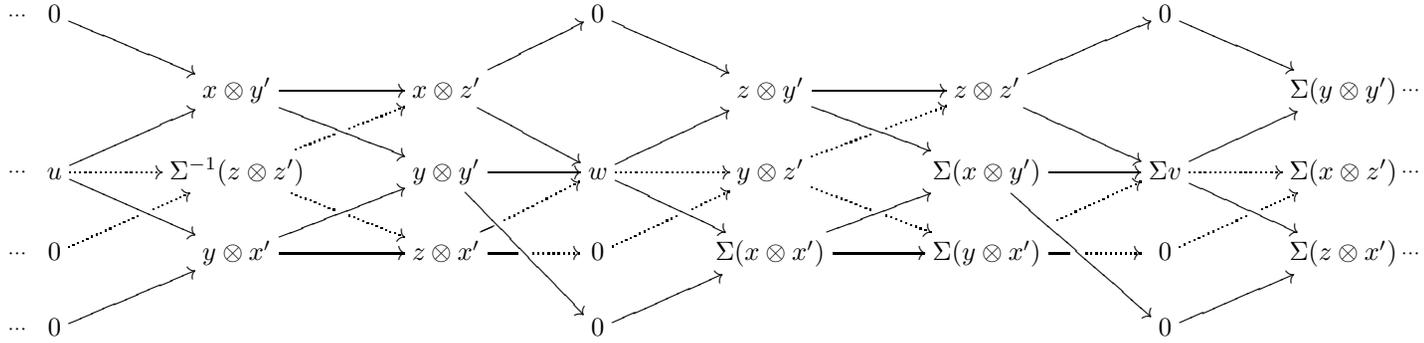
\begin{figure}
\centering
\xymatrix@-0.8pc{
\;\;\;\,\ar@{}[r]|\cdots & 0 \ar[drr] &&&&&& 0 \ar[drr] &&&&&& 0 \ar[drr]
\\
&&& x\otimes y' \ar[drr] \ar[rr] &&
x \otimes z' \ar[drr] \ar[urr] &
&&& z\otimes y' \ar[drr] \ar[rr] &&
z\otimes z' \ar[drr] \ar[urr] &
&&& \Sigma(y \otimes y') \ar@{}[r]|{\quad\cdots} &
\\
\ar@{}[r]|\cdots &
u \ar[urr] \ar@{.>}[rr] \ar[drr] &&
\Sigma\inv(z\otimes z') \ar@{.>}[urr]|\hole \ar@{.>}[drr]|\hole &&
y\otimes y' \ar[rr] \ar[ddrr] &&
w \ar[urr] \ar@{.>}[rr] \ar[drr] &&
y\otimes z' \ar@{.>}[urr]|\hole \ar@{.>}[drr]|\hole &&
\Sigma(x\otimes y') \ar[rr] \ar[ddrr] &&
\Sigma v \ar[urr] \ar@{.>}[rr] \ar[drr] &&
\Sigma(x \otimes z') \ar@{}[r]|{\quad\cdots} &
\\
\ar@{}[r]|\cdots & 0 \ar@{.>}[urr]|\hole &&
y\otimes x' \ar[urr] \ar[rr] &&
z\otimes x' \ar@{-}+UR!(-1,0);+(6.4,3.6)\ar@{.>}+R!(2.2,4.4);+(17.7,9.5) \ar@{-}+R;+(9.5,0)\ar@{.>}+R!(6,0);+(17.7,0) &&
0 \ar@{.>}[urr]|\hole &&
\Sigma(x\otimes x') \ar[urr] \ar[rr] &&
\Sigma(y\otimes x') \ar@{.>}+R!(1.1,4.4);+(20.4,9.5) \ar@{-}+R;+(11,0)\ar@{.>}+R!(5,0);+(20.4,0) &&
0 \ar@{.>}[urr]|\hole &&
\Sigma(z \otimes x') \ar@{}[r]|{\quad\cdots} &
\\
\ar@{}[r]|\cdots & 0 \ar[urr] &&&&&& 0 \ar[urr] &&&&&& 0 \ar[urr]
}
\caption{The extended coherent version of May's axiom (TC3') obtained by $\tau\inv$ from (TC3).}
\label{fig:tc3'}
\end{figure}

\end{landscape}

\end{document}